\documentclass{amsart}
\usepackage{amsmath,amsthm,fixltx2e}
\usepackage{color}
\usepackage{amsxtra}
\usepackage{amscd}
\usepackage{amsfonts}
\usepackage{amssymb}
\usepackage{eucal}

\newtheorem{theorem}{Theorem}[section]
\newtheorem{cor}[theorem]{Corollary}
\newtheorem{lem}[theorem]{Lemma}
\newtheorem{prop}[theorem]{Proposition}

\newtheorem{thm}[theorem]{Theorem}

\newtheorem{rem}[theorem]{Remark}

\newtheorem{defn}[theorem]{Definition}

\newcommand{\nc}{\newcommand}

\nc\ol{\overline} \nc\ul{\underline} \nc\wt{\widetilde}
\nc{\z}{\zeta}

\nc{\ZZ}{{\mathbb Z}} \nc{\NN}{{\mathbb N}} \nc{\CC}{{\mathbb C}}
\nc{\QQ}{{\mathbb Q}} \nc{\CP}{{\mathbb {CP}}} \nc{\MM}{{\mathbb M}}

\nc{\F}{{\mathcal F}} \nc{\N}{{\mathcal N}} \nc{\Aa}{{\mathcal A}}
\nc{\E}{{\mathcal E}} \nc{\sS}{{\mathbb S}} \nc{\K}{{\mathcal K}}
\nc{\Ll}{{\mathcal L}} \nc{\Y}{{\mathcal Y}} \nc{\SSS}{{\mathcal S}}
\nc{\U}{{\mathcal U}} \nc\D{{\mathfrak D}}  \nc\dd{{\mathfrak d}}
\nc{\A}{{\mathcal A}} \nc\M{{\mathfrak M}} \nc{\TT}{{\mathbb T}}
\nc{\CCD}{{\mathcal D}}

\newcommand\q{\mathfrak q}

\newcommand{\gl}{\mathfrak{gl}}
\newcommand{\ssl}{\mathfrak{sl}}
\newcommand{\g}{\mathfrak{g}}
\newcommand{\Sym}{\mathrm{Sym}}

\newcommand{\tr}{\mathrm{tr}}

\newcommand\pp{\mathbf p}
\newcommand\uu{\mathbf u}
\newcommand\vv{\mathbf v}
\newcommand\ww{\mathbf w}
\newcommand\bol{\boldsymbol{\lambda}}
\newcommand\ch{\mathrm{ch}}
\newcommand\kk{\mathrm {k}}
\nc{\sW}{{\mathbb W}}

\nc{\iso}{{\stackrel{\sim}{\longrightarrow}}}

\begin{document}

\author{Mikhail Bershtein}
\address{M.B.: Landau Institute for Theoretical Physics, Chernogolovka, Russia
\newline
  Center for Advanced Studies, Skoltech, Moscow, Russia
\newline
  Institute for Information Transmission Problems,  Moscow, Russia
\newline
  National Research University Higher School of Economics, Moscow, Russia
\newline
 Independent University of Moscow, Moscow, Russia}
\email{mbersht@gmail.com}

\author[Alexander Tsymbaliuk]{Alexander Tsymbaliuk}
\address{A.T.: Yale University, Department of Mathematics, New Haven, CT 06511, USA}
\email{oleksandr.tsymbaliuk@yale.edu}

\title[Homomorphisms between quantum toroidal and affine Yangian algebras]
{Homomorphisms between different quantum toroidal and affine Yangian algebras}

\begin{abstract}
This paper concerns the relation between the quantum toroidal algebras and
the affine Yangians of $\ssl_n$, denoted by $\U^{(n)}_{q_1,q_2,q_3}$ and
$\Y^{(n)}_{h_1,h_2,h_3}$, respectively. Our motivation arises from the
milestone work~\cite{GTL}, where a similar relation between the quantum
loop algebra $U_q(L\g)$ and the Yangian $Y_h(\g)$ has been established
by constructing an isomorphism of $\CC[[\hbar]]$-algebras
  $\Phi\colon\widehat{U}_{\exp(\hbar)}(L\g)\iso \widehat{Y}_\hbar(\g)$
(with $\ \widehat{}\ $ standing for the appropriate completions).
These two completions model the behavior of the algebras in the formal neighborhood
of $h=0$. The same construction can be applied to the toroidal setting with
$q_i=\exp(\hbar_i)$ for $i=1,2,3$ (see~\cite{GTL,T}). In the current paper,
we are interested in the more general relation:
  $q_1=\omega_{mn}e^{h_1/m}, q_2=e^{h_2/m}, q_3=\omega_{mn}^{-1}e^{h_3/m}$,
where $m,n\geq 1$ and $\omega_{mn}$ is an $mn$-th root of $1$.
Assuming $\omega_{mn}^m$ is a primitive $n$-th root of unity, we construct
a homomorphism $\Phi^{\omega_{mn}}_{m,n}$ between the completions of the
formal versions of $\U^{(m)}_{q_1,q_2,q_3}$ and $\Y^{(mn)}_{h_1/mn,h_2/mn,h_3/mn}$.
\end{abstract}

\maketitle


\section*{Introduction}

Given a simple Lie algebra $\g$, one can associate to it two interesting Hopf algebras:
the quantum loop algebra $U_q(L\g)$ and the Yangian $Y_h(\g)$.
Their \emph{classical limits}, corresponding to the limits $q\to 1$ or $h\to 0$, recover
the universal enveloping algebras $U(\g[z,z^{-1}])$ and $U(\g[w])$, respectively.
The representation theories of $U_q(L\g)$ and $Y_h(\g)$ have a lot of common features:

-the descriptions of finite dimensional simple representations involve \emph{Drinfeld polynomials},

-these algebras act on the equivariant $K$-theories/cohomologies of Nakajima quiver varieties.

\noindent
However, there was no explicit justification for that until the recent construction
from~\cite{GTL} (also cf.~\cite[Section 5]{G}). In~\cite{GTL}, the authors construct
a $\CC[[\hbar]]$-algebra isomorphism
  $$\Phi\colon\widehat{U}_{e^{\hbar}}(L\g)\iso \widehat{Y}_\hbar(\g)$$
of the appropriately completed formal versions of these algebras.
Taking the limit $h\to 0$ corresponds to factoring by $(\hbar)$ in the formal setting.
The \emph{classical limit} of the above isomorphism is induced by
  $\underset{\longleftarrow}\lim\ \CC[z,z^{-1}]/(z-1)^r\iso
   \underset{\longleftarrow}\lim\ \CC[w]/(w)^r\simeq \CC[[w]]\
   \mathrm{with}\ z^{\pm 1}\mapsto e^{\pm w}.$

\medskip
In the current paper, we generalize this construction to the case of the quantum
toroidal algebras and the affine Yangians of $\ssl_n$ and $\gl_1$. To make our
notations uniform, we use $\U^{(n)}_{q_1,q_2,q_3}$ to denote the quantum toroidal
algebra of $\ssl_n$ (if $n\geq 2$) and of $\gl_1$ (if $n=1$). This algebra depends
on three nonzero parameters $q_1,q_2,q_3$ such that $q_1q_2q_3=1$. We also use
$\Y^{(n)}_{h_1,h_2,h_3}$ to denote the affine Yangian of $\ssl_n$ (if $n\geq 2$)
and of $\gl_1$ (if $n=1$). This algebra depends on three parameters $h_1,h_2,h_3$
such that $h_1+h_2+h_3=0$. For $n\geq 2$, these algebras were introduced long time
ago by~\cite{GKV,G}\footnote{\ Actually, we will need to modify slightly their construction
in the $n=2$ case.}. However, the quantum toroidal algebra and the affine Yangian of
$\gl_1$ appeared only recently in the works of different people, see~\cite{M,FT,SV1,MO,SV2,T}.

The main result of this paper, Theorem~\ref{main 1}, provides a homomorphism
  $$\Phi^{\omega_{mn}}_{m,n}\colon\widehat{\U}^{(m),\omega_{mn}}_{\hbar_1,\hbar_2}
    \longrightarrow \widehat{\Y}^{(mn)}_{\hbar_1,\hbar_2}$$
from the completion of the \emph{formal version} of $\U^{(m)}_{q_1,q_2,q_3}$ to
the completion of the \emph{formal version} of $\Y^{(mn)}_{h_1,h_2,h_3}$.
Formal versions mean that we consider these algebras over the ring $\CC[[\hbar_1,\hbar_2]]$ with
  $$h_1=\hbar_1/mn,\ h_2=\hbar_2/mn,\ h_3=\hbar_3/mn\ \ \mathrm{and}\ \
    q_1=\omega_{mn}e^{\frac{\hbar_1}{m}},\ q_2=e^{\frac{\hbar_2}{m}},\
    q_3=\omega_{mn}^{-1}e^{\frac{\hbar_3}{m}},$$
where $\hbar_3=-\hbar_1-\hbar_2$ and $\omega_N\in \CC^\times$ is an $N$-th root
of unity. For $n=1=\omega_{mn}$, we recover an analogue of the homomorphism $\Phi$
applied in the toroidal setting (see~\cite{T} for $m=n=\omega_{mn}=1$). In contrast
to~\cite{GTL,T}, our new feature is that we construct homomorphisms between formal
versions of quantum and Yangian algebras corresponding to different Lie algebras.
Another difference is that $q_1$ is in the formal neighborhood of a root of unity,
not necessarily equal to 1.

The structures of formulas for $\Phi^{\omega_{mn}}_{m,n}$ are similar to those in~\cite{GTL}.
Let $\{e_{i,k}, f_{i,k},h_{i,k}\}_{0\leq i\leq m-1}^{k\in \ZZ}$ be the generators of
$\U^{(m),\omega_{mn}}_{\hbar_1,\hbar_2}$ and
$\{x^\pm_{i',r},\xi_{i',r}\}_{0\leq i'\leq mn-1}^{r\in \NN}$ be the generators of
$\Y^{(mn)}_{\hbar_1,\hbar_2}$. Let $\Y^{(mn),0}_{h_1,h_2,h_3}\subset \Y^{(mn)}_{h_1,h_2,h_3}$
be the subalgebra generated by $\xi_{i',r}$. Then, we have:
\begin{equation*}
  \Phi^{\omega_{mn}}_{m,n}(h_{i,k})\in \widehat{\Y}^{(mn),0}_{\hbar_1,\hbar_2},\
  \Phi^{\omega_{mn}}_{m,n}(e_{i,k})=
  \sum_{i'\underset{m}\equiv i}^{0\leq i'<mn}\sum_{r=0}^\infty g^{(k)}_{i',r}x^+_{i',r},\
  \Phi^{\omega_{mn}}_{m,n}(f_{i,k})=
  \sum_{i'\underset{m}\equiv i}^{0\leq i'<mn}\sum_{r=0}^\infty g^{(k)}_{i',r}x^-_{i',r}
\end{equation*}
for certain $g^{(k)}_{i',r}\in \widehat{\Y}^{(mn),0}_{\hbar_1,\hbar_2}$-- the completion
of $\Y^{(mn),0}_{\hbar_1,\hbar_2}$ with respect to the natural $\NN$-grading. These
formulas as well as explicit formulas for $g^{(k)}_{i',r}$ were found following
the arguments of~\cite{GTL} as well as understanding the \emph{classical limit} first
(see Theorem~\ref{Upsilon_0} and Proposition~\ref{limit_Phi}). However, in contrast
to~\cite{GTL}, we are not aware of the direct proof of the compatibility of this assignment
with the \emph{Serre relations}. Instead, we propose two indirect proofs. In the first one,
we construct an isomorphism between faithful representations of the algebras in the question,
compatible with the defining formulas for $\Phi^{\omega_{mn}}_{m,n}$. In the second one,
we utilize the shuffle approach.

Our motivation partially comes from~\cite{BBT}, where a 4d AGT relation on the
ALE space $X_n$ (minimal resolution of $A_{n-1}$ singularity $\CC^2/\ZZ_n$) was studied.
The main tool in~\cite{BBT} was the limit of $K$-theoretic (5 dimensional) AGT relation
on $\CC^2$, where $q_1\rightarrow \omega_n, q_2 \rightarrow 1$. Recall that the quantum
toroidal algebra $\U^{(1)}_{q_1,q_2,q_3}$ acts on the equivariant $K$-theory of the
moduli spaces of torsion free sheaves on $\CC^2$, while the affine Yangian
$\Y^{(n)}_{h_1,h_2,h_3}$ acts on the equivariant cohomologies of the moduli spaces
of torsion free sheaves on $X_n$. Therefore, it was conjectured in~\cite{BBT} that
the limit of $\U^{(1)}_{q_1,q_2,q_3}$ as $q_1\rightarrow \omega_n, q_2\rightarrow 1$
should be related to the affine Yangian of $\ssl_n$. The $m=1$ case of our
Theorem~\ref{main 1} can be viewed as a precise formulation of this idea.
We also refer an interested reader to~\cite{K} for the related work.

This paper is organized as follows:

$\bullet$
In Section~\ref{section basic definitions}, we recall the definition of the quantum
toroidal algebra $\U^{(n)}_{q_1,q_2,q_3}$ and the affine Yangian $\Y^{(n)}_{h_1,h_2,h_3}$
of $\ssl_n$ (if $n\geq 2$) and $\gl_1$ (if $n=1$). They depend on $n\in \ZZ_{>0}$ and
continuous parameters $q_1,q_2,q_3\in \CC^\times$ or $h_1,h_2,h_3\in \CC$ satisfying
$q_1q_2q_3=1$ and $h_1+h_2+h_3=0$. We also explain the way one can view the algebras
$\Y^{(n)}_{h_1,h_2,h_3}$ as \emph{additivizations} of $\U^{(n)}_{q_1,q_2,q_3}$.

We recall a family of Fock  $\U^{(n)}_{q_1,q_2,q_3}$-representations
$F^{p}(u)\ (p\in \ZZ/n\ZZ, u\in \CC^\times)$ from~\cite{FJMM1} and introduce a similar class
of Fock $\Y^{(n)}_{h_1,h_2,h_3}$-representations $^a F^{p}(v)\ (p\in \ZZ/n\ZZ, v\in \CC)$.

$\bullet$
In Section~\ref{section formal versions}, we introduce the formal versions of these algebras
and study their \emph{classical limits}. Let $\Y^{(n)}_{\hbar_1,\hbar_2}$ be an associative
algebra over $\CC[[\hbar_1,\hbar_2]]$ with the same collections of the generators and the
defining relations as for $\Y^{(n)}_{h_1,h_2,-h_1-h_2}$ with $h_1\rightsquigarrow \hbar_1/n$
and $h_2\rightsquigarrow \hbar_2/n$.

One can similarly define the formal versions of  $\U^{(m)}_{q_1,q_2,q_3}$,
but this heavily depends on the presentation of $q_1,q_2,q_3\in \CC[[\hbar_1,\hbar_2]]$.
In this paper, we are interested in the behavior of the algebras
$\U^{(m)}_{q_1,q_2,q_3}$ and $\Y^{(n)}_{h_1,h_2,h_3}$ as
$q_1\to \omega_{N}, q_2\to 1, q_3\to \omega_{N}^{-1}$ and $h_1,h_2,h_3\to 0$, respectively.
Therefore, we will be mainly concerned with the following relation
between $\{h_s\}$ and $\{q_s\}$:
  $$q_1=\omega_{N}\cdot \exp(h_1/m),\ q_2=\exp(h_2/m),\ q_3=\omega_{N}^{-1}\exp(h_3/m).$$
The formal version of the corresponding $\U^{(m)}_{q_1,q_2,q_3}$
will be denoted by $\U^{(m),\omega_N}_{\hbar_1,\hbar_2}$.

Taking the limit $h_2\to 0$ corresponds to factoring by $(\hbar_2)$ in the formal setting.
According to~\cite{T2}, the \emph{classical limits}
  $\U^{(m),\omega_{N}}_{\hbar_1}=\U^{(m),\omega_{N}}_{\hbar_1,\hbar_2}/(\hbar_2)$
and
  $\Y^{(n)}_{\hbar_1}=\Y^{(n)}_{\hbar_1,\hbar_2}/(\hbar_2)$
are closely related to the matrix algebras with values in the rings of
\emph{difference} or \emph{differential} operators on $\CC^\times$, respectively.
In Theorem~\ref{flatness}, we show that the algebras $\Y^{(n)}_{\hbar_1,\hbar_2}$ and
$\U^{(m),\omega_{N}}_{\hbar_1,\hbar_2}$ are flat $\CC[[\hbar_2]]$-deformations
of the corresponding limit algebras $\Y^{(n)}_{\hbar_1}$ and $\U^{(m),\omega_{N}}_{\hbar_1}$.
We also prove that the direct sum of all finite tensor products of Fock
modules (which are not \emph{in resonance}) for either $\Y^{(n)}_{\hbar_1,\hbar_2}$ or
$\U^{(m),\omega_{N}}_{\hbar_1,\hbar_2}$ form a faithful representation of the corresponding algebra.

$\bullet$
In Section~\ref{section main result}, we present the main result of this paper.
We construct the homomorphism
  $$\Phi^{\omega_{mn}}_{m,n}\colon\widehat{\U}^{(m),\omega_{mn}}_{\hbar_1,\hbar_2}
    \longrightarrow \widehat{\Y}^{(mn)}_{\hbar_1,\hbar_2}$$
for any $m,n\geq 1$ and an $mn$-th root of unity $\omega_{mn}=\exp(2\pi \mathrm{k}\mathbf{i}/mn)$
with $\mathrm{k}\in \ZZ,\ \gcd(\mathrm{k},n)=1$. We compute the \emph{classical limit} of
$\Phi^{\omega_{mn}}_{m,n}$ using the above identification of $\U^{(m),\omega_{mn}}_{\hbar_1}$
and $\Y^{(mn)}_{\hbar_1}$ with matrix algebras over the rings of difference or differential
operators on $\CC^\times$, see Theorem~\ref{Upsilon_0}. In Section~\ref{section partial proof},
following~\cite{GTL}, we provide a straightforward verification of the compatibility of
$\Phi^{\omega_{mn}}_{m,n}$ with all the defining relations, except the most complicated
\emph{Serre relations}, for which the argument of~\cite{GTL} fails. We propose two
alternatives proofs in Sections~\ref{section proof via reps},~\ref{section proof via shuffle}.

$\bullet$
In Section~\ref{section proof via reps}, we construct isomorphisms between tensor products
of the Fock modules for $\U^{(m),\omega_{mn}}_{\hbar_1,\hbar_2}$ and $\Y^{(mn)}_{\hbar_1,\hbar_2}$,
which are compatible with the defining formulas for $\Phi^{\omega_{mn}}_{m,n}$.
Combining this result with the faithfulness statement from Section~\ref{section formal versions},
we obtain a proof of Theorem~\ref{main 1}. In Section~\ref{section geometric interpretation},
we recall the geometric realization of tensor products of the Fock modules for the quantum
toroidal algebra and the affine Yangian of $\ssl_n$, and provide a geometric interpretation
of the aforementioned isomorphism of tensor products of the Fock modules.

$\bullet$
In Section~\ref{section proof via shuffle}, we recall the shuffle realization of the
\emph{positive halves} $\U^{(m),\omega_{mn},>}_{\hbar_1,\hbar_2}$ and
$\Y^{(mn),\geq}_{\hbar_1,\hbar_2}$ due to~\cite{Neg0, Neg1, Neg2}, see
Theorems~\ref{Negut theorem 1},~\ref{Negut theorem 2}.
In Theorem~\ref{shuffle homom}, we construct the homomorphism between the completions
of the corresponding shuffle algebras and show that it is compatible with the restriction
of $\Phi^{\omega_{mn}}_{m,n}$. This implies the compatibility of the latter with
the Serre relations, and therefore completes our direct proof of Theorem~\ref{main 1}
initiated in Section~\ref{section partial proof}.

\subsection*{Acknowledgments}
\

The authors are grateful to B.~Feigin, M.~Finkelberg, A.~Negut, N.~Nekrasov for their
comments and encouragement, and to the anonymous referee for many useful suggestions.
The authors would also like to thank the organizers of 2015 PCMI research program on the
Geometry of Moduli Spaces and Representation Theory, where part of the project was performed.

A.T. thanks the Max Planck Institute for Mathematics in Bonn for support and great working
conditions, in which the project was started. A.T. also gratefully acknowledges support from
the Simons Center for Geometry and Physics, Stony Brook University,
at which most of the research for this paper was performed,
as well as Yale University, where the final version of this paper was completed.

The work of A.T. was partially supported by the NSF Grants DMS--1502497, DMS--1821185.
M.B. acknowledges the financial support from  Russian Academic Excellence Project 5-100,
Young Russian Math Contest, Simons-IUM fellowship and RFBR grant mol\_a\_ved 15-32-20974.


\section{Basic definitions and constructions}\label{section basic definitions}

In this section, we introduce the key actors of this paper:
the quantum toroidal algebra and the affine Yangian of $\ssl_n$.
We also recall the Fock representations of these algebras.


\subsection{Quantum toroidal algebras of $\ssl_n\ (n\geq 2)$ and $\gl_1$}\label{section toroidal}
$\ $

The quantum toroidal algebras of $\ssl_n\ (n>2)$, depending on $q,d\in \CC^\times$,
were first introduced in~\cite{GKV}. The quantum toroidal algebra of $\gl_1$ was
introduced much later in the works of different people, see~\cite{M, FT, SV1}.
Finally, a similar definition of the quantum toroidal algebra of $\ssl_2$ was proposed
in~\cite{FJMM2}. To make our exposition shorter, we use the uniform notation
$\U^{(n)}_{q_1,q_2,q_3}$ for such algebras, where $n\in \ZZ_{>0}$ and
$q_1=d/q,q_2=q^2,q_3=1/dq$, so that $q_1q_2q_3=1$. This algebra coincides with
the quotient of the algebra $\mathcal{E}_n$ from~\cite{FJMM2} by $q^c=1$. Since
the former was called the \emph{quantum toroidal algebra of $\gl_n$} in \emph{loc. cit.},
we will refer to $\U^{(n)}_{q_1,q_2,q_3}$ as the \emph{quantum toroidal algebra of $\ssl_n$}
(see the above explanation for the cases of $n=1,2$).

For $n\in \ZZ_{>0}$, we set $[n]:=\{0,1,\ldots,n-1\}$ which will be viewed as a set
of mod $n$ residues. Let $A=(a_{i,j})_{i\in [n]}^{j\in [n]}$ be the Cartan matrix of type
$A_{n-1}^{(1)}$ for $n\geq 2$ and a zero matrix for $n=1$. Consider two more matrices
$(d_{i,j})_{i\in [n]}^{j\in [n]}$ and $(m_{i,j})_{i\in [n]}^{j\in [n]}$ defined by
  $$d_{i,j}:=
    \begin{cases}
      d^{\mp 1} & \text{if}\ \ j=i\pm 1\ \mathrm{and}\ n>2, \\
      -1 & \text{if}\ \ j\ne i\ \mathrm{and}\ n=2, \\
      1 &  \text{otherwise},
    \end{cases}
    \ \ \
    m_{i,j}:=
    \begin{cases}
      1 & \text{if}\ \ j=i-1\ \mathrm{and}\ n>2, \\
      -1 & \text{if}\ \ j=i+1\ \mathrm{and}\ n>2, \\
      0 & \text{otherwise}.
    \end{cases}$$
Finally, we define a collection of polynomials $\{g_{i,j}(z,w)\}_{i\in [n]}^{j\in [n]}$ as follows:
  $$g_{i,j}(z,w):=
    \begin{cases}
      z-q^{a_{i,j}}d^{-m_{i,j}}w & \text{if}\ \ n>2,\\
      z-q_2w & \text{if}\ \ n=2\ \mathrm{and}\ i=j, \\
      (z-q_1w)(z-q_3w) & \text{if}\ \ n=2\ \mathrm{and}\ i\ne j,\\
      (z-q_1w)(z-q_2w)(z-q_3w) & \text{if}\ \ n=1.
    \end{cases}$$


The algebra $\U^{(n)}_{q_1,q_2,q_3}$ is the unital associative $\CC$-algebra
generated by $\{e_{i,k}, f_{i,k}, \psi_{i,k}, \psi_{i,0}^{-1}\}_{i\in [n]}^{k\in \ZZ}$
with the defining relations (T0--T6) to be given below:
\begin{equation}\tag{T0} \label{T0}
  \psi_{i,0}\cdot \psi_{i,0}^{-1}=\psi_{i,0}^{-1}\cdot \psi_{i,0}=1,\
  [\psi_i^\pm(z),\psi_j^\pm(w)]=0,\ [\psi_i^+(z),\psi_j^-(w)]=0,
\end{equation}
\begin{equation}\tag{T1}\label{T1}
  [e_i(z),f_j(w)]=\frac{\delta_{i,j}}{q-q^{-1}}\cdot \delta(w/z)(\psi_i^+(w)-\psi_i^-(z)),
\end{equation}
\begin{equation}\tag{T2} \label{T2}
  d_{i,j}g_{i,j}(z,w)e_i(z)e_j(w)=-g_{j,i}(w,z)e_j(w)e_i(z),
\end{equation}
\begin{equation}\tag{T3}\label{T3}
  d_{j,i}g_{j,i}(w,z)f_i(z)f_j(w)=-g_{i,j}(z,w)f_j(w)f_i(z),
\end{equation}
\begin{equation}\tag{T4}\label{T4}
  d_{i,j}g_{i,j}(z,w)\psi_i^\pm(z)e_j(w)=-g_{j,i}(w,z)e_j(w)\psi_i^\pm(z),
\end{equation}
\begin{equation}\tag{T5}\label{T5}
  d_{j,i}g_{j,i}(w,z)\psi_i^{\pm}(z)f_j(w)=-g_{i,j}(z,w)f_j(w)\psi_i^\pm(z),
\end{equation}
where these generating series are defined as follows:
  $$e_i(z):=\sum_{k=-\infty}^{\infty}{e_{i,k}z^{-k}},\  \
    f_i(z):=\sum_{k=-\infty}^{\infty}{f_{i,k}z^{-k}},\  \
    \psi_i^{\pm}(z):=\psi_{i,0}^{\pm 1}+\sum_{r>0}{\psi_{i,\pm r}z^{\mp r}},\ \
    \delta(z):=\sum_{k=-\infty}^{\infty}{z^k}.$$

Let us now specify the \emph{Serre relations}~(T6) in each of the cases: $n>2,n=2, n=1$.
Set $[a,b]_x:=ab-x\cdot ba$, while $\underset{z_1,\ldots,z_r}\Sym$ will stand for the
symmetrization in $z_1,\ldots,z_r$.

\medskip
\noindent
$\bullet$ \emph{Case $n>2$}.
Then, we impose:
\begin{equation}\tag{T6}\label{T6.1}
\begin{split}
  & [e_i(z),e_j(w)]=0,\ [f_i(z),f_j(w)]=0\ \ \mathrm{if}\ a_{i,j}=0,\\
  & \underset{z_1,z_2}\Sym\ [e_i(z_1), [e_i(z_2),e_{i\pm 1}(w)]_q]_{q^{-1}}=0,\
  \underset{z_1,z_2}\Sym\ [f_i(z_1), [f_i(z_2), f_{i\pm 1}(w)]_q]_{q^{-1}}=0.
\end{split}
\end{equation}

\noindent
$\bullet$ \emph{Case $n=2$}.
Then, we impose
\begin{equation}\tag{T6}\label{T6.2}
\begin{split}
  & \underset{z_1,z_2,z_3}\Sym\ [e_i(z_1), [e_i(z_2),[e_i(z_3),e_{i+1}(w)]_{q^2}]]_{q^{-2}}=0,\\
  & \underset{z_1,z_2,z_3}\Sym\ [f_i(z_1), [f_i(z_2),[f_i(z_3),f_{i+1}(w)]_{q^2}]]_{q^{-2}}=0.
\end{split}
\end{equation}

\noindent
$\bullet$ \emph{Case $n=1$}.
Then, we impose
\begin{equation}\tag{T6}\label{T6.3}
   \underset{z_1,z_2,z_3}\Sym\ \frac{z_2}{z_3} [e_0(z_1),[e_0(z_2),e_0(z_3)]]=0,\
   \underset{z_1,z_2,z_3}\Sym\ \frac{z_2}{z_3} [f_0(z_1),[f_0(z_2),f_0(z_3)]]=0.
\end{equation}

\begin{rem}\label{isomorphism of different toroidal}
For any $n>1$ and $\omega_n=\sqrt[n]{1}\in \CC^\times$, there exists an algebra isomorphism
$\U^{(n)}_{q_1,q_2,q_3}\iso \U^{(n)}_{\omega_n\cdot q_1,q_2,\omega_n^{-1}\cdot q_3}$ given by
  $e_i(z)\mapsto e_i(\omega^{-i}_n z),
   f_i(z)\mapsto f_i(\omega^{-i}_n z),
   \psi^\pm_i(z)\mapsto \psi^\pm_i(\omega^{-i}_n z)$.
\end{rem}

It will be convenient to use the generators $\{h_{i,k}\}_{i\in [n]}^{k\in \ZZ\backslash\{0\}}$
instead of $\{\psi_{i,k}\}_{i\in [n]}^{k\in \ZZ\backslash\{0\}}$, defined by
  $$\exp\left(\pm(q-q^{-1})\sum_{r>0}h_{i,\pm r}z^{\mp r}\right)=\bar{\psi}_i^\pm(z):=
    \psi_{i,0}^{\mp 1}\psi^\pm_i(z),\ \
    h_{i,\pm r}\in \CC[\psi_{i,0}^{\mp 1},\psi_{i,\pm 1},\psi_{i,\pm2}, \ldots].$$
Then, the relations (T4,T5) are equivalent to the following
(we use notation $[m]_q:=\frac{q^m-q^{-m}}{q-q^{-1}}$):
\begin{equation}\tag{T4$'$}\label{T4'}
  \psi_{i,0}e_{j,l}=q^{a_{i,j}}e_{j,l}\psi_{i,0},\
  [h_{i, k}, e_{j,l}]=b_n(i,j;k)\cdot e_{j,l+k}\ \mathrm{for}\ k\ne 0,
\end{equation}
\begin{equation}\tag{T5$'$}\label{T5'}
  \psi_{i,0}f_{j,l}=q^{-a_{i,j}}f_{j,l}\psi_{i,0},\
  [h_{i, k},  f_{j,l}]=-b_n(i,j;k)\cdot f_{j,l+k}\ \mathrm{for}\ k\ne 0,
\end{equation}
where the constants $b_n(i,j;k)$ are given explicitly by:
\begin{equation}\label{sharp}
  b_n(i,j;k)=
  \begin{cases}
      \frac{[ka_{i,j}]_q}{k}\cdot d^{-km_{i,j}} & \text{if}\ \ n>2,\\
      \frac{[2k]_q}{k}\delta_{j,i} -\frac{[k]_q}{k}\cdot (d^k+d^{-k})\delta_{j,i+1} & \text{if}\ \ n=2,\\
      \frac{[k]_q}{k}\cdot (q^k+q^{-k}-d^k-d^{-k}) & \text{if} \ \ n=1.
    \end{cases}
\end{equation}


We equip the algebra $\U^{(n)}_{q_1,q_2,q_3}$ with the \emph{principal} $\ZZ$-grading by assigning
  $$\deg(e_{i,k})=1,\ \deg(f_{i,k})=-1,\ \deg(\psi_{i,k})=0
    \ \mathrm{for\ all}\ i\in [n], k\in \ZZ.$$
Following~\cite[Theorem 2.1]{DI}, we endow $\U^{(n)}_{q_1,q_2,q_3}$ with
a formal coproduct by assigning
\begin{equation}\label{coproduct 1}
\begin{split}
  & \Delta(e_i(z))=e_i(z)\otimes 1 + \psi^{-}_i(z)\otimes e_i(z),\
  \Delta(f_i(z))=f_i(z)\otimes \psi^{+}_i(z) +  1\otimes f_i(z),\\
  & \Delta(\psi_i^{\pm}(z))=\psi^\pm_i(z)\otimes \psi^\pm_i(z).
\end{split}
\end{equation}


\subsection{Affine Yangians of $\ssl_n\ (n\geq 2)$ and $\gl_1$}
$\ $

The affine Yangians of $\ssl_n\ (n>2)$, denoted by
  $\widehat{\textbf{Y}}_{\bar{\lambda},\bar{\beta}}\ (\bar{\lambda},\bar{\beta}\in \CC)$,
were first introduced in~\cite{G}. Their counterpart for $n=2$ is introduced below.
Finally, the affine Yangian of $\gl_1$ has recently appeared in the works of
Maulik-Okounkov~\cite{MO} and Schiffmann-Vasserot~\cite{SV2}. In the present paper,
we will need the \emph{loop presentation} of the latter algebra from~\cite{T}.

To make our exposition shorter, we call such algebras the \emph{affine Yangians of $\ssl_n$}
and use the uniform notation $\Y^{(n)}_{h_1,h_2,h_3}$ for them, where $n\in \ZZ_{>0}$ and
  $h_1=\beta-h,h_2=2h,h_3=-\beta-h\ \mathrm{with}\ \beta,h\in \CC$, so that $h_1+h_2+h_3=0$.
The algebra $\Y^{(n)}_{h_1,h_2,h_3}$ is the unital associative $\CC$-algebra generated by
$\{x^\pm_{i,r}, \xi_{i,r}\}_{i\in [n]}^{r\in \NN}$ with the defining relations (Y0--Y5)
to be given below.

The first two relations are independent of $n\in \ZZ_{>0}$:
\begin{equation}\tag{Y0}\label{Y0}
  [\xi_{i,r},\xi_{j,s}]=0,
\end{equation}
\begin{equation}\tag{Y1}\label{Y1}
  [x^+_{i,r},x^-_{j,s}]=\delta_{i,j}\cdot \xi_{i,r+s}.
\end{equation}

Let us now specify (Y2--Y5) in each of the cases: $n>2,n=2, n=1$. Set $\{a,b\}:=ab+ba$.

\medskip
\noindent
$\bullet$ \emph {Case $n>2$}. Then, we impose
\begin{equation}\tag{Y2}\label{Y2.1}
  [x^\pm_{i,r+1},x^\pm_{j,s}]-[x^\pm_{i,r},x^\pm_{j,s+1}]=
  -m_{i,j}\beta[x^\pm_{i,r},x^\pm_{j,s}]\pm a_{i,j}h \{x^\pm_{i,r},x^\pm_{j,s}\},
\end{equation}
\begin{equation}\tag{Y3}\label{Y3.1}
  [\xi_{i,r+1},x^\pm_{j,s}]-[\xi_{i,r},x^\pm_{j,s+1}]=
  -m_{i,j}\beta[\xi_{i,r},x^\pm_{j,s}]\pm a_{i,j}h \{\xi_{i,r},x^\pm_{j,s}\},
\end{equation}
\begin{equation}\tag{Y4}\label{Y4.1}
  [\xi_{i,0},x^\pm_{j,s}]=\pm a_{i,j} x^\pm_{j,s},
\end{equation}
\begin{equation}\tag{Y5}\label{Y5.1}
  \underset{r_1,r_2}\Sym\ [x^\pm_{i,r_1},[x^\pm_{i,r_2},x^\pm_{i\pm 1,s}]]=0\ \
  \mathrm{and}\ \ [x^\pm_{i,r},x^\pm_{j,s}]=0 \ \mathrm{if}\ a_{i,j}=0.
\end{equation}

\noindent
$\bullet$ \emph {Case $n=2$}. Then, we impose
\begin{equation}\tag{Y2.1}\label{Y2.1.2}
  [x^\pm_{i,r+1},x^\pm_{i,s}]-[x^\pm_{i,r},x^\pm_{i,s+1}]=\pm h_2\{x^\pm_{i,r},x^\pm_{i,s}\},
\end{equation}
\begin{equation}\tag{Y2.2}\label{Y2.2.2}
\begin{split}
  & [x^\pm_{i,r+2},x^\pm_{j,s}]-2[x^\pm_{i,r+1},x^\pm_{j,s+1}]+[x^\pm_{i,r},x^\pm_{j,s+2}]=\\
  & -h_1h_3[x^\pm_{i,r},x^\pm_{j,s}]\mp h_2(\{x^\pm_{i,r+1},x^\pm_{j,s}\}-\{x^\pm_{i,r},x^\pm_{j,s+1}\})\ \mathrm{for}\ j\ne i,
\end{split}
\end{equation}
\begin{equation}\tag{Y3.1}\label{Y3.1.2}
  [\xi_{i,r+1},x^\pm_{i,s}]-[\xi^\pm_{i,r},x^\pm_{i,s+1}]=\pm h_2\{\xi_{i,r},x^\pm_{i,s}\},
\end{equation}
\begin{equation}\tag{Y3.2}\label{Y3.2.2}
\begin{split}
  & [\xi_{i,r+2},x^\pm_{j,s}]-2[\xi_{i,r+1},x^\pm_{j,s+1}]+[\xi_{i,r},x^\pm_{j,s+2}]=\\
  & -h_1h_3[\xi_{i,r},x^\pm_{j,s}]\mp h_2(\{\xi_{i,r+1},x^\pm_{j,s}\}-\{\xi_{i,r},x^\pm_{j,s+1}\})\ \mathrm{for}\ j\ne i,
\end{split}
\end{equation}
\begin{equation}\tag{Y4}\label{Y4.2}
  [\xi_{i,0},x^\pm_{j,s}]=\pm a_{i,j} x^\pm_{j,s},\ \
  [\xi_{i,1},x^\pm_{i+1,s}]=\mp (2x^\pm_{i+1,s+1}+h_2\{\xi_{i,0},x^\pm_{i+1,s}\}),
\end{equation}
\begin{equation}\tag{Y5}\label{Y5.2}
  \underset{r_1,r_2,r_3}\Sym [x^\pm_{i,r_1},[x^\pm_{i,r_2},[x^\pm_{i,r_3},x^\pm_{i+1,s}]]]=0.
\end{equation}

\noindent
$\bullet$ \emph {Case $n=1$}. Then, we impose
\begin{equation}\tag{Y2}\label{Y2.3}
\begin{split}
  & [x^\pm_{0,r+3},x^\pm_{0,s}]-3[x^\pm_{0,r+2},x^\pm_{0,s+1}]+3[x^\pm_{0,r+1},x^\pm_{0,s+2}]-[x^\pm_{0,r},x^\pm_{0,s+3}]=\\
  & -\sigma_2([x^\pm_{0,r+1},x^\pm_{0,s}]-[x^\pm_{0,r},x^\pm_{0,s+1}]) \pm \sigma_3\{x^\pm_{0,r},x^\pm_{0,s}\},
\end{split}
\end{equation}
\begin{equation}\tag{Y3}\label{Y3.3}
\begin{split}
  & [\xi_{0,r+3},x^\pm_{0,s}]-3[\xi_{0,r+2},x^\pm_{0,s+1}]+3[\xi_{0,r+1},x^\pm_{0,s+2}]-[\xi_{0,r},x^\pm_{0,s+3}]=\\
  & -\sigma_2([\xi_{0,r+1},x^\pm_{0,s}]-[\xi_{0,r},x^\pm_{0,s+1}])\pm \sigma_3\{\xi_{0,r},x^\pm_{0,s}\},
\end{split}
\end{equation}
\begin{equation}\tag{Y4}\label{Y4.3}
  [\xi_{0,0},x^\pm_{0,s}]=0,\ \ [\xi_{0,1},x^\pm_{0,s}]=0,\ \ [\xi_{0,2},x^\pm_{0,s}]=\pm 2h_1h_3x^\pm_{0,s},
\end{equation}
\begin{equation}\tag{Y5}\label{Y5.3}
  \underset{r_1,r_2,r_3}\Sym [x^\pm_{0,r_1},[x^\pm_{0,r_2},x^\pm_{0,r_3+1}]]=0,
\end{equation}
where we set $\sigma_2:=h_1h_2+h_1h_3+h_2h_3,\ \sigma_3:=h_1h_2h_3$.

\begin{rem}
(a)  For $n>2$, the algebras $\Y^{(n)}_{h_1,h_2,h_3}$ coincide with those of~\cite{G}.
Explicitly, we have an isomorphism
 $\Y^{(n)}_{h_1,h_2,-h_1-h_2}\simeq \widehat{\bf{Y}}_{h_2,\frac{1}{2}h_2-\frac{n}{4}(2h_1+h_2)}$.

\noindent
(b) Our definition of $\Y^{(2)}_{h_1,h_2,h_3}$ coincides with the corrected version of
$Y_{-h_3,-h_1}(\widehat{\ssl}_2)$ from~\cite{K}.

\noindent
(c) Our definition of $\Y^{(1)}_{h_1,h_2,h_3}$ first appeared in~\cite{T} under the name
``the affine Yangian of $\gl_1$''.
\end{rem}


\subsection{Affine Yangians as \emph{additivizations} of quantum toroidal algebras}\label{section Yangian addit}
$\ $

The algebras $\Y^{(n)}_{h_1,h_2,h_3}$ can be considered as natural \emph{additivizations}
of the algebras $\U^{(n)}_{q_1,q_2,q_3}$ in the same way as $Y_h(\g)$ is an \emph{additivization}
of $U_q(L\g)$. We explain this by rewriting (Y0--Y5) in a form similar to the defining
relations (T0--T6). We also define an algebra $\CCD\Y^{(n)}_{h_1,h_2,h_3}$.

Let us introduce the generating series:
  $$x^\pm_i(z):=\sum_{r\geq 0}{x^\pm_{i,r}z^{-r-1}},\
    \xi_i(z):=1+h_2\sum_{r\geq 0}{\xi_{i,r}z^{-r-1}}.$$
We also define a collection of polynomials $\{p_{i,j}(z,w)\}_{i\in [n]}^{j\in [n]}$ as follows:
  $$p_{i,j}(z,w):=
    \begin{cases}
      z-w+m_{i,j}\beta-a_{i,j}h & \text{if}\ \ n>2,\\
      z-w-h_2 & \text{if}\ \ n=2\ \mathrm{and}\ i=j,\\
      (-1)^{\delta_{j,1}}(z-w-h_1)(z-w-h_3) & \text{if}\ \ n=2\ \mathrm{and}\ i\ne j,\\
      (z-w-h_1)(z-w-h_2)(z-w-h_3) & \text{if}\ \ n=1.
    \end{cases}$$


Let $\Y^{(n),<},\Y^{(n),0},\Y^{(n),>}$ be the subalgebras of $\Y^{(n)}_{h_1,h_2,h_3}$ generated by
$\{x^-_{i,r}\}_{i\in [n]}^{r\in \NN}, \{\xi_{i,r}\}_{i\in [n]}^{r\in \NN}$, and
$\{x^+_{i,r}\}_{i\in [n]}^{r\in \NN}$, respectively. Let $\Y^{(n),\geq}$ and $\Y^{(n),\leq}$ be
the subalgebras of $\Y^{(n)}_{h_1,h_2,h_3}$ generated by $\Y^{(n),0},\Y^{(n),>}$ and
$\Y^{(n),0},\Y^{(n),<}$, respectively. The following result is standard:

\begin{prop}\label{triangular}
(a) $\Y^{(n),0}$ is isomorphic to a polynomial algebra
in the generators $\{\xi_{i,r}\}_{i\in [n]}^{r\in \NN}$.

\noindent
(b) $\Y^{(n),\gtrless}$ are isomorphic to the algebras generated by
$\{x^\pm_{i,r}\}_{i\in [n]}^{r\in \NN}$ subject to~(Y2,Y5).

\noindent
(c) $\Y^{(n),\gtreqqless}$ are isomorphic to the algebras generated by
$\{\xi_{i,r},x^\pm_{i,r}\}_{i\in [n]}^{r\in \NN}$ subject to~(Y0,Y2--Y5).
\end{prop}

Consider the homomorphisms $\sigma_i^\pm\colon \Y^{(n),\gtreqqless}\to \Y^{(n),\gtreqqless}$
defined by $\xi_{j,r}\mapsto \xi_{j,r}, x^\pm_{j,r}\mapsto x^\pm_{j,r+\delta_{i,j}}$.
These are well-defined due to Proposition~\ref{triangular}(c). Let
  $\mu\colon \Y^{(n)}_{h_1,h_2,h_3}\otimes \Y^{(n)}_{h_1,h_2,h_3}\to \Y^{(n)}_{h_1,h_2,h_3}$
be the multiplication map. The following is straightforward:

\begin{prop}\label{yangian_generating}
(a) The relation (Y0) is equivalent to
 $[\xi_i(z),\xi_j(w)]=0.$

\noindent
(b) The relation (Y1) is equivalent to
$h_2\cdot(w-z)[x^+_i(z),x^-_j(w)]=\delta_{i,j}(\xi_i(z)-\xi_i(w)).$

\noindent
(c) The relations (Y3,Y4) are equivalent to
  $$p_{i,j}(z,\sigma^+_j)\xi_i(z)x^+_{j,s}=-p_{j,i}(\sigma^+_j,z)x^+_{j,s}\xi_i(z),\
    p_{j,i}(\sigma^-_j,z)\xi_i(z)x^-_{j,s}=-p_{i,j}(z,\sigma^-_j)x^-_{j,s}\xi_i(z).$$

\noindent
(d) The relation (Y2) is equivalent to
  $$\partial_z^{\deg_{i,j;n}} \mu\left(p_{i,j}(z,\sigma^{+,(2)}_j)x^+_i(z)\otimes x^+_{j,s} +
    p_{j,i}(\sigma^{+,(1)}_j,z)x^+_{j,s}\otimes x^+_i(z)\right)=0,$$
  $$\partial_z^{\deg_{i,j;n}} \mu\left(p_{j,i}(\sigma^{-,(2)}_j,z)x^-_i(z)\otimes x^-_{j,s} +
    p_{i,j}(z,\sigma^{-,(1)}_j)x^-_{j,s}\otimes x^-_i(z)\right)=0,$$
where we set
 $\deg_{i,j;n}:=\deg(p_{i,j}(z,w)),
  \sigma_j^{\pm, (1)}(a\otimes b):=\sigma_j^\pm(a)\otimes b,
  \sigma_j^{\pm, (2)}(a\otimes b):=a\otimes \sigma_j^\pm(b)$.
\end{prop}

\begin{rem}\label{Double}
Let $\CCD\Y^{(n)}_{h_1,h_2,h_3}$ be the unital associative $\CC$-algebra generated by
$\{x^\pm_{i,k}, \xi_{i,k}\}_{i\in [n]}^{k\in \ZZ}$ with the defining relations (Y0--Y5).
A similar construction for $Y_h(\g)$ was first introduced in~\cite{D} (see also~\cite{KT}).
We equip $\CCD\Y^{(n)}_{h_1,h_2,h_3}$ with a formal coproduct by assigning
\begin{equation}\label{coproduct 2}
\begin{split}
  & \Delta(\wt{x}^+_i(z))=\wt{x}^+_i(z)\otimes 1 + \wt{\xi}^{-}_i(z)\otimes \wt{x}^+_i(z),\
  \Delta(\wt{x}^-_i(z))=\wt{x}^-_i(z)\otimes \wt{\xi}^{+}_i(z) +  1\otimes \wt{x}^-_i(z),\\
  & \Delta(\wt{\xi}^\pm_i(z))=\wt{\xi}^\pm_i(z)\otimes \wt{\xi}^\pm_i(z),
\end{split}
\end{equation}
where
  $\wt{x}^\pm_i(z):=\sum_{k\in \ZZ} x^\pm_{i,k}z^{-k-1},\
   \wt{\xi}^+_i(z):=1+\sum_{r\geq 0}\xi_{i,r}z^{-r-1},\
   \wt{\xi}^-_i(z):=1-\sum_{s<0}\xi_{i,s}z^{-s-1}$.
\end{rem}


\subsection{Fock representations}\label{section Fock}
$\ $

For $p\in [n]$ and $u\in \CC^\times$, let $F^p(u)$ be a $\CC$-vector space
with the basis $\{|\lambda\rangle\}$ labeled by all partitions $\lambda$.
Given such $\lambda=(\lambda_1,\lambda_2,\ldots)$ and $s\in \ZZ_{>0}$, let
$\lambda\pm 1_s:=(\lambda_1,\ldots,\lambda_s\pm 1,\ldots)$
and define $c_{s}(\lambda):=p+s-\lambda_s\in \ZZ$. We also write $a\equiv b$ if
$a-b$ is divisible by $n$ and set $\bar{\delta}_{a,b}:=\delta_{a\equiv b}$. We write
$g(z)^\pm$ for the expansion of a rational function $g(z)$ in $z^{\mp 1}$, respectively.
Set $\psi(z):=(q-q^{-1}z)/(1-z)$. The following result is due to~\cite[Section 2.5]{FJMM2}.

\begin{prop}\label{Fock_m}
(a) For $n>1$, the following formulas define an action of
$\U^{(n)}_{q_1,q_2,q_3}\ \mathrm{on}\ F^p(u):$
  $$\langle \lambda+1_l|e_j(z)|\lambda\rangle=
    \bar{\delta}_{c_l(\lambda),j+1}
    \prod_{1\leq s<l}^{c_s(\lambda)\equiv j} \psi(q_1^{\lambda_s-\lambda_l-1}q_3^{s-l})
    \prod_{1\leq s<l}^{c_s(\lambda)\equiv j+1} \psi(q_1^{\lambda_l-\lambda_s}q_3^{l-s})
    \cdot \delta(q_1^{\lambda_l}q_3^{l-1}u/z),$$
  $$\langle \lambda|f_j(z)|\lambda+1_l\rangle=
    \bar{\delta}_{c_l(\lambda),j+1}
    \prod_{s>l}^{c_s(\lambda)\equiv j} \psi(q_1^{\lambda_s-\lambda_l-1}q_3^{s-l})
    \prod_{s>l}^{c_s(\lambda)\equiv j+1} \psi(q_1^{\lambda_l-\lambda_s}q_3^{l-s})
    \cdot \delta(q_1^{\lambda_l}q_3^{l-1}u/z),$$
  $$\langle \lambda|\psi^\pm_j(z)|\lambda\rangle=
    \prod_{s\geq 1}^{c_s(\lambda)\equiv j} \psi(q_1^{\lambda_s-1}q_3^{s-1}u/z)^\pm
    \prod_{s\geq 1}^{c_s(\lambda)\equiv j+1} \psi(z/q_1^{\lambda_s}q_3^{s-1}u)^\pm,$$
while all other matrix coefficients are set to be zero.

\noindent
(b) For $n=1$, the same formulas with the matrix coefficient of $f_0(z)$ multiplied by
$\frac{q(1-q_3)}{1-q_1^{-1}}$ define an action of $\U^{(1)}_{q_1,q_2,q_3}$ on $F^0(u)$.
\end{prop}

\begin{rem}
(a) If $\lambda+1_l$ (resp. $\lambda$) is not a partition, while $\lambda$
(resp. $\lambda+1_l$) is a partition, then the right-hand side of the first
(resp. second) formula is zero, hence, the equality is vacuous.

\noindent
(b) The above infinite products can be simplified to finite products,
due to $\psi(1/z)\psi(q_2z)=1$.

\noindent
(c) The Fock representations $F^p(u)$ were originally constructed from the
``vector representations'' by using the semi-infinite wedge construction and
the formal coproduct~(\ref{coproduct 1}) on $\U^{(n)}_{q_1,q_2,q_3}$.
\end{rem}

Let us define analogous \emph{Fock representations} of $\Y^{(n)}_{h_1,h_2,h_3}$.
For $p\in [n]$ and $v\in \CC$, let $^{a}F^p(v)$ be a $\CC$-vector space with
the basis $\{|\lambda\rangle\}$. We also set $\phi(z):=\frac{z-h_2}{z}$ and
$\delta^+(z):=\sum_{r=0}^\infty z^r$.

\begin{prop}\label{Fock_a}
(a) For $n>1$, the following formulas define an action of
$\Y^{(n)}_{h_1,h_2,h_3}\ \mathrm{on}\ {^{a}F}^p(v):$
  $$\langle \lambda+1_l|x^+_j(z)|\lambda\rangle=
    \frac{\bar{\delta}_{c_l(\lambda),j+1}}{z}
    \prod_{1\leq s<l}^{c_s(\lambda)\equiv j} \phi((\lambda_s-\lambda_l-1)h_1+(s-l)h_3)
    \prod_{1\leq s<l}^{c_s(\lambda)\equiv j+1} \phi((\lambda_l-\lambda_s)h_1+(l-s)h_3)$$
  $$\times\delta^+((\lambda_lh_1+(l-1)h_3+v)/z),$$
  $$\langle \lambda|x^-_j(z)|\lambda+1_l\rangle=
    \frac{\bar{\delta}_{c_l(\lambda),j+1}}{z}
    \prod_{s>l}^{c_s(\lambda)\equiv j} \phi((\lambda_s-\lambda_l-1)h_1+(s-l)h_3)
    \prod_{s>l}^{c_s(\lambda)\equiv j+1} \phi((\lambda_l-\lambda_s)h_1+(l-s)h_3)$$
  $$\times\delta^+((\lambda_lh_1+(l-1)h_3+v)/z),$$
  $$\langle \lambda|\xi_j(z)|\lambda\rangle=
    \prod_{s\geq 1}^{c_s(\lambda)\equiv j} \phi((\lambda_s-1)h_1+(s-1)h_3+v-z)^+
    \prod_{s\geq 1}^{c_s(\lambda)\equiv j+1} \phi(z-(\lambda_s h_1+(s-1)h_3+v))^+,$$
while all other matrix coefficients are set to be zero.

\noindent
(b) For $n=1$, the same formulas with the matrix coefficient of $x^-_0(z)$ multiplied by $-h_3/h_1$
define an action of $\Y^{(1)}_{h_1,h_2,h_3}$ on ${^{a}F}^0(v)$, cf.~\cite[Proposition 4.4]{T}.
\end{prop}

\begin{rem}\label{Fock double}
For $v\notin\{-ah_1-bh_3|a,b\in \NN\}$, we get an action of $\CCD\Y^{(n)}_{h_1,h_2,h_3}$
(from Remark~\ref{Double}) on ${^{a}F}^p(v)$ by changing
$\delta^+(\cdots)\rightsquigarrow \delta(\cdots)$ and
$\phi(\cdots)^+\rightsquigarrow \phi(\cdots)^\pm$ in the above formulas.
\end{rem}


\subsection{Tensor products of Fock representations}\label{section tensor Fock}
$\ $

In addition to the Fock modules, we will also need their tensor products. Given $r\in \ZZ_{>0}$
and $\pp=(p_1,\ldots,p_r) \in [n]^r,\uu=(u_1,\ldots,u_r)\in(\CC^\times)^r$, consider the
Fock modules $\{F^{p_k}(u_k)\}_{k=1}^r$. Using the formal coproduct~(\ref{coproduct 1})
on the algebra $\U^{(n)}_{q_1,q_2,q_3}$, one can define an action of
  $\U^{(n)}_{q_1,q_2,q_3}$ on $F^{\pp}(\uu):=F^{p_1}(u_1)\otimes \cdots \otimes F^{p_r}(u_r)$,
but only if $\{u_k\}_{k=1}^r$ are \emph{not in resonance}, see~\cite{FJMM1}.
This module has the basis $\{|\boldsymbol{\lambda}\rangle\}$ labeled by $r$-tuples
of partitions $\boldsymbol{\lambda}=(\lambda^{(1)},\ldots,\lambda^{(r)})$.
Define $c^{(a)}_s(\bol):=c_s(\lambda^{(a)})$ and let $\bol+1_s^{(a)}$ denote
$(\lambda^{(1)},\ldots,\lambda^{(a)}+1_s,\ldots,\lambda^{(r)})$. For $1\leq a,b\leq r$ and $s,l\in \ZZ_{>0}$,
we say $(a,s)\prec (b,l)$ if either $a<b$ or $a=b,s<l$. We also set
$\chi^{(a)}_s:=q_1^{\lambda^{(a)}_s}q_3^{s-1}u_a$.

\begin{prop}\label{Fock_mr}
(a) For $n>1$, the following formulas define an action of
$\U^{(n)}_{q_1,q_2,q_3}\ \mathrm{on}\ F^{\pp}(\uu)$
\begin{equation*}
\begin{aligned}
  \langle\boldsymbol{\lambda}+1_l^{(b)}| e_j(z)|\boldsymbol{\lambda}\rangle&=
  \bar{\delta}_{c^{(b)}_l(\bol),j+1}
  \prod_{(a,s)\prec (b,l)}^{c^{(a)}_s(\bol)\equiv j} \psi\left(\chi^{(a)}_s/q_1\chi^{(b)}_l\right)
  \prod_{(a,s)\prec (b,l)}^{c^{(a)}_s(\bol)\equiv j+1} \psi\left(\chi^{(b)}_l/\chi^{(a)}_s\right)
  \cdot \delta(\chi^{(b)}_l/z),
  \\
  \langle \boldsymbol{\lambda}|f_j(z)|\boldsymbol{\lambda}+1_l^{(b)}\rangle&=
  \bar{\delta}_{c^{(b)}_l(\bol),j+1}
  \prod_{(a,s)\succ (b,l)}^{c^{(a)}_s(\bol)\equiv j} \psi\left(\chi^{(a)}_s/q_1\chi^{(b)}_l\right)
  \prod_{(a,s)\succ (b,l)}^{c^{(a)}_s(\bol)\equiv j+1} \psi\left(\chi^{(b)}_l/\chi^{(a)}_s\right)
  \cdot \delta(\chi^{(b)}_l/z),
  \\
  \langle \boldsymbol{\lambda}|\psi^\pm_j(z)|\boldsymbol{\lambda}\rangle&=
  \prod_{a=1}^r \prod_{s\geq 1}^{c^{(a)}_s(\bol)\equiv j} \psi(\chi^{(a)}_s/q_1z)^\pm
  \prod_{a=1}^r \prod_{s\geq 1}^{c^{(a)}_s(\bol)\equiv j+1} \psi(z/\chi^{(a)}_s)^\pm,
\end{aligned}
\end{equation*}
while all other matrix coefficients are set to be zero.

\noindent
(b) For $n=1$, the same formulas with the matrix coefficient of $f_0(z)$ multiplied by
$\frac{q(1-q_3)}{1-q_1^{-1}}$ define an action of $\U^{(1)}_{q_1,q_2,q_3}$ on $F^{\mathbf{0}}(\uu)$.
\end{prop}

\begin{rem}
The parameters $\{u_k\}$ are \emph{not in resonance} exactly when
the first two formulas are well-defined (do not have zeroes in denominators)
for any $r$-tuples of partitions $\boldsymbol{\lambda}, \boldsymbol{\lambda}+1^{(b)}_l$.
\end{rem}

Let $r\in \ZZ_{>0}, \pp\in [n]^r, \vv\in \CC^r$, and assume that $\{v_k\}_{k=1}^r$
are \emph{not in resonance}. Considering the \emph{additivization} of the above proposition,
we get an action of $\Y^{(n)}_{h_1,h_2,h_3}$ on the vector space ${{^a}F}^{\pp}(\vv)$ with
the basis $\{|\boldsymbol{\lambda}\rangle\}$ labeled by $r$-tuples of partitions.
Set $x^{(a)}_s:=\lambda^{(a)}_sh_1+(s-1)h_3+v_a$.

\begin{prop}\label{Fock_ar}
(a) For $n>1$, the following formulas define an action of
$\Y^{(n)}_{h_1,h_2,h_3}\ \mathrm{on}\ {{^a}F}^{\pp}(\vv)$
\begin{equation*}
\begin{aligned}
  \langle\boldsymbol{\lambda}+1_l^{(b)}| x^+_j(z)|\boldsymbol{\lambda}\rangle&=
  \bar{\delta}_{c^{(b)}_l(\bol), j+1}
  \prod_{(a,s)\prec (b,l)}^{c^{(a)}_s(\bol)\equiv j}   \frac{x^{(a)}_s-x^{(b)}_l+h_3}{x^{(a)}_s-x^{(b)}_l-h_1}
  \prod_{(a,s)\prec (b,l)}^{c^{(a)}_s(\bol)\equiv j+1} \frac{x^{(b)}_l-x^{(a)}_s-h_2}{x^{(b)}_l-x^{(a)}_s}
  \cdot \frac{\delta^+(\frac{x^{(b)}_l}{z})}{z},
\\
  \langle \boldsymbol{\lambda}|x^-_j(z)|\boldsymbol{\lambda}+1_l^{(b)}\rangle&=
  \bar{\delta}_{c^{(b)}_l(\bol), j+1}
  \prod_{(a,s)\succ (b,l)}^{c^{(a)}_s(\bol)\equiv j}   \frac{x^{(a)}_s-x^{(b)}_l+h_3}{x^{(a)}_s-x^{(b)}_l-h_1}
  \prod_{(a,s)\succ (b,l)}^{c^{(a)}_s(\bol)\equiv j+1} \frac{x^{(b)}_l-x^{(a)}_s-h_2}{x^{(b)}_l-x^{(a)}_s}
  \cdot \frac{\delta^+(\frac{x^{(b)}_l}{z})}{z},
\\
  \langle \boldsymbol{\lambda}|\xi_j(z)|\boldsymbol{\lambda}\rangle&=
  \left(\prod_{a=1}^r \prod_{s\geq 1}^{c^{(a)}_s(\bol)\equiv j}   \frac{x^{(a)}_s-z+h_3}{x^{(a)}_s-z-h_1}
        \prod_{a=1}^r \prod_{s\geq 1}^{c^{(a)}_s(\bol)\equiv j+1} \frac{z-x^{(a)}_s-h_2}{z-x^{(a)}_s}\right)^+,
\end{aligned}
\end{equation*}
while all other matrix coefficients are set to be zero.

\noindent
(b) For $n=1$, the same formulas with the matrix coefficient of $x^-_0(z)$ multiplied by $-h_3/h_1$
define an action of $\Y^{(1)}_{h_1,h_2,h_3}$ on ${{^a}F}^{\mathbf{0}}(\vv)$.
\end{prop}

\begin{rem}
A short proof of Proposition~\ref{Fock_ar} is based on the identification
${{^a}F}^{\pp}(\vv)\simeq {{^a}F}^{p_1}(v_1)\otimes \cdots \otimes {{^a}F}^{p_r}(v_r)$,
where ${{^a}F}^{p_k}(v_k)$ are viewed as $\CCD\Y^{(n)}_{h_1,h_2,h_3}$-modules,
see Remarks~\ref{Double},~\ref{Fock double}.
\end{rem}


\section{Formal algebras and their classical limits}\label{section formal versions}

In this section, we introduce the formal versions of our algebras of interest and
relate their \emph{classical limits} to the well-known algebras of difference
and differential operators on $\CC^\times$. We work in the formal setting, that is,
over $\CC[[\hbar]]$ or $\CC[[\hbar_1,\hbar_2]]$ where $\hbar,\hbar_1, \hbar_2$ are
formal variables (here $\CC[[\hbar_1,\hbar_2]]$ is the completion of $\CC[\hbar_1,\hbar_2]$
with respect to the $\NN$-grading with $\deg(\hbar_1)=\deg(\hbar_2)=1$).
Our notations follow~\cite{T2}.


\subsection{Algebras $\dd^{(n)}_q$ and $\bar{\dd}^{(n)}_q$}
$\ $

For $q\in \CC[[\hbar]]^\times$, define the algebra of \emph{$q$-difference operators
on $\CC^\times$}, denoted by $\dd_q$, to be the unital associative $\CC[[\hbar]]$-algebra
topologically generated by $Z^{\pm 1}, D^{\pm 1}$ with the defining relations:
  $$Z^{\pm1}Z^{\mp 1}=1,\ D^{\pm 1}D^{\mp 1}=1,\ DZ=q\cdot ZD.$$
Define the associative algebra $\dd^{(n)}_q:=\MM_n\otimes \dd_q$,
where $\MM_n$ stands for the algebra of $n\times n$ matrices (so that $\dd^{(n)}_q$
is the algebra of $n\times n$ matrices with values in $\dd_q$). We will view $\dd^{(n)}_q$ as
a Lie algebra with the natural Lie bracket -- the commutator $[\cdot,\cdot]$.
It is easy to check that the following formula defines a 2-cocycle
$\phi_{\dd_q^{(n)}}\in C^2(\dd^{(n)}_q,\CC[[\hbar]])$:
  $$\phi_{\dd_q^{(n)}}(M_1\otimes D^{k_1}Z^{l_1},M_2\otimes D^{k_2}Z^{-l_2})=
    \left\{
     \begin{array}{ll}
        \tr(M_1M_2)\cdot q^{-l_1k_2}\frac{1-q^{l_1(k_1+k_2)}}{1-q^{k_1+k_2}} & \text{if}\ l_1=l_2,\\
        0 & \text{otherwise},
     \end{array}
    \right.$$
for any $M_1,M_2\in \MM_n$ and $k_1,k_2,l_1,l_2\in \ZZ$.
Here $\frac{1-q^{l_1(k_1+k_2)}}{1-q^{k_1+k_2}}\in \CC[[\hbar]]$ is understood
in the sense of evaluating $\frac{1-x^{l_1}}{1-x}\in \CC[x^{\pm 1}]$ at $x=q^{k_1+k_2}$.
In particular, $\frac{1-q^{l_1(k_1+k_2)}}{1-q^{k_1+k_2}}=l_1$ if $k_1+k_2=0$.

This endows $\bar{\dd}^{(n)}_q:=\dd^{(n)}_q\oplus \CC[[\hbar]]\cdot c_\dd$ with
the Lie algebra structure via
  $[X+\lambda c_\dd, Y+\mu c_\dd]=XY-YX+\phi_{\dd_q^{(n)}}(X,Y)c_\dd$
for any $X,Y\in \dd^{(n)}_q$ and $\lambda,\mu\in \CC[[\hbar]]$, so that $c_\dd$ is central.


\subsection{Algebras $\D^{(n)}_\hbar$ and $\bar{\D}^{(n)}_\hbar$}
$\ $

Define the algebra of \emph{$\hbar$-differential operators on $\CC^\times$}, denoted  by
$\D_\hbar$, to be the unital associative $\CC[[\hbar]]$-algebra topologically generated
by $\partial, x^{\pm 1}$ with following defining relations:
  $$x^{\pm 1}x^{\mp 1}=1,\ \partial x=x(\partial+\hbar).$$
Define the associative algebra $\D^{(n)}_\hbar:=\MM_n\otimes \D_\hbar$ (so that
$\D^{(n)}_\hbar$ is the algebra of $n\times n$ matrices with values in $\D_\hbar$).
We will view $\D^{(n)}_\hbar$ as a Lie algebra with the natural Lie bracket -- the
commutator $[\cdot,\cdot]$. Following~\cite[Formula (2.3)]{BKLY}, consider a 2-cocycle
$\phi_{\D_\hbar^{(n)}}\in C^2(\D^{(n)}_\hbar,\CC[[\hbar]])$:
  $$\phi_{\D_\hbar^{(n)}}(M_1\otimes f_1(\partial)x^{k}, M_2\otimes f_2(\partial)x^{l})=
    \left\{
     \begin{array}{llr}
        \tr(M_1M_2)\sum_{a=0}^{k-1}f_1(a\hbar)f_2((a-k)\hbar) & \text{if}\ k=-l>0,\\
        -\tr(M_1M_2)\sum_{a=0}^{-k-1}f_2(a\hbar)f_1((a+k)\hbar) & \text{if}\ k=-l<0,\\
        0 & \text{otherwise},
     \end{array}
    \right.$$
for arbitrary polynomials $f_1,f_2$ and any $M_1,M_2\in \MM_n,\ k,l\in \ZZ$.

This endows $\bar{\D}^{(n)}_\hbar:=\D^{(n)}_\hbar\oplus \CC[[\hbar]]\cdot c_\D$ with
the Lie algebra structure via
  $[X+\lambda c_\D, Y+\mu c_\D]=XY-YX+\phi_{\D_\hbar^{(n)}}(X,Y)c_\D$
for any $X,Y\in \D^{(n)}_\hbar$ and $\lambda,\mu\in \CC[[\hbar]]$, so that $c_\D$ is central.


\subsection{Homomorphism $\bar{\Upsilon}^{\omega_n}_{m,n}$}\label{section completions}
$\ $

In this section, we assume that $q-\omega_N\in \hbar\CC[[\hbar]]^\times$ for a certain
$N$-th root of unity $\omega_N=\sqrt[N]{1}\in \CC^\times$. Let us consider the completions
of $\dd^{(n)}_q,\bar{\dd}^{(n)}_q$ and $\D^{(n)}_\hbar,\bar{\D}^{(n)}_\hbar$ with respect
to the ideals $J_{\dd^{(n)}_q}=\MM_n\otimes (D^N-1,q-\omega_N)$ and
$J_{\D^{(n)}_\hbar}=\MM_n\otimes (\partial,\hbar)$:

\noindent
$\bullet$
  $\widehat{\dd}^{(n)}_q:=
   \underset{\longleftarrow}\lim\ \dd^{(n)}_q/\dd^{(n)}_q\cdot (D^N-1,q-\omega_N)^r$
and
  $\widehat{\bar{\dd}}^{(n)}_q:=
   \underset{\longleftarrow}\lim\ \bar{\dd}^{(n)}_q/\bar{\dd}^{(n)}_q\cdot (D^N-1,q-\omega_N)^r$;

\noindent
$\bullet$
  $\widehat{\D}^{(n)}_\hbar:=
   \underset{\longleftarrow}\lim\ \D^{(n)}_\hbar/\D^{(n)}_\hbar\cdot (\partial,\hbar)^r$
and
  $\widehat{\bar{\D}}^{(n)}_\hbar:=
   \underset{\longleftarrow}\lim\ \bar{\D}^{(n)}_\hbar/\bar{\D}^{(n)}_\hbar\cdot (\partial,\hbar)^r$.

\begin{rem}\label{classical algebras}
(a) Taking completions of $\dd^{(n)}_q$ and $\D^{(n)}_\hbar$ with respect to the ideals
$J_{\dd^{(n)}_q}$ and $J_{\D^{(n)}_\hbar}$ commutes with taking central extensions
with respect to the 2-cocycles $\phi_{\dd^{(n)}_q}$ and $\phi_{\D^{(n)}_\hbar}$.

\noindent
(b) Specializing $\hbar$ or $q$ to complex parameters $h_0\in \CC$ or $q_0\in \CC^\times$,
we get the matrix algebras $\dd^{(n)}_{q_0}$ and $\D^{(n)}_{h_0}$ with values in the classical
$\CC$-algebras of difference/differential operators on $\CC^\times$ as well as their
one-dimensional central extensions. The latter are the $\CC$-algebras given by the same
collections of the generators and the defining relations. However, one can not define their
completions as above. This is one of the key reasons we choose to work in the formal setting.
\end{rem}

For $m,n\in \ZZ_{>0}$, we identify $\MM_m\otimes\MM_n\simeq \MM_{mn}$ via
$E_{a,b}\otimes E_{k,l}\mapsto E_{m(k-1)+a,m(l-1)+b}$ for any
$1\leq a,b\leq m,\ 1\leq k,l\leq n$. Our next result relates the above
different families of completions.

\begin{thm}\label{Upsilon_0}
(a) Fix an $n$-th root of unity $\omega_n$ and set $q:=\omega_n\exp(\hbar)$.
The assignment
  $$D\mapsto
    \left(
    \begin{array}{llllll}
      q^{n-1}e^{n\partial}&0&0&\cdots&0&0\\
      0&q^{n-2}e^{n\partial}&0&\cdots&0&0\\
      0&0&q^{n-3}e^{n\partial}&\cdots&0&0\\
      \vdots&\vdots&\vdots&\ddots&\vdots&\vdots\\
      0&0&0&\cdots&qe^{n\partial}&0\\
      0&0&0&\cdots&0&e^{n\partial}
    \end{array}
    \right),
    Z\mapsto
    \left(
    \begin{array}{llllll}
      0&1&0& \cdots &0&0\\
      0&0&1& \cdots &0&0\\
      \vdots&\vdots&\vdots&\ddots&\vdots&\vdots\\
      0&0&0& \cdots &1 &0\\
      0&0&0& \cdots &0&1\\
      x&0&0& \cdots &0&0\\
    \end{array}
    \right)$$
gives rise to  a $\CC[[\hbar]]$-algebra homomorphism
  $\Upsilon^{\omega_n}_{1,n}\colon\widehat{\dd}^{(1)}_{\omega_n\exp(\hbar)}\to \widehat{\D}^{(n)}_\hbar.$

\noindent
(b) Combining $\Upsilon^{\omega_n}_{1,n}$ from part (a) with the above
identification $\MM_m\otimes\MM_n\simeq \MM_{mn}$, we get a $\CC[[\hbar]]$-algebra
homomorphism
  $\Upsilon^{\omega_n}_{m,n}\colon\widehat{\dd}^{(m)}_{\omega_n\exp(\hbar)}\to \widehat{\D}^{(mn)}_\hbar$
for any $m,n\in \ZZ_{>0}$.

\noindent
(c) The assignment
  $c_\dd\mapsto c_\D,\ A\mapsto \Upsilon^{\omega_n}_{m,n}(A)\
   \mathrm{for}\ A \in \dd^{(m)}_{\omega_n\exp(\hbar)}$
gives rise to a $\CC[[\hbar]]$-algebra homomorphism
  $\bar{\Upsilon}^{\omega_n}_{m,n}\colon\widehat{\bar{\dd}}^{(m)}_{\omega_n\exp(\hbar)}
   \to \widehat{\bar{\D}}^{(mn)}_\hbar$.

\noindent
(d) If $\omega_n$ is a primitive $n$-th root of unity, then $\Upsilon^{\omega_n}_{m,n}$
and $\bar{\Upsilon}^{\omega_n}_{m,n}$ are isomorphisms.
\end{thm}

\begin{proof}[Proof of Theorem~\ref{Upsilon_0}]
$\ $

(a) Let us denote the above $n\times n$ matrices by $X$ and $Y$, respectively.
They are invertible and satisfy the identity $XY=qYX$ (which follows from
$e^{n\partial}xe^{-n\partial}=e^{n\hbar}x=q^nx$). Hence, there exists a
$\CC[[\hbar]]$-algebra homomorphism
  $\Upsilon^{\omega_n}_{1,n}\colon\dd^{(1)}_{\omega_n\exp(\hbar)}\to \widehat{\D}^{(n)}_\hbar$
such that $\Upsilon^{\omega_n}_{1,n}(D^{\pm 1})=X^{\pm 1}$ and
$\Upsilon^{\omega_n}_{1,n}(Z^{\pm 1})=Y^{\pm 1}$. Since  $q-\omega_n\in \hbar\CC[[\hbar]]$ and
$\Upsilon^{\omega_n}_{1,n}(D^n-1)\in \widehat{\D}^{(n)}_\hbar\cdot (\partial,\hbar)$, the above
homomorphism induces the homomorphism
$\widehat{\dd}^{(1)}_{\omega_n\exp(\hbar)}\to \widehat{\D}^{(n)}_\hbar$
also denoted by $\Upsilon^{\omega_n}_{1,n}$.

(b) Follows immediately from (a).

(c) It suffices to check the following equality:
  $$\phi_{\dd^{(m)}_{\omega_n\exp(\hbar)}}(M_1\otimes D^{k_1}Z^{l_1},M_2\otimes D^{k_2}Z^{l_2})=
    \phi_{\widehat{\D}_\hbar^{(mn)}}(\Upsilon^{\omega_n}_{m,n}(M_1\otimes D^{k_1}Z^{l_1}),\Upsilon^{\omega_n}_{m,n}(M_2\otimes D^{k_2}Z^{l_2}))$$
for any $M_1,M_2\in \MM_m$ and $k_1,k_2,l_1,l_2\in \ZZ$.
This is a straightforward verification.

(d) Let us now assume that $\omega_n$ is a primitive $n$-th root of unity.
To prove $\Upsilon^{\omega_n}_{m,n}$ is an isomorphism, it suffices to show
that the induced linear map
  $\Upsilon^{\omega_n;r}_{m,n}\colon \dd^{(m)}_{\omega_n\exp(\hbar)}/(D^n-1,\hbar)^r
   \to \D^{(mn)}_\hbar/(\partial,\hbar)^r$
is an isomorphism for any $r\in \ZZ_{>0}$. For the latter, it suffices
to prove that $\Upsilon^{\omega_n;r}_{1,n}$ is an isomorphism, due
to our definition of $\Upsilon^{\omega_n}_{m,n}$.

For any $r\in \ZZ_{>0}$, the following holds:

\noindent
$\bullet$
 $\{\hbar^{s_1}D^{s_2}Z^k|k\in \ZZ, s_1,s_2\in \NN,ns_1+s_2<nr\}$
is a $\CC$-basis of $\dd^{(1)}_{\omega_n\exp(\hbar)}/(D^n-1,\hbar)^r$,

\noindent
$\bullet$
  $\{E_{a,b}\otimes \hbar^{s_1}\partial^{s_2} x^k|1\leq a,b\leq n, k\in \ZZ, s_1,s_2\in \NN, s_1+s_2<r\}$
is a $\CC$-basis of $\D^{(n)}_\hbar/(\partial,\hbar)^r$,

\noindent
$\bullet$
 the linear map $\Upsilon^{\omega_n;r}_{1,n}$ induces a linear map
 $\Upsilon^{\omega_n;r,0}_{1,n}\colon V_{r,0}\to W_{r,0}$, where
\begin{align*}
  V_{r,0}&\ :=\mathrm{span}_\CC\{\hbar^{s_1}D^{s_2}|s_1,s_2\in \NN,ns_1+s_2<nr\}-
  \mathrm{subspace\ of}\ \dd^{(1)}_{\omega_n\exp(\hbar)}/(D^n-1,\hbar)^r,\\
  W_{r,0}&\ :=\mathrm{span}_\CC\{E_{a,a}\otimes \hbar^{s_1}\partial^{s_2}|1\leq a\leq n, s_1,s_2\in \NN,s_1+s_2<r\}-
  \mathrm{subspace\ of}\ \D^{(n)}_\hbar/(\partial,\hbar)^r.
\end{align*}

Explicit formulas for powers of the matrix $Y$ imply that $\Upsilon^{\omega_n;r}_{1,n}$
is an isomorphism if and only if $\Upsilon^{\omega_n;r,0}_{1,n}$ is an isomorphism.
The latter is equivalent to $\Upsilon^{\omega_n;r,0}_{1,n}$ being surjective as
  $$\dim (V_{r,0})=\frac{nr(r+1)}{2}=\dim(W_{r,0}).$$
For any $0\leq s\leq r-1$, the restriction of $\Upsilon^{\omega_n;r,0}_{1,n}$
to $\mathrm{span}_\CC\{\hbar^s\cdot (D^n-1)^{r-s-1}\cdot D^i|0\leq i\leq n-1\}$
maps it isomorphically onto
  $\mathrm{span}_\CC\{E_{k,k}\otimes \hbar^s(n\partial+(n-k)\hbar)^{r-s-1}|1\leq k\leq n\}$.
It is here that we use the fact that $\omega_n$ is a primitive $n$-th root of unity.
Therefore:
  $$\{E_{k,k}\otimes \hbar^{s}\partial^{r-s-1}|1\leq k\leq n, 0\leq s\leq r-1\}\subset
    \mathrm{Im}(\Upsilon^{\omega_n;r,0}_{1,n}).$$
Considering now the restriction of $\Upsilon^{\omega_n;r,0}_{1,n}$ to
  $\mathrm{span}_\CC\{\hbar^s\cdot (D^n-1)^{r-s-2}\cdot D^i\}_{0\leq s\leq r-2}^{0\leq i\leq n-1}$
and combining this with the aforementioned inclusion, we get
  $$\{E_{k,k}\otimes \hbar^{s}\partial^{r-s-2}|1\leq k\leq n, 0\leq s\leq r-2\}\subset
    \mathrm{Im}(\Upsilon^{\omega_n;r,0}_{1,n}).$$
Proceeding further by induction, we see that $\Upsilon^{\omega_n;r,0}_{1,n}$ is surjective. Therefore:
  $$\Upsilon^{\omega_n;r,0}_{1,n}-\mathrm{isomorphism}\Rightarrow
    \Upsilon^{\omega_n;r}_{1,n}-\mathrm{isomorphism}\Rightarrow
    \Upsilon^{\omega_n;r}_{m,n}-\mathrm{isomorphism}\Rightarrow
    \Upsilon^{\omega_n}_{m,n}-\mathrm{isomorphism}.$$
Combining this with part (c), we also see that $\bar{\Upsilon}^{\omega_n}_{m,n}$
is a $\CC[[\hbar]]$-algebra isomorphism.
\end{proof}


\subsection{Algebras $\U^{(m),\omega}_{\hbar_1,\hbar_2}$ and $\U^{(m),\omega}_{\hbar_1}$}
$\ $

Throughout this section, we fix a root of unity $\omega\in \CC^\times$ and let
$\hbar_1,\hbar_2$ be formal variables, while we set $\hbar_3:=-\hbar_1-\hbar_2$.
First, we introduce the \emph{formal version} of the quantum toroidal algebra
$\U^{(m)}_{q_1,q_2,q_3}$ with
  $q_1=\omega e^{h_1/m}, q_2=e^{h_2/m}, q_3=\omega^{-1} e^{-(h_1+h_2)/m}$.
Define
  $$\q_1:=\omega \exp(\hbar_1/m), \q_2:=\exp(\hbar_2/m),
    \q_3:=\omega^{-1} \exp(\hbar_3/m)\in \CC[[\hbar_1,\hbar_2]]^\times.$$
Note that replacing $q_i$ by $\q_i$, the relations (T0,T2--T6) are defined
over $\CC[[\hbar_1,\hbar_2]]$, while (T1) is not well-defined as we have
$\q-\q^{-1}$ in the denominator. To fix this, we will rather use the generators
$h_{i,k}$, where we present $\psi^{\pm 1}_{i,0}$ in the form
  $\psi^{\pm 1}_{i,0}=\exp\left(\pm \frac{\hbar_2}{2m}h_{i,0}\right)$,
so that
  $$\psi^\pm_i(z)=\exp\left(\pm \frac{\hbar_2}{2m}h_{i,0}\right)\cdot
    \exp\left(\pm(\q-\q^{-1})\sum_{r>0} h_{i,\pm r}z^{\mp r}\right)\
    \mathrm{with}\ \q=\sqrt{\q_2}=\exp\left(\frac{\hbar_2}{2m}\right).$$

Switching from $\{\psi_{i,k},\psi_{i,0}^{-1}\}_{i\in [m]}^{k\in \ZZ}$ to
$\{h_{i,k}\}_{i\in [m]}^{k\in \ZZ}$, the relations (T4,T5) get modified to
\begin{equation}\tag{H}\label{H}
  [h_{i,k}, e_{j,l}]=b_m(i,j;k)\cdot e_{j,l+k},\
  [h_{i,k}, f_{j,l}]=-b_m(i,j;k)\cdot f_{j,l+k}\
  \mathrm{for}\ i,j\in [m],\ k,l\in \ZZ,
\end{equation}
where $b_m(i,j;0):=a_{i,j}$, while $b_m(i,j;k)$ is given by the formula~(\ref{sharp})
from Section~\ref{section toroidal} for $k\ne 0$. These relations are well-defined in
the formal setting as $[k]_{\q}\in \CC[[\hbar_1,\hbar_2]]$. We also note that the
right-hand side of (T1) is now a series in $z^{\pm 1},w^{\pm 1}$ with coefficients in
$\CC[[\hbar_1,\hbar_2]][\{h_{i,k}\}_{i\in [m]}^{k\in \ZZ}]$.

\begin{defn}
$\U^{(m),\omega}_{\hbar_1,\hbar_2}$ is the unital associative
$\CC[[\hbar_1,\hbar_2]]$-algebra topologically generated by
$\{e_{i,k},f_{i,k},h_{i,k}\}_{i\in [m]}^{k\in \ZZ}$ with the defining relations
(T0--T3,H,T6) whereas $q_i\rightsquigarrow \q_i, n\rightsquigarrow m$.
\end{defn}

Its \emph{classical limit} $\U^{(m),\omega}_{\hbar_1}$ is defined by
  $$\U^{(m),\omega}_{\hbar_1}:=\U^{(m),\omega}_{\hbar_1,\hbar_2}/(\hbar_2).$$
It is the unital associative $\CC[[\hbar_1]]$-algebra topologically generated by
 $\{e_{i,k},f_{i,k},h_{i,k}\}_{i\in [m]}^{k\in \ZZ}$
subject to the relations (T2,T3,T6) (whereas
  $q_1\rightsquigarrow \q_1, q_2\rightsquigarrow 1, q_3\rightsquigarrow \q_1^{-1},
   q\rightsquigarrow 1, d\rightsquigarrow \q_1, n\rightsquigarrow m$)
and
\begin{equation}\tag{T0L}\label{T40L}
  [h_{i,k}, h_{j,l}]=0,
\end{equation}
\begin{equation}\tag{T1L}
  [e_{i,k},f_{j,l}]=\delta_{i,j}\cdot h_{i,l+k},
\end{equation}
\begin{equation}\tag{T4L}\label{T4L}
  [h_{i,k}, e_{j,l}]=b'_m(i,j;k)\cdot e_{j,l+k},
\end{equation}
\begin{equation}\tag{T5L}\label{T5L}
  [h_{i,k}, f_{j,l}]=-b'_m(i,j;k)\cdot f_{j,l+k},
\end{equation}
for all $i,j\in [m]$ and $k,l\in \ZZ$. Here $b'_m(i,j;k)\in \CC[[\hbar_1]]$ is
the image of $b_m(i,j;k)\in \CC[[\hbar_1,\hbar_2]]$:
  $$b'_m(i,j;k)=-\q_1^k\delta_{j,i+1}+2\delta_{j,i}-\q_1^{-k}\delta_{j,i-1}=
    \begin{cases}
      a_{i,j}\cdot \q_1^{-km_{i,j}} & \text{if}\ \ m>2,\\
      2\delta_{i,j}-(\q_1^k+\q_1^{-k})\delta_{i+1,j} & \text{if}\ \ m=2,\\
      2-\q_1^k-\q_1^{-k} & \text{if}\ \ m=1.
    \end{cases}$$

\begin{rem}\label{numeric limit m}
Specializing $\hbar_1$ to a complex parameter $h_1\in \CC$, we obtain
a $\CC$-algebra $\U^{(m),\omega}_{h_1}$ generated by
$\{e_{i,k},f_{i,k},h_{i,k}\}_{i\in [m]}^{k\in \ZZ}$ with the same defining
relations (T0L,T1L,T2,T3,T4L,T5L,T6) whereas
$\q_1\rightsquigarrow q_1:=\omega e^{\frac{h_1}{m}}\in \CC^\times$.
\end{rem}

The following result is straightforward:

\begin{prop}\label{limit_m}
The assignment
  $$e_{0,k}\mapsto E_{m,1}\otimes D^kZ,\ f_{0,k}\mapsto E_{1,m}\otimes Z^{-1}D^k,\
    h_{0,k}\mapsto E_{m,m}\otimes D^k-E_{1,1}\otimes (\q_1^mD)^k+c_\dd,$$
  $$e_{i,k}\mapsto E_{i,i+1}\otimes (\q_1^{m-i}D)^k,\ f_{i,k}\mapsto E_{i+1,i}\otimes (\q_1^{m-i}D)^k,
    h_{i,k}\mapsto (E_{i,i}-E_{i+1,i+1})\otimes (\q_1^{m-i}D)^k$$
for $i\in [m]\backslash\{0\}, k\in \ZZ$, gives rise to a $\CC[[\hbar_1]]$-algebra
homomorphism $\theta^{(m)}\colon \U^{(m),\omega}_{\hbar_1}\to U(\bar{\dd}^{(m)}_{\q_1^m})$.
\end{prop}

Define a free $\CC[[\hbar_1]]$-submodule
$\bar{\dd}^{(m),0}_{\q_1^m}\subset \bar{\dd}^{(m)}_{\q_1^m}$ as follows:

\noindent
$\bullet$
For $m\geq 2$, $\bar{\dd}^{(m),0}_{\q_1^m}$ is spanned by
  $$\{\CC[[\hbar_1]]c_\dd, A_{k,l}\otimes D^kZ^l|k,l\in \ZZ,A_{k,l}\in \MM_m\otimes \CC[[\hbar_1]],
    \tr(A_{k,l})\in \hbar_1\CC[[\hbar_1]], \tr(A_{0,0})=0\};$$

\noindent
$\bullet$
 For $m=1$, $\bar{\dd}^{(m),0}_{\q_1^m}$ is spanned by
 $\{\CC[[\hbar_1]]c_\dd, \hbar_1\CC[[\hbar_1]]D^{\pm s}, \hbar_1^{s-1}\CC[[\hbar_1]]D^kZ^{\pm s}|k\in \ZZ, s\in \ZZ_{>0}\}.$

\begin{lem}
 $\bar{\dd}^{(m),0}_{\q_1^m}$ is a Lie subalgebra of $\bar{\dd}^{(m)}_{\q_1^m}$ and
 $\mathrm{Im}(\theta^{(m)})\subset U(\bar{\dd}^{(m),0}_{\q_1^m})$.
\end{lem}

In fact, we have the following result:

\begin{thm}\label{limit_1}
$\theta^{(m)}$ gives rise to an isomorphism
 $\theta^{(m)}\colon \U^{(m),\omega}_{\hbar_1}\iso U(\bar{\dd}^{(m),0}_{\q_1^m})$.
\end{thm}

Actually, a more general result is proved in~\cite[Theorem 2.1]{T2}:

\begin{thm}\label{limit_1.1}
For $h_1\in \CC\backslash \{\mathbb{Q}\cdot \pi \sqrt{-1}\}$, let
$\U^{(m),\omega}_{h_1}$ be as in Remark~\ref{numeric limit m} and
$\bar{\dd}^{(m),0}_{q_1^m}\subset \bar{\dd}^{(m)}_{q_1^m}$ be the Lie subalgebra
spanned by
  $\{c_\dd,A_{k,l}\otimes D^kZ^l|k,l\in \ZZ, A_{k,l}\in \MM_m, \tr(A_{0,0})=0\}$.
Then, the same formulas define a $\CC$-algebra isomorphism
$\theta^{(m)}\colon \U^{(m),\omega}_{h_1}\iso U(\bar{\dd}^{(m),0}_{q_1^m})$.
\end{thm}

Since all the defining relations of $\U^{(m),\omega}_{\hbar_1}$ are of \emph{Lie type},
it is isomorphic to an enveloping algebra of the Lie algebra generated by
$\{e_{i,k},f_{i,k}, h_{i,k}\}_{i\in [m]}^{k\in \ZZ}$ with the same defining relations.
Thus, Theorem~\ref{limit_1} provides a presentation of the Lie algebra
$\bar{\dd}^{(m),0}_{\q_1^m}$ by generators and relations.


\subsection{Algebras $\Y^{(n)}_{\hbar_1,\hbar_2}$ and $\Y^{(n)}_{\hbar_1}$}\label{section formal Yangian}
$\ $

Analogously to the previous section, let $\hbar_1, \hbar_2$ be formal variables
and set $\hbar_3:=-\hbar_1-\hbar_2$.

\begin{defn}
$\Y^{(n)}_{\hbar_1,\hbar_2}$ is the unital associative $\CC[[\hbar_1,\hbar_2]]$-algebra
topologically generated by $\{x^\pm_{i,r},\xi_{i,r}\}_{i\in [n]}^{r\in \NN}$
with the defining relations (Y0--Y5) whereas $h_i \rightsquigarrow \hbar_i/n$.
\end{defn}

We equip the algebra $\Y^{(n)}_{\hbar_1,\hbar_2}$ with the $\NN$-grading via
$\deg(x^\pm_{i,r})=\deg(\xi_{i,r})=r, \deg(\hbar_s)=1$ for all
$i\in [n], r\in \NN, s\in \{1,2,3\}$. Its \emph{classical limit} $\Y^{(n)}_{\hbar_1}$
(a formal version of the $\CC$-algebra
$\Y^{(n)}_{h_1}:=\Y^{(n)}_{h_1/n,0,-h_1/n}\ \mathrm{with}\ h_1\in \CC$) is defined by
  $$\Y^{(n)}_{\hbar_1}:=\Y^{(n)}_{\hbar_1,\hbar_2}/(\hbar_2).$$
It is a unital associative $\CC[[\hbar_1]]$-algebra.
The following result is straightforward:

\begin{prop}\label{limit_a}
The assignment
  $$x^+_{0,r}\mapsto E_{n,1}\otimes \partial^rx,\ x^+_{i,r}\mapsto E_{i,i+1}\otimes (\partial+(1-i/n)\hbar_1)^r,$$
  $$x^-_{0,r}\mapsto E_{1,n}\otimes x^{-1}\partial^r,\ x^-_{i,r}\mapsto E_{i+1,i}\otimes (\partial+(1-i/n)\hbar_1)^r,$$
  $$\xi_{0,r}\mapsto E_{n,n}\otimes \partial^r - E_{1,1}\otimes (\partial+\hbar_1)^r+\delta_{0,r}c_\D,\
    \xi_{i,r}\mapsto (E_{i,i}-E_{i+1,i+1})\otimes (\partial+(1-i/n)\hbar_1)^r$$
for $i\in [n]\backslash\{0\}, r\in \NN$, gives rise to a $\CC[[\hbar_1]]$-algebra
homomorphism $\vartheta^{(n)}\colon \Y^{(n)}_{\hbar_1}\to U(\bar{\D}^{(n)}_{\hbar_1})$.
\end{prop}

Define a free $\CC[[\hbar_1]]$-submodule
$\bar{\D}^{(n),0}_{\hbar_1}\subset \bar{\D}^{(n)}_{\hbar_1}$ as follows:

\noindent
$\bullet$
 For $n\geq 2$, $\bar{\D}^{(n),0}_{\hbar_1}$ is spanned by
  $$\{\CC[[\hbar_1]]c_\D, A_{r,l}\otimes \partial^rx^l|r\in \NN,l\in \ZZ, A_{r,l}\in \MM_n\otimes \CC[[\hbar_1]],
    \tr(A_{r,l})\in \hbar_1\CC[[\hbar_1]]\};$$

\noindent
$\bullet$
For $n=1$, $\bar{\D}^{(n),0}_{\hbar_1}$ is spanned by
 $\{\CC[[\hbar_1]]c_\D, \hbar_1\CC[[\hbar_1]]\partial^r, \hbar_1^{s-1}\CC[[\hbar_1]]\partial^rx^{\pm s}|r\in \NN, s\in \ZZ_{>0}\}.$

\begin{lem}
 $\bar{\D}^{(n),0}_{\hbar_1}$ is a Lie subalgebra of $\bar{\D}^{(n)}_{\hbar_1}$ and
 $\mathrm{Im}(\vartheta^{(n)})\subset U(\bar{\D}^{(n),0}_{\hbar_1})$.
\end{lem}

In fact, we have the following result:

\begin{thm}\label{limit_2}
$\vartheta^{(n)}$ gives rise to an isomorphism
 $\vartheta^{(n)}\colon \Y^{(n)}_{\hbar_1}\iso U(\bar{\D}^{(n),0}_{\hbar_1})$.
\end{thm}

Actually, a more general result is proved in~\cite[Theorem 2.2]{T2}:

\begin{thm}\label{limit_2.1}
For $h_1\in \CC^\times$, the same formulas define an isomorphism
 $\vartheta^{(n)}\colon \Y^{(n)}_{h_1}\iso U(\bar{\D}^{(n)}_{h_1})$.
\end{thm}

Since all the defining relations of $\Y^{(n)}_{\hbar_1}$ are of \emph{Lie type}, it is
isomorphic to an enveloping algebra of the Lie algebra generated by
$\{x^\pm_{i,r},\xi_{i,r}\}_{i\in [n]}^{r\in \NN}$ with the same defining relations.
Thus, Theorem~\ref{limit_2} provides a presentation of the Lie algebra
$\bar{\D}^{(n),0}_{\hbar_1}$ by generators and relations.


\subsection{Flatness and faithfulness}
$\ $

The main result of this section is:

\begin{thm}\label{flatness}
(a) The algebra $\U^{(m),\omega}_{\hbar_1,\hbar_2}$ is a flat
$\CC[[\hbar_2]]$-deformation of $\U^{(m),\omega}_{\hbar_1}\simeq U(\bar{\dd}^{(m),0}_{\q_1^m})$.

\noindent
(b) The algebra $\Y^{(n)}_{\hbar_1,\hbar_2}$ is a flat
$\CC[[\hbar_2]]$-deformation of $\Y^{(n)}_{\hbar_1}\simeq U(\bar{\D}^{(n),0}_{\hbar_1})$.
\end{thm}

\begin{proof}[Proof of Theorem~\ref{flatness}]
$\ $

To prove Theorem~\ref{flatness}, it suffices to provide a faithful
$U(\bar{\dd}^{(m),0}_{\q_1^m})$-representation
(resp. $U(\bar{\D}^{(n),0}_{\hbar_1})$-representation) which admits
a flat deformation to a representation of $\U^{(m),\omega}_{\hbar_1,\hbar_2}$
(resp. $\Y^{(n)}_{\hbar_1,\hbar_2}$). To make use of the representations from
Sections~\ref{section Fock},~\ref{section tensor Fock}, we will need to work
not over $\CC[[\hbar_1,\hbar_2]]$, but rather over the ring $R$, defined as
a localization of  $\CC[[\hbar_1,\hbar_2]]$ by the multiplicative set
  $\{(\hbar_1-\nu_1 \hbar_2)\cdots (\hbar_1-\nu_s \hbar_2)\}_{s\in \ZZ_{>0}}^{\nu_r\in  \CC}$.
Note that $\bar{R}:=R/(\hbar_2)\simeq \CC((\hbar_1))$.
We define
  $$\U^{(m),\omega}_{R}:=\U^{(m),\omega}_{\hbar_1,\hbar_2}\otimes_{\CC[[\hbar_1,\hbar_2]]} R,\ \ \
    \Y^{(n)}_R:=\Y^{(n)}_{\hbar_1,\hbar_2}\otimes_{\CC[[\hbar_1,\hbar_2]]} R.$$

Consider the Lie algebra
  $\gl_{\infty}:=\{\sum_{k,l\in \ZZ}a_{k,l} E_{k,l}| a_{k,l}\in \CC[[\hbar_1]]\
   \mathrm{and}\ a_{k,l}=0\ \mathrm{for}\ |k-l|\gg 0\}.$
Let
  $\bar{\gl}_{\infty}:=\gl_\infty\oplus \CC[[\hbar_1]]\cdot \kappa$
be the central extension of this Lie algebra via the 2-cocycle
  $$\phi_{\gl_\infty}\left(\sum a_{k,l}E_{k,l}, \sum b_{k,l}E_{k,l}\right)=
    \sum_{k<0\leq l} a_{k,l}b_{l,k}-\sum_{l<0\leq k}a_{k,l}b_{l,k}.$$
For any $u,Q\in \CC[[\hbar_1]]^\times$, consider the homomorphism
$\tau_u\colon U_{\bar{R}}(\bar{\dd}^{(m),0}_Q)\to U_{\bar{R}}(\bar{\gl}_{\infty})$
such that
  $$E_{\alpha,\beta}\otimes Z^kD^l\mapsto \sum_{a\in \ZZ} u^{l} Q^{a l} E_{m(a+k)-\alpha, ma-\beta}+
    \delta_{k,0}\delta_{\alpha,\beta}\frac{1-u^lQ^l}{1-Q^l}\kappa,\ c_{\dd}\mapsto -\kappa,$$
where we set $\frac{1-u^lQ^l}{1-Q^l}:=0$ if $Q^l=1$. In what follows,
we choose $Q:=\q_1^m=\omega^m\exp(\hbar_1)$.

Let
 $\varpi_u\colon \U^{(m),\omega}_{\bar{R}}\to U_{\bar{R}}(\bar{\gl}_{\infty})$
be the composition of $\theta^{(m)}$ and $\tau_u$.
For any $i\in [m]$, we get
  $$\varpi_u(e_i(z))=\sum_{a \in \ZZ} \delta(\q_1^{ma+m-i}u/z) E_{ma-i,ma-i-1},\
    \varpi_u(f_i(z))=\sum_{a \in \ZZ} \delta(\q_1^{ma+m-i}u/z) E_{ma-i-1,ma-i}.$$

For any $0\leq p\leq m-1$, let $F^p_\infty$ be the $(-p-1)$-th fundamental representation
of $\bar{\gl}_{\infty}$. It is realized on $\wedge^{-p-1+\infty/2} \CC^\infty$ with
the highest weight vector $w_{-p-1}\wedge w_{-p-2}\wedge w_{-p-3}\wedge \cdots$
(here $\CC^\infty$ is a $\CC$-vector spaces with the basis $\{w_k\}_{k\in \ZZ}$).
Comparing the formulas for the Fock $\U^{(m),\omega}_{R}$-module $F^{p}_R(u)$ with those
for the $\bar{\gl}_{\infty}$-action on $F^{p}_\infty$, we see that $F^{p}_R(u)$ degenerates
to $\varpi_{\q_1^{p-m}u}^*(F^{p}_\infty)$ (the intertwining linear map is given by
$|\lambda\rangle \mapsto w_{-p+\lambda_1-1}\wedge w_{-p+\lambda_2-2}\wedge \cdots$).
Moreover, it is easy to see that any finite tensor product
$F^{p_1}_R(u_1)\otimes\cdots\otimes F^{p_r}_R(u_r)$ (with $u_1,\ldots,u_r$ not in resonance)
degenerates to
$\varpi_{\q_1^{p_1-m}u_1}^*(F^{p_1}_\infty)\otimes\cdots\otimes \varpi_{\q_1^{p_r-m}u_r}^*(F^{p_r}_\infty)$.
It remains to prove that
  $\bigoplus_{r\geq 1}\bigoplus^{p_1,\ldots,p_r\in [m]}_{u_1,\ldots,u_r\in \CC[[\hbar_1]]^\times}
   \tau_{u_1}^*(F^{p_1}_\infty)\otimes \cdots \otimes \tau_{u_r}^*(F^{p_r}_\infty)$
is a faithful representation of $U(\bar{\dd}^{(m),0}_Q)$.
This follows from the corresponding statement after factoring by $(\hbar_1)$,
where it is obvious.

In the case of $\Y^{(n)}_{\bar{R}}$, we use the homomorphism
  $\varsigma_v\colon U_{\bar{R}}(\bar{\D}^{(n),0}_{\hbar_1})
   \to U_{\bar{R}}(\bar{\gl}_{\infty})$
defined by
\begin{equation*}
  E_{\alpha,\beta}\otimes x^k\partial^r\mapsto \sum_{a\in \ZZ} (v+a \hbar_1)^r E_{n(a+k)-\alpha, na-\beta}-
    \delta_{k,0}\delta_{\alpha,\beta}c_r\kappa,\ c_{\D}\mapsto -\kappa,
\end{equation*}
with the constants $c_N\in {\bar{R}}$ determined recursively from
 $\sum_{a=1}^N \binom{N}{a}\hbar_1^ac_{N-a}=(v+\hbar_1)^N$ for $N\geq 1$.
The rest of the arguments are the same.
This completes our proof of Theorem~\ref{flatness}.
\end{proof}

The above proof also implies the following result:

\begin{cor}\label{faithfulness}
(a) The following is a faithful $\U^{(m),\omega}_{R}$-representation:
  $$\mathbf{F}_R:=\bigoplus_{r\geq 1}\bigoplus^{p_1,\ldots,p_r\in [m]}_{u_1,\ldots,u_r\in \CC[[\hbar_1]]^\times-\mathrm{not\ in\ resonance}}
    F^{p_1}_R(u_1)\otimes \cdots\otimes F^{p_r}_R(u_r).$$

\noindent
(b) The following is a faithful $\Y^{(n)}_{R}$-representation:
  $$^{a}\mathbf{F}_R:=\bigoplus_{r\geq 1}\bigoplus^{p_1,\ldots,p_r\in [n]}_{v_1,\ldots,v_r\in \hbar_1\CC[[\hbar_1]]-\mathrm{not\ in\ resonance}}
    {^{a}F}^{p_1}_R(v_1)\otimes \cdots\otimes {^{a}F}^{p_r}_R(v_r).$$
\end{cor}


\section{Main Result}\label{section main result}

Fix $m,n\geq 1$ and an $mn$-th root of unity $\omega_{mn}=\exp(2\pi \kk\mathbf{i}/mn)$
with $\kk\in \ZZ, \gcd(\kk,n)=1$, whereas $\mathbf{i}$ denotes $\mathbf{i}=\sqrt{-1}$.
Following~\cite{GTL}, we construct a $\CC[[\hbar_1,\hbar_2]]$-algebra homomorphism
  $$\Phi^{\omega_{mn}}_{m,n}\colon \widehat{\U}^{(m),\omega_{mn}}_{\hbar_1,\hbar_2}
    \longrightarrow \widehat{\Y}^{(mn)}_{\hbar_1,\hbar_2}$$
between the appropriate completions of the two algebras of interest.


\subsection{Homomorphism $\Phi^{\omega_{mn}}_{m,n}$}
$\ $

To state our main result, we introduce the following notations (compare to~\cite{GTL}):

\noindent
$\bullet$
Let $\widehat{\Y}^{(mn)}_{\hbar_1,\hbar_2}$ be the completion of $\Y^{(mn)}_{\hbar_1,\hbar_2}$
with respect to the $\NN$-grading from Section~\ref{section formal Yangian}.

\noindent
$\bullet$
Let $\mathfrak{J}\subset \U^{(m),\omega_{mn}}_{\hbar_1,\hbar_2}$ be the kernel
of the composition
  $$\U^{(m),\omega_{mn}}_{\hbar_1,\hbar_2}\overset{\hbar_2\to 0}\longrightarrow
    \U^{(m),\omega_{mn}}_{\hbar_1}\hookrightarrow U(\bar{\dd}^{(m)}_{\q_1^m})\longrightarrow
    U(\bar{\dd}^{(m)}_{\q_1^m}/\MM_m\otimes (D^n-1,\hbar_1)),$$
where the latter quotient is as in Section~\ref{section completions}. We define
  $$\widehat{\U}^{(m),\omega_{mn}}_{\hbar_1,\hbar_2}:=
    \underset{\longleftarrow}\lim\ \U^{(m),\omega_{mn}}_{\hbar_1,\hbar_2}/\mathfrak{J}^r$$
to be the completion of $\U^{(m),\omega_{mn}}_{\hbar_1,\hbar_2}$ with respect
to the ideal $\mathfrak{J}$.

\noindent
$\bullet$
For $i',j'\in [mn]$, we write $i'\equiv j'$ if $i'-j'$ is divisible by  $m$.

\noindent
$\bullet$
For $i\in [m], i'\in [mn]$, we write $i'\equiv i$ if $i=i'\ \mathrm{mod}\ m$.

\noindent
$\bullet$
For $i'\in [mn]$, we define $\xi_{i'}(z)$ as in Section~\ref{section Yangian addit}:
  $$\xi_{i'}(z):=1+\frac{\hbar_2}{mn}\sum_{r\geq 0}\xi_{i',r}z^{-r-1}\in \Y^{(mn)}_{\hbar_1,\hbar_2}[[z^{-1}]].$$

\noindent
$\bullet$
For $i'\in [mn], r\in \NN$, we define $t_{i',r}\in \Y^{(mn)}_{\hbar_1,\hbar_2}$ via
  $$\sum_{r\geq 0}t_{i',r} z^{-r-1}=t_{i'}(z):=\log(\xi_{i'}(z)).$$

\noindent
$\bullet$
Consider the \emph{inverse Borel transform}
  $$B\colon z^{-1}\CC[[z^{-1}]]\longrightarrow \CC[[w]]\ \mathrm{defined\ by}\
    \sum_{r=0}^\infty \frac{a_r}{z^{r+1}}\mapsto \sum_{r=0}^\infty \frac{a_r}{r!}w^r.$$

\noindent
$\bullet$
For $i'\in [mn]$, we define $B_{i'}(w):=B(t_{i'}(z))\in \hbar_2\Y^{(mn)}_{\hbar_1,\hbar_2}[[w]]$.

\noindent
$\bullet$
For $i',j'\in [mn]$ such that $i'\equiv j'$, we define $H_{i',j'}(v)\in 1+v\CC[[v]]$ by
  $$H_{i',j'}(v):=
    \begin{cases}
      \frac{nv}{e^{\frac{nv}{2}}-e^{-\frac{nv}{2}}} & \text{if}\ \ i'=j',\\
      \frac{\omega_{mn}^{-i'}-\omega_{mn}^{-j'}}{\omega_{mn}^{-i'}e^{\frac{nv}{2}}-\omega_{mn}^{-j'}e^{-\frac{nv}{2}}} & \text{if}\ \ i'\ne j'.
    \end{cases}$$

\noindent
$\bullet$
 For $i',j'\in [mn]$ such that $i'\equiv j'$, we define $G_{i',j'}(v):=\log (H_{i',j'}(v)) \in v\CC[[v]].$

\noindent
$\bullet$
For $i',j'\in [mn]$ such that $i'\equiv j'$, we define
  $$\gamma_{i',j'}(v):=-B_{j'}(-\partial_v)\partial_vG_{i',j'}(v)\in \widehat{\Y}^{(mn)}_{\hbar_1,\hbar_2}[[v]].$$

\noindent
$\bullet$
 For $i'\in [mn]$, we define $g_{i'}(v):=\sum_{r\geq 0} g_{i',r} v^r\in \widehat{\Y}^{(mn)}_{\hbar_1,\hbar_2}[[v]]$ by
  $$g_{i'}(v):=\left(\frac{\hbar_2}{m(\q-\q^{-1})}\right)^{1/2}\cdot
    \exp \left(\frac{1}{2}\sum_{j'\equiv i'}^{j'\in [mn]} \gamma_{i',j'}(v)\right).$$

Now we are ready to state our main result:

\begin{thm}\label{main 1}
Fix $m,n\geq 1$ and $\omega_{mn}=\exp(2\pi \kk \mathbf{i}/mn)$ with $\gcd(\kk,n)=1$.
The assignment
\begin{equation}\tag{$\Phi 0$}\label{Phi0}
  h_{i,0} \mapsto \sum^{i'\in [mn]}_{i'\equiv i}\xi_{i',0},
\end{equation}
\begin{equation}\tag{$\Phi 1$}\label{Phi1}
  h_{i,l}\mapsto \frac{n}{\q-\q^{-1}}\sum_{i'\equiv i}^{i'\in [mn]} \omega_{mn}^{-li'}B_{i'}(ln),
\end{equation}
\begin{equation}\tag{$\Phi 2$}\label{Phi2}
  e_{i,k}\mapsto \sum_{i'\equiv i}^{i'\in [mn]} \omega_{mn}^{-ki'} e^{kn\sigma_{i'}^+}g_{i'}(\sigma_{i'}^+)x^+_{i',0},
\end{equation}
\begin{equation}\tag{$\Phi 3$}\label{Phi3}
  f_{i,k}\mapsto \sum_{i'\equiv i}^{i'\in [mn]} \omega_{mn}^{-ki'} e^{kn\sigma_{i'}^-}g_{i'}(\sigma_{i'}^-)x^-_{i',0},
\end{equation}
for $i\in [m], k\in \ZZ, l\in \ZZ\backslash\{0\}$, gives rise to a $\CC[[\hbar_1,\hbar_2]]$-algebra homomorphism
  $$\Phi^{\omega_{mn}}_{m,n}\colon \widehat{\U}^{(m),\omega_{mn}}_{\hbar_1,\hbar_2}\longrightarrow \widehat{\Y}^{(mn)}_{\hbar_1,\hbar_2}.$$
\end{thm}

We present two different proofs of this result in
Sections~\ref{section proof via reps},~\ref{section proof via shuffle},
see also Section~\ref{section partial proof} below.


\subsection{Classical limit of $\Phi^{\omega_{mn}}_{m,n}$}
$\ $

Recall the isomorphisms
  $$\theta^{(m)}\colon \U^{(m),\omega_{mn}}_{\hbar_1,\hbar_2}/(\hbar_2)\iso
    U(\bar{\dd}^{(m),0}_{\q_1^m})\ \mathrm{and}\
    \vartheta^{(mn)}\colon \Y^{(mn)}_{\hbar_1,\hbar_2}/(\hbar_2)\iso
    U(\bar{\D}^{(mn),0}_{\hbar_1})$$
of Theorems~\ref{limit_1} and~\ref{limit_2}, where
  $\q_1=\omega_{mn}\exp(\hbar_1/m)\Rightarrow \q_1^m=\omega_n \exp(\hbar_1)$
with $\omega_n:=\omega^m_{mn}$. Considering factors by $(\hbar_2)$, we get the
\emph{classical limit} of $\Phi^{\omega_{mn}}_{m,n}$ which will be viewed as
  $$\bar{\Phi}^{\omega_{mn}}_{m,n}\colon
    U\left(\widehat{\bar{\dd}}^{(m),0}_{\omega_n\exp(\hbar_1)}\right)
    \longrightarrow U\left(\widehat{\bar{\D}}^{(mn),0}_{\hbar_1}\right).$$
On the other hand, recall the homomorphism
  $\bar{\Upsilon}^{\omega_n}_{m,n}\colon\widehat{\bar{\dd}}^{(m)}_{\omega_n\exp(\hbar_1)}
   \to \widehat{\bar{\D}}^{(mn)}_{\hbar_1}$
from Theorem~\ref{Upsilon_0}.

\begin{prop}\label{limit_Phi}
The limit homomorphism $\bar{\Phi}^{\omega_{mn}}_{m,n}$ is induced
by $\bar{\Upsilon}^{\omega_n}_{m,n}$.
\end{prop}

\begin{proof}[Proof of Proposition~\ref{limit_Phi}]
$\ $

Note that
  $\frac{\hbar_2}{m(\q-\q^{-1})}\equiv 1\ (\mathrm{mod}\ \hbar_2),
   \frac{mn}{\hbar_2}B_{i'}(v)\equiv \sum_{r=0}^\infty \frac{v^r}{r!}\xi_{i',r}\ (\mathrm{mod}\ \hbar_2)
   \Rightarrow g_{i'}(v)\equiv 1\ (\mathrm{mod}\ \hbar_2)$.
Combining this with the identity
  $\sum_{r=0}^\infty \frac{(kn)^r}{r!}(\partial+\frac{s}{mn}\hbar_1)^r=(\omega_{mn}^{-1}\q_1)^{ks}e^{kn\partial}$,
we get:
  $$\Phi^{\omega_{mn}}_{m,n}(h_{i,k})\equiv \sum_{i'\equiv i}\omega_{mn}^{-ki'} \sum_{r=0}^\infty \frac{(kn)^r}{r!}\xi_{i',r}\equiv
    \sum_{a=1}^n\sum_{r=0}^\infty \omega_{mn}^{-k(m(a-1)+i)} \frac{(kn)^r}{r!}\xi_{m(a-1)+i,r}\ (\mathrm{mod}\ \hbar_2),$$
  $$\Phi^{\omega_{mn}}_{m,n}(e_{i,k})\equiv \sum_{i'\equiv i}\omega_{mn}^{-ki'}\sum_{r=0}^\infty \frac{(kn)^r}{r!}x^+_{i',r}\equiv
    \sum_{a=1}^n\sum_{r=0}^\infty \omega_{mn}^{-k(m(a-1)+i)} \frac{(kn)^r}{r!}x^+_{m(a-1)+i,r}\ (\mathrm{mod}\ \hbar_2),$$
  $$\Phi^{\omega_{mn}}_{m,n}(f_{i,k})\equiv \sum_{i'\equiv i}\omega_{mn}^{-ki'}\sum_{r=0}^\infty \frac{(kn)^r}{r!}x^-_{i',r}\equiv
    \sum_{a=1}^n\sum_{r=0}^\infty \omega_{mn}^{-k(m(a-1)+i)} \frac{(kn)^r}{r!}x^-_{m(a-1)+i,r}\ (\mathrm{mod}\ \hbar_2).$$

Recalling the formulas of Propositions~\ref{limit_m} and~\ref{limit_a} for
the images of $\{e_{i,k}, f_{i,k}, h_{i,k}\}_{i\in [m]}^{k\in \ZZ}$ and
$\{x^\pm_{i',r},\xi_{i',r}\}_{i'\in [mn]}^{r\in \NN}$ under $\theta^{(m)}$ and
$\vartheta^{(mn)}$, respectively, we get the result.
\end{proof}

Combining this result with Theorems~\ref{Upsilon_0}(d),~\ref{flatness},
and the condition $\gcd(\kk,n)=1$, we get:

\begin{cor}\label{injective_Phi}
The homomorphism $\Phi^{\omega_{mn}}_{m,n}$ is injective.
\end{cor}


\subsection{Partial compatibility of $\Phi^{\omega_{mn}}_{m,n}$}\label{section partial proof}
$\ $

Note that Theorem~\ref{main 1} is equivalent to the assignment
$\Phi^{\omega_{mn}}_{m,n}$ given by~(\ref{Phi0}--\ref{Phi3}) to be compatible
with the defining relations (T0--T3,H,T6). In this section, we provide a
straightforward verification of the compatibility with (T0--T3,H) in spirit of~\cite{GTL}.
However, we are not aware of any direct verification of the Serre relations (T6)
(the arguments in~\cite{GTL} heavily rely on the existence of subalgebras
$U_q(L\ssl_2)\subset U_q(L\g)$ for which there are no Serre relations).

\medskip
\noindent
$\bullet$
\emph{Compatibility of $\Phi^{\omega_{mn}}_{m,n}$ and (T0)}.
$\ $

The equality $[\Phi^{\omega_{mn}}_{m,n}(h_{i,k}), \Phi^{\omega_{mn}}_{m,n}(h_{j,l})]=0$
follows from~(\ref{Phi0}--\ref{Phi1}) and the relation~(\ref{Y0}).

\medskip
\noindent
$\bullet$ \emph{Compatibility of $\Phi^{\omega_{mn}}_{m,n}$ and (H)}.
$\ $

If $k=0$, then applying formulas~(\ref{Phi0},\ref{Phi2},\ref{Phi3}) and
the relation~(\ref{Y0}), we get
\begin{equation*}
  [\Phi^{\omega_{mn}}_{m,n}(h_{i,0}), \Phi^{\omega_{mn}}_{m,n}(e_{j,l})]=
  \sum_{j'\equiv j}^{j'\in [mn]} \sum_{i'\equiv i}^{i'\in [mn]}
  a^{(mn)}_{i',j'} \omega_{mn}^{-lj'}e^{ln\sigma^+_{j'}}g_{j'}(\sigma^+_{j'})x^+_{j',0},
\end{equation*}
\begin{equation*}
  [\Phi^{\omega_{mn}}_{m,n}(h_{i,0}), \Phi^{\omega_{mn}}_{m,n}(f_{j,l})]=
  \sum_{j'\equiv j}^{j'\in [mn]} \sum_{i'\equiv i}^{i'\in [mn]}
  (-a^{(mn)}_{i',j'}) \omega_{mn}^{-lj'}e^{ln\sigma^-_{j'}}g_{j'}(\sigma^-_{j'})x^-_{j',0}.
\end{equation*}
It remains to note that $\sum_{i'\equiv i}^{i'\in [mn]} a^{(mn)}_{i',j'}=a^{(m)}_{i,j}$
for any $i,j\in [m],j'\in [mn]$ such that $j'\equiv j$, where the superscripts $(m),(mn)$
are used to distinguish between the two matrices $A$ involved.

To treat the $k\ne 0$ case, we note first that the identity
$B\left(\log\left(1-\frac{\nu}{z}\right)\right)=\frac{1-e^{\nu w}}{w}$ implies:

\begin{lem}\label{B-commutator}
The equalities of Proposition~\ref{yangian_generating}(c) applied to
$\Y^{(mn)}_{\hbar_1,\hbar_2}$ are equivalent to
  $$[B_{i'}(v), x^\pm_{j',s}]=\pm \frac{c_{mn}(i',j';v)}{v}\cdot e^{\sigma_{j'}^\pm v} x^\pm_{j',s}
    \ \mathrm{for\ any}\ i',j'\in[mn],\ s\in\NN,$$
where
  $c_{mn}(i',j';v):=\delta_{j',i'+1}(e^{\frac{\hbar_1v}{mn}}-e^{-\frac{\hbar_3v}{mn}})+
   \delta_{j',i'}(e^{\frac{\hbar_2v}{mn}}-e^{-\frac{\hbar_2v}{mn}})+
   \delta_{j',i'-1}(e^{\frac{\hbar_3v}{mn}}-e^{-\frac{\hbar_1v}{mn}}).$
\end{lem}

Therefore, applying formulas~(\ref{Phi0},\ref{Phi2}) and Lemma~\ref{B-commutator}, we get
\begin{equation*}
  [\Phi^{\omega_{mn}}_{m,n}(h_{i,k}), \Phi^{\omega_{mn}}_{m,n}(e_{j,l})]=
  \sum_{j'\equiv j}^{j'\in [mn]} \sum_{i'\equiv i}^{i'\in [mn]} \frac{\omega_{mn}^{k(j'-i')}c_{mn}(i',j';kn)}{k(\q-\q^{-1})} \omega_{mn}^{-(k+l)j'}e^{(k+l)n\sigma^+_{j'}}g_{j'}(\sigma^+_{j'})x^+_{j',0},
\end{equation*}
\begin{equation*}
  [\Phi^{\omega_{mn}}_{m,n}(h_{i,k}), \Phi^{\omega_{mn}}_{m,n}(f_{j,l})]=
  \sum_{j'\equiv j}^{j'\in [mn]} \sum_{i'\equiv i}^{i'\in [mn]} \frac{-\omega_{mn}^{k(j'-i')}c_{mn}(i',j';kn)}{k(\q-\q^{-1})} \omega_{mn}^{-(k+l)j'}e^{(k+l)n\sigma^-_{j'}}g_{j'}(\sigma^-_{j'})x^-_{j',0}.
\end{equation*}

To complete the verification of compatibility with~(\ref{H}), it remains to prove:

\begin{lem}\label{b vs c}
If $i,j\in [m], j'\in [mn]$ and  $j'\equiv j$, then
 $\sum_{i'\equiv i}^{i'\in [mn]} \frac{\omega_{mn}^{k(j'-i')}c_{mn}(i',j';kn)}{k(\q-\q^{-1})}=b_m(i,j;k)$.
\end{lem}

\begin{proof}[Proof of Lemma~\ref{b vs c}]
$\ $

Follows directly from the identity
  $b_m(i,j;k)=\frac{\delta_{j,i+1}(\q_1^k-\q_3^{-k})+\delta_{j,i}(\q_2^k-\q_2^{-k})+
   \delta_{j,i-1}(\q_3^k-\q_1^{-k})}{k(\q-\q^{-1})}$.
\end{proof}

\noindent
$\bullet$ \emph{Compatibility of $\Phi^{\omega_{mn}}_{m,n}$ and (T2)}.
$\ $

First, let us note that the first equality of Proposition~\ref{yangian_generating}(d)
for $\Y^{(mn)}_{\hbar_1,\hbar_2}$ is equivalent to
\begin{equation}\label{quadratic1}
  A(\sigma^+_{i'},\sigma^+_{j'})
  (p^{(mn)}_{i',j'}(\sigma^+_{i'},\sigma^+_{j'})x^+_{i',0}x^+_{j',0}+
  p^{(mn)}_{j',i'}(\sigma^+_{j'},\sigma^+_{i'})x^+_{j',0}x^+_{i',0})=0,
\end{equation}
\begin{equation}\label{quadratic2}
  \mu\left(B(\sigma^{+,(1)}_{i'},\sigma^{+,(2)}_{i'})p^{(mn)}_{i',i'}
  (\sigma^{+,(1)}_{i'}, \sigma^{+,(2)}_{i'})x^+_{i',0}\otimes x^+_{i',0}\right)=0
\end{equation}
for any $i',j'\in [mn]$ and $A(x,y),B(x,y)\in \CC[[x,y]]$,
such that $i'\ne j',\ B(x,y)=B(y,x)$.

To rewrite $\Phi^{\omega_{mn}}_{m,n}(e_i(z))\Phi^{\omega_{mn}}_{m,n}(e_j(w))$
and $\Phi^{\omega_{mn}}_{m,n}(e_j(w))\Phi^{\omega_{mn}}_{m,n}(e_i(z))$ in the
form with all Cartan terms taken to the left, we will need the following counterpart
of~\cite[Proposition 2.10]{GTL}:

\begin{lem}\label{lambda action}
(a) There are linear operators $\{\lambda^\pm_{i',s}\}_{i'\in [mn]}^{s\in \NN}$
on $\Y^{(mn),0}_{\hbar_1,\hbar_2}$ (cf. Section~\ref{section Yangian addit}) such that
for any $r\in \NN$ and $\xi\in \Y^{(mn),0}_{\hbar_1,\hbar_2}$, we have
$x^\pm_{i',r}\xi=\sum_{s\geq 0} \lambda^\pm_{i',s}(\xi)x^\pm_{i',r+s}$.

\noindent
(b) Let $\lambda^\pm_{i'}(v)\colon \Y^{(mn),0}_{\hbar_1,\hbar_2}\to \Y^{(mn),0}_{\hbar_1,\hbar_2}[v]$
be given by $\lambda^\pm_{i'}(v)(\xi)=\sum_{s\geq 0} \lambda^\pm_{i',s}(\xi)v^s$. Then,
$\lambda^\pm_{i'}(v)$ is an algebra homomorphism.

\noindent
(c) We have $\lambda^\pm_{i'}(u)(B_{j'}(v))=B_{j'}(v)\mp \frac{c_{mn}(j',i';v)}{v}e^{uv}$.
\end{lem}

Using these operators, we obtain:
\begin{multline}\label{computation 1}
  \Phi^{\omega_{mn}}_{m,n}(e_i(z))\Phi^{\omega_{mn}}_{m,n}(e_j(w))=\\
  \sum_{i'\equiv i,j'\equiv j}^{i',j'\in [mn], i'\ne j'}
  \delta\left(\frac{\omega_{mn}^{-i'}e^{n\sigma^+_{i'}}}{z}\right)
  \delta\left(\frac{\omega_{mn}^{-j'}e^{n\sigma^+_{j'}}}{w}\right)
  g_{i'}(\sigma^+_{i'})\lambda^+_{i'}(\sigma^+_{i'})(g_{j'}(\sigma^+_{j'}))x^+_{i',0}x^+_{j',0}+\\
  \sum_{i\equiv i'\equiv j}^{i'\in [mn]}
  \mu\left(\delta\left(\frac{\omega_{mn}^{-i'}e^{n\sigma^{+,(1)}_{i'}}}{z}\right)
  \delta\left(\frac{\omega_{mn}^{-i'}e^{n\sigma^{+,(2)}_{i'}}}{w}\right)
  g_{i'}(\sigma^{+,(1)}_{i'})\lambda^+_{i'}(\sigma^{+,(1)}_{i'})(g_{i'}(\sigma^{+,(2)}_{i'}))x^+_{i',0}\otimes x^+_{i',0}\right),
\end{multline}
\begin{multline}\label{computation 2}
  \Phi^{\omega_{mn}}_{m,n}(e_j(w))\Phi^{\omega_{mn}}_{m,n}(e_i(z))=\\
  \sum_{i'\equiv i,j'\equiv j}^{i',j'\in [mn], i'\ne j'}
  \delta\left(\frac{\omega_{mn}^{-j'}e^{n\sigma^+_{j'}}}{w}\right)
  \delta\left(\frac{\omega_{mn}^{-i'}e^{n\sigma^+_{i'}}}{z}\right)
  g_{j'}(\sigma^+_{j'})\lambda^+_{j'}(\sigma^+_{j'})(g_{i'}(\sigma^+_{i'}))x^+_{j',0}x^+_{i',0}+\\
  \sum_{i\equiv i'\equiv j}^{i'\in [mn]}
  \mu\left(\delta\left(\frac{\omega_{mn}^{-i'}e^{n\sigma^{+,(1)}_{i'}}}{w}\right)
  \delta\left(\frac{\omega_{mn}^{-i'}e^{n\sigma^{+,(2)}_{i'}}}{z}\right)
  g_{i'}(\sigma^{+,(1)}_{i'})\lambda^+_{i'}(\sigma^{+,(1)}_{i'})(g_{i'}(\sigma^{+,(2)}_{i'}))x^+_{i',0}\otimes x^+_{i',0}\right).
\end{multline}

Combining~(\ref{quadratic1},\ref{quadratic2}) with~(\ref{computation 1},\ref{computation 2}),
the compatibility of $\Phi^{\omega_{mn}}_{m,n}$ with (T2) follows from the next result:

\begin{prop}\label{compatibility with T2}
For any $i,j\in [m]$ and $i',j'\in [mn]$ such that $i'\equiv i,j'\equiv j$, we have
\begin{equation*}
  \frac{d^{(m)}_{i,j}g^{(m)}_{i,j}(\omega_{mn}^{-i'}e^{nu},\omega_{mn}^{-j'}e^{nv})}
  {p^{(mn)}_{i',j'}(u,v)} g_{i'}(u)\lambda^+_{i'}(u)(g_{j'}(v))=
  \frac{g^{(m)}_{j,i}(\omega_{mn}^{-j'}e^{nv},\omega_{mn}^{-i'}e^{nu})}
  {p^{(mn)}_{j',i'}(v,u)} g_{j'}(v)\lambda^+_{j'}(v)(g_{i'}(u)).
\end{equation*}
\end{prop}

\begin{proof}[Proof of Proposition~\ref{compatibility with T2}]
$\ $

Due to Lemma~\ref{lambda action}(c), for $a\in [mn]$ such that $a\equiv j'$, we have
\begin{multline}\label{lambda action on g}
  \lambda^\pm_{i'}(u)(\exp(\gamma_{j',a}(v)/2))=F^\pm_{i',j',a}(u,v)^{1/2}\cdot\exp(\gamma_{j',a}(v)/2),\
  \mathrm{where}\  F^\pm_{i',j',a}(u,v):=\\
  \left(\frac{H_{j',a}(v-u+\frac{\hbar_1}{mn})}{H_{j',a}(v-u-\frac{\hbar_3}{mn})}\right)^{\pm\delta_{a,i'+1}}
  \left(\frac{H_{j',a}(v-u+\frac{\hbar_2}{mn})}{H_{j',a}(v-u-\frac{\hbar_2}{mn})}\right)^{\pm\delta_{a,i'}}
  \left(\frac{H_{j',a}(v-u+\frac{\hbar_3}{mn})}{H_{j',a}(v-u-\frac{\hbar_1}{mn})}\right)^{\pm\delta_{a,i'-1}}.
\end{multline}
Recalling the formulas for $g_{i'}(u)$ and $g_{j'}(v)$, we immediately obtain
\begin{equation}\label{lambda action on gg}
  \lambda^+_{i'}(u)(g_{j'}(v))=g_{j'}(v)\prod_{a\equiv j'}^{a\in [mn]} F^+_{i',j',a}(u,v)^{1/2},\
  \lambda^+_{j'}(v)(g_{i'}(u))=g_{i'}(u)\prod_{b\equiv i'}^{b\in [mn]} F^+_{j',i',b}(v,u)^{1/2}.
\end{equation}
On the other hand, we have the equalities
\begin{equation}\label{ratio of p}
  \frac{p^{(mn)}_{i',j'}(u,v)}{p^{(mn)}_{j',i'}(v,u)}=-
  \left(\frac{u-v-\frac{\hbar_1}{mn}}{u-v+\frac{\hbar_3}{mn}}\right)^{\delta_{j',i'+1}}
  \left(\frac{u-v-\frac{\hbar_2}{mn}}{u-v+\frac{\hbar_2}{mn}}\right)^{\delta_{j',i'}}
  \left(\frac{u-v-\frac{\hbar_3}{mn}}{u-v+\frac{\hbar_1}{mn}}\right)^{\delta_{j',i'-1}},
\end{equation}
\begin{equation}\label{ratio of g}
  \frac{d^{(m)}_{i,j}g^{(m)}_{i,j}(z,w)}{g^{(m)}_{j,i}(w,z)}=-\q^{-a^{(m)}_{i,j}}
  \left(\frac{z-\q_1 w}{z-\q_3^{-1}w}\right)^{\delta_{j,i+1}}
  \left(\frac{z-\q_2 w}{z-\q_2^{-1}w}\right)^{\delta_{j,i}}
  \left(\frac{z-\q_3 w}{z-\q_1^{-1}w}\right)^{\delta_{j,i-1}}.
\end{equation}
It remains to combine the above formulas~(\ref{lambda action on g}--\ref{ratio of g}) together.
\end{proof}

\noindent
$\bullet$ \emph{Compatibility of $\Phi^{\omega_{mn}}_{m,n}$ and (T3)}.
$\ $

Analogously to the previous verification, this compatibility follows from the following result:

\begin{prop}\label{compatibility with T3}
For any $i,j\in [m]$ and $i',j'\in [mn]$ such that $i'\equiv i,j'\equiv j$, we have
\begin{equation*}
  \frac{d^{(m)}_{j,i}g^{(m)}_{j,i}(\omega_{mn}^{-j'}e^{nv},\omega_{mn}^{-i'}e^{nu})}
  {p^{(mn)}_{j',i'}(v,u)} g_{i'}(u)\lambda^-_{i'}(u)(g_{j'}(v))=
  \frac{g^{(m)}_{i,j}(\omega_{mn}^{-i'}e^{nu},\omega_{mn}^{-j'}e^{nv})}
  {p^{(mn)}_{i',j'}(u,v)} g_{j'}(v)\lambda^-_{j'}(v)(g_{i'}(u)).
\end{equation*}
\end{prop}

The proof is analogous to that of Proposition~\ref{compatibility with T2}
and is based on the above formulas~(\ref{lambda action on g}--\ref{lambda action on gg}).

\medskip
\noindent
$\bullet$ \emph{Compatibility of $\Phi^{\omega_{mn}}_{m,n}$ and (T1)}.
$\ $

Define $g^{(k)}_{i'}(v):=\sum_{r\geq 0} g^{(k)}_{i',r} v^r\in \widehat{\Y}^{(mn)}_{\hbar_1,\hbar_2}[[v]]$ via
$g^{(k)}_{i'}(v):=e^{knv}g_{i'}(v)$. Then
\begin{equation*}
  \Phi^{\omega_{mn}}_{m,n}(e_{i,k})\Phi^{\omega_{mn}}_{m,n}(f_{j,l})=
  \sum_{i'\equiv i, j'\equiv j}^{i',j'\in [mn]} \omega_{mn}^{-ki'-lj'}
  \sum_{r_1,r_2,s\geq 0}g^{(k)}_{i',r_1}\lambda^+_{i',s}(g^{(l)}_{j',r_2})x^+_{i',r_1+s}x^-_{j',r_2},
\end{equation*}
\begin{equation*}
  \Phi^{\omega_{mn}}_{m,n}(f_{j,l})\Phi^{\omega_{mn}}_{m,n}(e_{i,k})=
  \sum_{i'\equiv i, j'\equiv j}^{i',j'\in [mn]} \omega_{mn}^{-ki'-lj'}
  \sum_{r_1,r_2,s\geq 0}g^{(l)}_{j',r_1}\lambda^-_{j',s}(g^{(k)}_{i',r_2}) x^-_{j',r_1+s}x^+_{i',r_2}.
\end{equation*}
Combining this with
  $x^+_{i',r_1+s}x^-_{j',r_2}=x^-_{j',r_2}x^+_{i',r_1+s}+\delta_{i',j'}\xi_{i',r_1+r_2+s}$,
we see that compatibility of $\Phi^{\omega_{mn}}_{m,n}$ with (T1) follows from the following
result (compare to~\cite[Lemma 3.5]{GTL}):

\begin{prop}\label{compatibility with T1}
(a) For any $i',j'\in [mn]$, we have
\begin{equation*}
  g_{i'}(u)\lambda^+_{i'}(u)(g_{j'}(v))=g_{j'}(v)\lambda^-_{j'}(v)(g_{i'}(u)).
\end{equation*}

\noindent
(b) For any $i\in [m], N\in \ZZ$, we have
\begin{equation*}
  \sum_{i'\equiv i}^{i'\in [mn]} \omega_{mn}^{-Ni'} \left\{e^{Nnu}g_{i'}(u)\lambda^+_{i'}(u)(g_{i'}(u))\right\}_{\mid{u^r\mapsto \xi_{i',r}}}=
  \Phi^{\omega_{mn}}_{m,n}\left(\frac{\psi^+_{i,N}-\psi^-_{i,N}}{\q-\q^{-1}}\right).
\end{equation*}
\end{prop}

\begin{proof}[Proof of Proposition~\ref{compatibility with T1}]
$\ $

Part (a) is due to the formulas~(\ref{lambda action on g}--\ref{lambda action on gg})
and the equality
  $\prod_{a\equiv j'} F^+_{i',j',a}(u,v)=\prod_{b\equiv i'} F^-_{j',i',b}(v,u)$.

Consider a homomorphism
  $\Phi^{\omega_{mn},0}_{m,n}\colon \U^{(m),\omega_{mn},0}_{\hbar_1,\hbar_2}
   \to \widehat{\Y}^{(mn),0}_{\hbar_1,\hbar_2}$
defined by~(\ref{Phi0},\ref{Phi1}). Our proof of part (b) is based on the following
result (compare to~\cite[Proposition 4.2]{GTL}).

\begin{prop}\label{cartan homom}
For any $i\in [m], N\in \ZZ$, we have
  $$\Phi^{\omega_{mn},0}_{m,n}\left(\frac{\psi^+_{i,N}-\psi^-_{i,N}}{\q-\q^{-1}}\right)=
    \sum_{i'\equiv i} Q^{(N)}_{i'}(u)_{\mid{u^r\mapsto \xi_{i',r}}},\
    Q^{(N)}_{i'}(u):=\frac{\hbar_2\omega_{mn}^{-Ni'}e^{Nnu}}{m(\q-\q^{-1})}\prod_{j'\equiv i'}\exp(\gamma_{i',j'}(u)).$$
\end{prop}

Combining Proposition~\ref{cartan homom} with
$Q^{(N)}_{i'}(u)=\omega_{mn}^{-Ni'}e^{Nnu}\cdot g_{i'}(u)^2$ and the equality
  $$\lambda^+_{i'}(u)(g_{i'}(u))=g_{i'}(u)\prod_{a\equiv i'}^{a\in [mn]} F^+_{i',i',a}(u,u)=g_{i'}(u)$$
completes our proof of Proposition~\ref{compatibility with T1}(b).
\end{proof}

For completeness of our exposition, we conclude this section with
a proof of Proposition~\ref{cartan homom}.

\begin{proof}[Proof of Proposition~\ref{cartan homom}]
$\ $

Fix $\bar{s}=(s_0,\ldots,s_{mn-1})\in \NN^{mn}$ and set
$S_{\bar{s}}:=\prod_{i'=0}^{mn-1} S_{s_{i'}}$. Consider the rings
  $$R(\bar{s}):=(\CC[[\hbar_1,\hbar_2]][\{a^{(i')}_k\}]_{i'\in [mn]}^{1\leq k\leq s_{i'}})^{S_{\bar{s}}}\ \mathrm{and}\
    S(\bar{s}):=(\CC[[\hbar_1,\hbar_2]][\{(A^{(i')}_k)^{\pm 1}\}]_{i'\in [mn]}^{1\leq k\leq s_{i'}})^{S_{\bar{s}}}.$$
Define homomorphisms
  $\CCD^Y\colon \Y^{(mn),0}_{\hbar_1,\hbar_2}\to R(\bar{s})$
and
  $\CCD^U\colon \U^{(m),\omega_{mn},0}_{\hbar_1,\hbar_2}\to S(\bar{s})$
via
  $$\CCD^Y(\xi_{i'}(u))=\prod_{k=1}^{s_{i'}} \left(\frac{u+\frac{\hbar_2}{mn}-a^{(i')}_k}{u-a^{(i')}_k}\right)^+,\
   \CCD^U(\psi^\pm_i(z))=\prod_{i'\equiv i}^{i'\in [mn]} \prod_{k=1}^{s_{i'}}
   \left(\frac{\q z-\q^{-1}A^{(i')}_k}{z-A^{(i')}_k}\right)^\pm.$$

The following is straightforward (cf. the proof of~\cite[Proposition 4.4]{GTL}):

\begin{lem}\label{explicit D-homom}
(a) For any $i'\in [mn], r\in \NN$, we have
\begin{equation*}
  \CCD^Y(\xi_{i',r})=\sum_{k=1}^{s_{i'}}(a^{(i')}_k)^r \prod_{1\leq k'\leq s_{i'}}^{k'\ne k}
  \frac{a^{(i')}_k-a^{(i')}_{k'}+\frac{\hbar_2}{mn}}{a^{(i')}_k-a^{(i')}_{k'}},\
  \CCD^Y(B_{i'}(v))=\frac{1-e^{-\frac{\hbar_2}{mn}v}}{v}\sum_{k=1}^{s_{i'}} e^{a^{(i')}_kv}.
\end{equation*}

\noindent
(b) For any $i\in [m], r\in \ZZ_{>0}$, we have
\begin{equation*}
\begin{split}
  & \CCD^U(\psi_{i,\pm r})=\pm (\q-\q^{-1})\sum_{i'\equiv i} \sum_{k=1}^{s_{i'}}(A^{(i')}_k)^{\pm r}
  \prod_{j'\equiv i,1\leq k'\leq s_{j'}}^{(j',k')\ne (i',k)}\frac{\q A^{(i')}_{k}-\q^{-1} A^{(j')}_{k'}}{A^{(i')}_k-A^{(j')}_{k'}},\\
  & \CCD^U(h_{i,0})=\sum_{i'\equiv i} s_{i'},\
  \CCD^U(h_{i,\pm r})=\frac{1-\q^{\mp 2r}}{\pm r(\q-\q^{-1})}\sum_{i'\equiv i}\sum_{k=1}^{s_{i'}} (A^{(i')}_k)^{\pm r}.
\end{split}
\end{equation*}
\end{lem}

Let $\widehat{R}(\bar{s})$ be the completion of $R(\bar{s})$ with respect
to the $\NN$-grading defined by $\deg(\hbar_s)=\deg(a^{(i')}_k)=1$. As $\CCD^Y$
preserves the grading, it extends to a homomorphism
  $\widehat{\Y}^{(mn),0}_{\hbar_1,\hbar_2}\to \widehat{R}(\bar{s})$
also denoted by $\CCD^Y$. Consider the homomorphism
  $\ch\colon S(\bar{s})\to \widehat{R}(\bar{s})$
defined by $A^{(i')}_k\mapsto \omega^{-i'}_{mn}e^{na^{(i')}_k}$.

Our proof of Proposition~\ref{cartan homom} is crucially based on the following result:

\begin{lem}\label{compatibility cartan}
(a) We have $\ch\circ \CCD^U=\CCD^Y\circ \Phi^{\omega_{mn},0}_{m,n}$.

\noindent
(b) For $N\in \ZZ$, we have
  $\ch\circ \CCD^U \left(\frac{\psi^+_{i,N}-\psi^-_{i,N}}{\q-\q^{-1}}\right)=
   \CCD^Y(\sum_{i'\equiv i} Q^{(N)}_{i'}(u)_{\mid{u^r\mapsto \xi_{i',r}}}).$
\end{lem}

\begin{proof}[Proof of Lemma~\ref{compatibility cartan}]
$\ $

Part (a) follows by comparing the images of $h_{i,k}$ via Lemma~\ref{explicit D-homom}.

Let us now verify part (b). Using Lemma~\ref{explicit D-homom}(b),
the left-hand side can be written as
\begin{equation*}
  \ch\circ \CCD^U \left(\frac{\psi^+_{i,N}-\psi^-_{i,N}}{\q-\q^{-1}}\right)=
  \sum_{i'\equiv i} \sum_{k=1}^{s_{i'}} \omega_{mn}^{-Ni'}e^{Nna^{(i')}_k}
  \prod_{j'\equiv i, 1\leq k'\leq s_{j'}}^{(j',k')\ne (i',k)}
  \frac{\q\omega_{mn}^{-i'}e^{na^{(i')}_k}-\q^{-1}\omega_{mn}^{-j'}e^{na^{(j')}_{k'}}}
  {\omega_{mn}^{-i'}e^{na^{(i')}_k}-\omega_{mn}^{-j'}e^{na^{(j')}_{k'}}}.
\end{equation*}
On the other hand, due to Lemma~\ref{explicit D-homom}(a), we also have
\begin{equation*}
   \CCD^Y\left(\sum_{i'\equiv i} Q^{(N)}_{i'}(u)_{\mid{u^r\mapsto \xi_{i',r}}}\right)=
   \sum_{i'\equiv i}\sum_{k=1}^{s_{i'}} \CCD^Y(Q^{(N)}_{i'}(u))_{\mid{u\mapsto a^{(i')}_k}}
   \prod_{1\leq k'\leq s_{i'}}^{k'\ne k}\frac{a^{(i')}_k-a^{(i')}_{k'}+\frac{\hbar_2}{mn}}{a^{(i')}_k-a^{(i')}_{k'}}.
\end{equation*}

To evaluate $\CCD^Y(Q^{(N)}_{i'}(u))$, we note that the second equality
of Lemma~\ref{explicit D-homom}(a) implies
  $\CCD^Y(\gamma_{i',j'}(u))=\sum_{k'=1}^{s_{j'}}
   \left(G_{i',j'}(u-a^{(j')}_{k'})-G_{i',j'}(u-a^{(j')}_{k'}+\frac{\hbar_2}{mn})\right)$.
The result follows.
\end{proof}

Lemma~\ref{compatibility cartan} and an injectivity of
  $\oplus \CCD^Y\colon \widehat{\Y}^{(mn),0}_{\hbar_1,\hbar_2}\to
   \oplus_{\bar{s}\in \NN^{mn}}\widehat{R}(\bar{s})$
imply Proposition~\ref{cartan homom}.
\end{proof}


\section{Compatible isomorphisms of representations}\label{section proof via reps}

In this section, we construct isomorphisms of representations compatible with
$\Phi^{\omega_{mn}}_{m,n}$. Combining this with Corollary~\ref{faithfulness} yields
a short proof of Theorem~\ref{main 1}.


\subsection{Isomorphisms $I_{m,n;\omega_{mn}}^{r;\pp,\vv}$}\label{I intertwiner}
$\ $

Given $m,n\geq 1$ and $\omega_{mn}=\exp(2\pi \kk \mathbf{i}/mn)$ with $\gcd(\kk,n)=1$,
we consider the two algebras of interest:
  $\U^{(m),\omega_{mn}}_{\hbar_1,\hbar_2}$ and $\Y^{(mn)}_{\hbar_1,\hbar_2}$.
To proceed further, choose $r\geq 1$ and the following $r$-tuples:
  $$\pp=(p_1,\ldots,p_r) \in [mn]^r,\ \vv=(v_1,\ldots,v_r)\in ((\hbar_1,\hbar_2)\CC[[\hbar_1,\hbar_2]])^r,$$
  $$\pp'=(p'_1,\ldots,p'_r) \in [m]^r,\ \uu'=(u'_1,\ldots,u'_r)\in (\CC[[\hbar_1,\hbar_2]]^\times)^r.$$
Associated to this data, we have a collection of Fock
$\U^{(m),\omega_{mn}}_{R}$-representations $\{F^{p'_k}_R(u'_k)\}_{k=1}^r$ and
Fock $\Y^{(mn)}_{R}$-representations $\{{^{a}F}^{p_k}_R(v_k)\}_{k=1}^r$.
Following Section~\ref{section tensor Fock}, we consider
  $$F^{\pp'}_R(\uu'):=F^{p'_1}_R(u'_1)\otimes F^{p'_2}_R(u'_2)\otimes \cdots \otimes F^{p'_r}_R(u'_r) -
    \mathrm{a\ representation\ of}\ \U^{(m),\omega_{mn}}_{R},$$
  $${^{a}F}^{\pp}_R(\vv):={^{a}F}^{p_1}_R(v_1)\otimes {^{a}F}^{p_2}_R(v_2)\otimes \cdots \otimes {^{a}F}^{p_r}_R(v_r) -
    \mathrm{a\ representation\ of}\ \Y^{(mn)}_{R},$$
whenever these representations are well-defined, i.e.,
$\{u'_k\}_{k=1}^r$ and $\{v_k\}_{k=1}^r$ are \emph{not in resonance}.
Both of these tensor products have natural bases $\{|\boldsymbol{\lambda}\rangle\}$
labeled by $r$-tuples of partitions
   $$\boldsymbol{\lambda}=(\lambda^{(1)},\ldots,\lambda^{(r)})\
     \mathrm{with}\ \lambda^{(k)} - \mathrm{a\ partition}\ (1\leq k\leq r).$$
The action of the generators $\{h_{i,k}, e_{i,k}, f_{i,k}\}_{i\in [m]}^{k\in \ZZ}$
and $\{\xi_{i',r}, x^\pm_{i',r}\}_{i'\in [mn]}^{r\in \NN}$ in these bases is given
by the explicit formulas of Propositions~\ref{Fock_mr} and~\ref{Fock_ar} whereas
$\{h_s\}$ and $\{q_s\}$ are replaced by
  $$h_s\rightsquigarrow \frac{\hbar_s}{mn},\
    q_1\rightsquigarrow \q_1:=\omega_{mn}e^{\frac{\hbar_1}{m}},\
    q\rightsquigarrow \q:=e^{\frac{\hbar_2}{2m}},\
    q_2\rightsquigarrow \q_2:=\q^2,\
    q_3\rightsquigarrow \q_3:=\omega_{mn}^{-1}e^{\frac{\hbar_3}{m}}.$$

Our next result establishes an isomorphism of these tensor products,
compatible with $\Phi^{\omega_{mn}}_{m,n}$.

\begin{thm}\label{intertwiner 1}
For any $r,\pp,\vv$ as above, define $p'_k,u'_k$ via
$p'_k:=p_k\ \mathrm{mod}\ m$ and $u'_k:=\omega_{mn}^{-p_k}e^{nv_k}$.
There exists a unique collection of constants
$c_{\boldsymbol{\lambda}}(m,n;\omega_{mn})\in R$ such that
$c_{\boldsymbol{\emptyset}}(m,n;\omega_{mn})=1$ and the corresponding
$R$-linear isomorphism of vector spaces
  $$I_{m,n;\omega_{mn}}^{r;\pp,\vv}\colon F^{\pp'}_R(\uu')\iso  {^{a}F}^{\pp}_R(\vv)\ \mathrm{given\ by}\
    |\boldsymbol{\lambda}\rangle \mapsto c_{\boldsymbol{\lambda}}(m,n;\omega_{mn})\cdot |\boldsymbol{\lambda}\rangle$$
satisfies the property
\begin{equation}\label{dag}
  I_{m,n;\omega_{mn}}^{r;\pp,\vv}(X(w))=
  \Phi^{\omega_{mn}}_{m,n}(X)(I_{m,n;\omega_{mn}}^{r;\pp,\vv}(w))
  \ \forall\ w\in F^{\pp'}_R(\uu'), X\in \{h_{i,k}, e_{i,k}, f_{i,k}\}_{i\in [m]}^{k\in \ZZ}.
\end{equation}
\end{thm}

We say that $I_{m,n;\omega_{mn}}^{r;\pp,\vv}$ is \emph{compatible} with
$\Phi^{\omega_{mn}}_{m,n}$ if~(\ref{dag}) holds.

\begin{proof}[Proof of Theorem~\ref{intertwiner 1}]
$\ $

First, we claim that~(\ref{dag}) holds for any $w=|\boldsymbol{\lambda}\rangle,\ X=h_{i,k}$,
and an arbitrary choice  of $c_{\boldsymbol{\lambda}}(m,n;\omega_{mn})$.
This follows from the following result:

\begin{lem}\label{Compatibility with Cartan}
We have $\langle \bol| h_{i,k}|\bol\rangle=\langle \bol| \Phi(h_{i,k})|\bol\rangle$
for any $i\in [m],k\in \ZZ$, and $\bol$, where $\Phi(h_{i,k})$ is defined by~(\ref{Phi0}--\ref{Phi1}).
\end{lem}

\begin{proof}[Proof of Lemma~\ref{Compatibility with Cartan}]
\

Define $\chi^{(a)}_s:=\q_1^{\lambda^{(a)}_s}\q_3^{s-1}u'_a$ and
$x^{(a)}_s:=\lambda^{(a)}_s\frac{\hbar_1}{mn}+(s-1)\frac{\hbar_3}{mn}+v_a$.

\noindent
$\bullet$
For $k=0$, we have
  $$\langle \bol| h_{i,0}|\bol\rangle=
    \#\{(a,s)|c^{(a)}_s(\bol)\underset{m}\equiv i\}-\#\{(a,s)|c^{(a)}_s(\bol)\underset{m}\equiv i+1\},$$
  $$\langle \bol| \xi_{i',0}|\bol\rangle=
    \#\{(a,s)|c^{(a)}_s(\bol)\underset{mn}\equiv i'\}-\#\{(a,s)|c^{(a)}_s(\bol)\underset{mn}\equiv i'+1\}.$$
Hence, the equality
  $\langle \bol| h_{i,0}|\bol\rangle=\langle \bol| \sum^{i'\in [mn]}_{i'\equiv i} \xi_{i',0}|\bol\rangle$.

\noindent
$\bullet$
For $k\ne 0$, we have
\begin{equation}\label{h-eigenvalue}
  \langle \bol| h_{i,k}|\bol\rangle=
  \sum_{(a,s)}^{c^{(a)}_s(\bol)\underset{m}\equiv i}\frac{\q_1^{-k}-\q_3^k}{k(\q-\q^{-1})}\left(\chi^{(a)}_s\right)^k+
  \sum_{(a,s)}^{c^{(a)}_s(\bol)\underset{m}\equiv i+1}\frac{1-\q_2^k}{k(\q-\q^{-1})}\left(\chi^{(a)}_s\right)^k.
\end{equation}
Meanwhile, using the equality
  $B\left(\log\left(1-\frac{\nu}{z}\right)\right)=\frac{1-e^{\nu w}}{w}$,
we also get
\begin{equation}\label{B-eigenvalue}
  \langle \bol| B_{i'}(w)|\bol\rangle=
  \sum_{(a,s)}^{c^{(a)}_s(\bol)\underset{mn}\equiv i'} \frac{e^{w(x^{(a)}_s-\frac{\hbar_1}{mn})}-e^{w(x^{(a)}_s+\frac{\hbar_3}{mn})}}{w}+
  \sum_{(a,s)}^{c^{(a)}_s(\bol)\underset{mn}\equiv i'+1}\frac{e^{wx^{(a)}_s}-e^{w(x^{(a)}_s+\frac{\hbar_2}{mn})}}{w}.
\end{equation}
Recalling the explicit formulas for $\q_s$ and $u'_k$, the above
formulas~(\ref{h-eigenvalue}--\ref{B-eigenvalue}) immediately imply
the claimed equality
  $\langle \bol| h_{i,k}|\bol\rangle=
   \left\langle \bol| \frac{n}{\q-\q^{-1}}\sum^{i'\in [mn]}_{i'\equiv i} \omega_{mn}^{-ki'}B_{i'}(kn)|\bol\right\rangle$.
\end{proof}

Next, we will see under which conditions~(\ref{dag}) holds for all
$w=|\boldsymbol{\lambda}\rangle$ and $X=e_{i,k}\ \mathrm{or}\ f_{i,k}$.
To state the result, we introduce the following constants:
\begin{equation}\label{compatibility ratio 1}
\begin{split}
  & d_{\boldsymbol{\lambda},1^{(b)}_l}(m,n;\omega_{mn}):=
  (\q(1-\q_3)/(1-\q_1^{-1}))^{\delta_{m,1}/2}\cdot (-\hbar_1/\hbar_3)^{\delta_{mn,1}/2}\times\\
  & \prod_{(a,s)\ne (b,l)}^{c^{(a)}_s(\bol)\underset{m}\equiv c^{(b)}_l(\bol)-1}
       \psi\left(\q^{-1}_1\chi^{(a)}_s/\chi^{(b)}_l\right)^{\epsilon^{(a,s)}_{(b,l)}}
  \cdot
  \prod_{(a,s)\ne (b,l)}^{c^{(a)}_s(\bol)\underset{m}\equiv c^{(b)}_l(\bol)}
       \psi\left(\chi^{(b)}_l/\chi^{(a)}_s\right)^{\epsilon^{(a,s)}_{(b,l)}}
  \times\\
  & \prod_{(a,s)\ne (b,l)}^{c^{(a)}_s(\bol)\underset{mn}\equiv c^{(b)}_l(\bol)-1}
       \left(\frac{x^{(a)}_s-x^{(b)}_l+\frac{\hbar_3}{mn}}{x^{(a)}_s-x^{(b)}_l-\frac{\hbar_1}{mn}}\right)^{-\epsilon^{(a,s)}_{(b,l)}}
  \cdot
  \prod_{(a,s)\ne (b,l)}^{c^{(a)}_s(\bol)\underset{mn}\equiv c^{(b)}_l(\bol)}
       \left(\frac{x^{(b)}_l-x^{(a)}_s-\frac{\hbar_2}{mn}}{x^{(b)}_l-x^{(a)}_s}\right)^{-\epsilon^{(a,s)}_{(b,l)}},
\end{split}
\end{equation}
where we set
  $\epsilon^{(a,s)}_{(b,l)}:=
    \begin{cases}
      1/2 & \text{if}\ \ (a,s)\succ (b,l),\\
      -1/2 & \text{if}\ \ (a,s)\prec (b,l).
    \end{cases}$

\begin{lem}\label{Compatibility with e-f}
Both equalities
  $$I_{m,n;\omega_{mn}}^{r;\pp,\vv}(e_{i,k}(|\boldsymbol{\lambda}\rangle))=
    \Phi^{\omega_{mn}}_{m,n}(e_{i,k})(I_{m,n;\omega_{mn}}^{r;\pp,\vv}(|\boldsymbol{\lambda}\rangle))\
    \mathrm{for\ all}\ \boldsymbol{\lambda}, i\in [m], k\in \ZZ$$
and
  $$I_{m,n;\omega_{mn}}^{r;\pp,\vv}(f_{i,k}(|\boldsymbol{\lambda}\rangle))=
    \Phi^{\omega_{mn}}_{m,n}(f_{i,k})(I_{m,n;\omega_{mn}}^{r;\pp,\vv}(|\boldsymbol{\lambda}\rangle))\
    \mathrm{for\ all}\ \boldsymbol{\lambda}, i\in [m], k\in \ZZ$$
with $\Phi^{\omega_{mn}}_{m,n}(e_{i,k}), \Phi^{\omega_{mn}}_{m,n}(f_{i,k})$
defined by~(\ref{Phi2}--\ref{Phi3}) are equivalent to
\begin{equation}\label{ratio 1}
  \frac{c_{\boldsymbol{\lambda}+1^{(b)}_l}(m,n;\omega_{mn})}{c_{\boldsymbol{\lambda}}(m,n;\omega_{mn})}=
  d_{\boldsymbol{\lambda},1^{(b)}_l}(m,n;\omega_{mn}).
\end{equation}
\end{lem}

\begin{proof}[Proof of Lemma~\ref{Compatibility with e-f}]
\

This is a straightforward verification. The matrix coefficients of $e_{i,k}$ and $f_{i,k}$
are given by Proposition~\ref{Fock_mr}. To compute the matrix coefficients of
$\Phi^{\omega_{mn}}_{m,n}(e_{i,k})$ and $\Phi^{\omega_{mn}}_{m,n}(f_{i,k})$, one needs
to combine the formulas of Proposition~\ref{Fock_ar} with the identity~(\ref{B-eigenvalue})
and the general formula $e^{\nu\partial_v}G(v)=G(v+\nu)$.
The details are left to the interested reader.
\end{proof}


The uniqueness of $c_{\boldsymbol{\lambda}}(m,n;\omega_{mn})\in R$
satisfying the relation (\ref{ratio 1}) with the initial condition
$c_{\boldsymbol{\emptyset}}(m,n;\omega_{mn})=1$ is obvious.
The existence of such $c_{\boldsymbol{\lambda}}(m,n;\omega_{mn})$ is equivalent to
  $$d_{\boldsymbol{\lambda}+1^{(b)}_l,1^{(a)}_s}(m,n;\omega_{mn})\cdot
    d_{\boldsymbol{\lambda}, 1^{(b)}_l}(m,n;\omega_{mn})=
    d_{\boldsymbol{\lambda}+1^{(a)}_s,1^{(b)}_l}(m,n;\omega_{mn})\cdot
    d_{\boldsymbol{\lambda}, 1^{(a)}_s}(m,n;\omega_{mn})$$
for all possible $\boldsymbol{\lambda}, 1^{(b)}_l, 1^{(a)}_s$.
The verification of this identity is straightforward.
\end{proof}


\subsection{First proof of Theorem~\ref{main 1}}
$\ $


Recall the faithful $\Y^{(mn)}_R$-representation
  ${^{a}\bf{F}}_R:=\bigoplus_r \bigoplus_{\vv}^{\pp} {^{a}}F^{\pp}_R(\vv)$
from Corollary~\ref{faithfulness}(b). Let
  ${\bf{F}}^{0}_R\subset {\bf{F}}_R:=\bigoplus_r \bigoplus_{\uu'}^{\pp'} F^{\pp'}_R(\uu')$
be the subspace corresponding to $u'_k,p'_k$ as in Theorem~\ref{intertwiner 1}.
According to Theorem~\ref{intertwiner 1}, we have an $R$-linear isomorphism
${\bf{I}}\colon {\bf{F}}^{0}_R\iso {^{a}\bf{F}}_R$ compatible with $\Phi^{\omega_{mn}}_{m,n}$
in the following sense:
  $${\bf{I}}(X(w))=\Phi^{\omega_{m,n}}_{m,n}(X)({\bf{I}}(w))\
    \mathrm{for\ any}\ w\in {\bf{F}}^0_R,\ X\in \{h_{i,k},e_{i,k},f_{i,k}\}_{i\in [m]}^{k\in \ZZ}.$$
For any $X\in \{h_{i,k},e_{i,k},f_{i,k}\}_{i\in [m]}^{k\in \ZZ}$, consider the assignment
$X\mapsto \Phi^{\omega_{mn}}_{m,n}(X)$ defined by~(\ref{Phi0}--\ref{Phi3}).
As mentioned in Section~\ref{section partial proof}, Theorem~\ref{main 1} is equivalent to
this assignment being compatible with
all the defining relations of $\U^{(m),\omega_{mn}}_R$.
The latter follows immediately from the faithfulness of ${^{a}\bf{F}}_R$
combined with an existence of the compatible isomorphism $\bf{I}$.


\subsection{Geometric interpretation}\label{section geometric interpretation}
$\ $

The goal of this section is to provide geometric realization for

\noindent
$\bullet$
the Fock modules $F^p(u)$ and ${{^a}F}^p(v)$ of Section~\ref{section Fock},

\noindent
$\bullet$
the tensor products of Fock modules $F^{\pp}(\uu)$ and ${{^a}F}^{\pp}(\vv)$
of Section~\ref{section tensor Fock},

\noindent
$\bullet$
the intertwining isomorphisms $I^{r;\pp,\vv}_{m,n;\omega_{mn}}$ of
Section~\ref{I intertwiner}.

Given a quiver $Q$ and dimension vectors $\vv,\ww\in \NN^{\mathrm{vert}(Q)}$
($\mathrm{vert}(Q)$ is the set of vertices of $Q$), one can define the associated
Nakajima quiver variety $\M^Q(\vv,\ww)$. These varieties play a crucial role in
the geometric representation theory of quantum and Yangian algebras.
For the purposes of our paper, we will be interested only in the following set
of quivers $Q$ (labeled by $n\in \ZZ_{>0}$):

\noindent
$\bullet$
$Q_1$ is the Jordan quiver with one vertex ($\mathrm{vert}(Q)=[1]$) and one loop,

\noindent
$\bullet$
$Q_n$ (with $n>1$) is the cyclic quiver with $\mathrm{vert}(Q)=[n]$.

For any $Q_n$ as above and $\vv,\ww\in \NN^{[n]}$, consider $[n]$-graded
vectors spaces $V=\bigoplus_{i\in [n]} V_i$ and $W=\bigoplus_{i\in [n]} W_i$ such that
$\dim(V_i)=v_i$ and $\dim(W_i)=w_i$. Define
  $$M(\vv,\ww):=\bigoplus_{i\in [n]}\mathrm{Hom}(V_i, V_{i+1})\oplus
    \bigoplus_{i\in [n]}\mathrm{Hom}(V_i, V_{i-1})\oplus
    \bigoplus_{i\in [n]}\mathrm{Hom}(W_i, V_i)\oplus
    \bigoplus_{i\in [n]}\mathrm{Hom}(V_i, W_i).$$
Elements of $M(\vv,\ww)$ can be written as
  $\left(\mathbf{B}=\{B_i\}, \mathbf{\bar{B}}=\{\bar{B}_i\}, \mathbf{a}=\{a_i\}, \mathbf{b}=\{b_i\}\right)_{i\in [n]}$.
Consider the moment map
  $\mu\colon M(\vv,\ww)\to \bigoplus_{i\in [n]} \mathrm{End}(V_i)$
defined by
  $$\mu(\mathbf{B},\mathbf{\bar{B}},\mathbf{a},\mathbf{b})=
    \sum_{i\in [n]} (B_{i-1}\bar{B}_i-\bar{B}_{i+1}B_i+a_ib_i).$$
A point $(\mathbf{B},\mathbf{\bar{B}},\mathbf{a},\mathbf{b})\in \mu^{-1}(0)$
is said to be \emph{stable} if there is no non-zero $(\mathbf{B},\mathbf{\bar{B}})$-invariant
subspace of $V$ contained in $\mathrm{Ker}(\mathbf{b})$. Let us denote by $\mu^{-1}(0)^s$ the
set of stable points. An important property of
$\mu^{-1}(0)^s$ is that the group $G_\vv=\prod_{i\in [n]} \mathrm{GL}(V_i)$ acts freely
on $\mu^{-1}(0)^s$. The Nakajima quiver variety $\M(\vv,\ww)$
is defined as a geometric quotient
  $$\M(\vv,\ww)=\M^{Q_n}(\vv,\ww)=\mu^{-1}(0)^s/G_\vv.$$

There is a natural action of the torus
  $\TT_{\ww}:=\CC^\times \times \CC^\times \times \prod_{i\in [n]} (\CC^\times)^{w_i}$
on $\M(\vv,\ww)$ for any $\vv$. Moreover, it is known that the set of
$\TT_{\ww}$-fixed points is parametrized by the tuples of Young diagrams
  $\boldsymbol{\lambda}=\{\lambda^{(i,k)}\}_{i\in [n]}^{1\leq k\leq w_i}$
satisfying the following requirement. For any $i,k$ as above, let us color
the boxes of $\lambda^{(i,k)}$ into $n$ colors $[n]$, so that the box staying
in the $a$-th row and $b$-th column has color $i+a-b$. Our requirement is that
the total number of color $\iota$ boxes equals $v_\iota$ for every $\iota\in [n]$.

For $\ww\in \NN^{[n]}$, consider the direct sum of equivariant cohomology
$H(\ww)=\bigoplus_{\vv} H^\bullet_{\TT_{\ww}}(\M(\vv,\ww))$. It is a module over
  $H^\bullet_{\TT_{\ww}}(\mathrm{pt})=
   \CC[\mathfrak{t}_{\ww}]=\CC\left[s_1,s_2,\{x_{i,k}\}_{i\in [n]}^{1\leq k\leq w_i}\right]$,
where $\mathfrak{t}_\ww:=\mathrm{Lie}(\TT_\ww)$. Define
  $H(\ww)_{\mathrm{loc}}:=
   H(\ww)\otimes_{H^\bullet_{\TT_\ww}(\mathrm{pt})} \mathrm{Frac}(H^\bullet_{\TT_\ww}(\mathrm{pt}))$.
Let $[\boldsymbol{\lambda}]$ be the direct image of the fundamental cycle
of the $\TT_\ww$-fixed point, corresponding to $\boldsymbol{\lambda}$. The set
$\{[\boldsymbol{\lambda}]\}$ forms a basis of $H(\ww)_{\mathrm{loc}}$.

Let us consider an analogous direct sum of equivariant K-groups
$K(\ww)=\bigoplus_{\vv} K^{\TT_{\ww}}(\M(\vv,\ww))$. It is a module over
  $K^{\TT_{\ww}}(\mathrm{pt})=\CC[\TT_\ww]=
   \CC\left[t^{\pm 1}_1, t^{\pm 1}_2,\{\chi^{\pm 1}_{i,k}\}_{i\in [n]}^{1\leq k\leq w_i}\right]$.
Define the localized version
  $K(\ww)_{\mathrm{loc}}:=
   K(\ww)\otimes_{K^{\TT_\ww}(\mathrm{pt})} \mathrm{Frac}(K^{\TT_\ww}(\mathrm{pt}))$.
Let $[\boldsymbol{\lambda}]$ be the direct image of the structure sheaf
of the $\TT_\ww$-fixed point, corresponding to $\boldsymbol{\lambda}$.
The set $\{[\boldsymbol{\lambda}]\}$ forms a basis of $K(\ww)_{\mathrm{loc}}$.

The following result goes back to~\cite{Na,V} for $n>1$ and~\cite{SV1,SV2,T} for $n=1$ (cf.~\cite{K}):
\begin{thm}
(a) For any $\ww\in \NN^{[n]}$, there is a natural action of $\Y^{(n)}_{s_1,-s_1-s_2,s_2}$
on $H(\ww)_{\mathrm{loc}}$.

\noindent
(b) For any $\ww\in \NN^{[n]}$, there is a natural action of
$\U^{(n)}_{t_1,t^{-1}_1t^{-1}_2,t_2}$ on $K(\ww)_{\mathrm{loc}}$.
\end{thm}

In what follows, we set
$h_1=s_1,\ h_2=-s_1-s_2,\ h_3=s_2$ and $q_1=t_1,\ q_2=t^{-1}_1t^{-1}_2,\ q_3=t_2$.

\begin{prop}\label{geometry Fock}
For $p\in [n]$, define $\ww^{(p)}=(0,\ldots,1,\ldots,0)\in \NN^{[n]}$
with $1$ at the $p$-th place.

\noindent
(a) There is an isomorphism of $\U^{(n)}_{q_1,q_2,q_3}$-representations
 $\alpha\colon F^p(\chi_{p,1})\iso K(\ww^{(p)})_{\mathrm{loc}}$.

\noindent
(b) There is an isomorphism of $\Y^{(n)}_{h_1,h_2,h_3}$-representations
 ${^a}\alpha\colon {{^a}F}^p(x_{p,1})\iso H(\ww^{(p)})_{\mathrm{loc}}$.

\noindent
(c) Both isomorphisms $\alpha$ and ${^a}\alpha$ are given by diagonal matrices in
the bases $\{|\lambda\rangle\}$ and $\{[\lambda]\}$.
\end{prop}

\begin{proof}[Proof of Proposition~\ref{geometry Fock}]
$\ $

The $n=1$ case of this result was treated in~\cite[Section 4]{T},
while the general case can be deduced from the former
by the standard procedure of ``taking a $\ZZ/n\ZZ$-invariant part''.
\end{proof}

The higher-rank generalization of this result is straightforward:

\begin{prop}\label{geometry tensor Fock}
For any $n\in \ZZ_{>0}$ and $\ww=(w_0,\ldots,w_{n-1})\in \NN^{[n]}$, the following holds:

\noindent
(a) There is an isomorphism of $\U^{(n)}_{q_1,q_2,q_3}$-representations
  $\alpha\colon \bigotimes_{i=0}^{n-1} \bigotimes_{k=1}^{w_i} F^{i}(\chi_{i,k})\iso K(\ww)_{\mathrm{loc}}$.

\noindent
(b) There is an isomorphism of $\Y^{(n)}_{h_1,h_2,h_3}$-representations
  ${^a}\alpha\colon \bigotimes_{i=0}^{n-1} \bigotimes_{k=1}^{w_i} {{^a}F}^{i}(x_{i,k})\iso H(\ww)_{\mathrm{loc}}$.

\noindent
(c) Both isomorphisms $\alpha$ and ${^a}\alpha$ are given by diagonal matrices in the bases
$\{|\boldsymbol{\lambda}\rangle\}$ and $\{[\boldsymbol{\lambda}]\}$.

\noindent
(d) Parts (a,b) hold for an arbitrary reordering of the tensor products from the left-hand sides.
\end{prop}

There exists a well-known relation between the Nakajima quiver varieties
associated to the quivers $Q_m$ and $Q_{mn}$. Let
$\ww=\sum_{k=1}^r\ww^{(p_k)}$ and $\ww'=\sum_{k=1}^r\ww^{(p_k')}$ with
$p_k\in [mn]$ and $p'_k \in [m]$, where $p'_k:=p_k \bmod m$ (compare to Theorem~\ref{intertwiner 1}).
Then, there is an action of the group $\ZZ/mn\ZZ$ (which factors through its quotient
$(\ZZ/mn\ZZ)/(\ZZ/m\ZZ)\simeq \ZZ/n\ZZ$) on $\bigsqcup_{\vv'} \M^{Q_m}(\vv',\ww')$,
such that the variety of fixed points is isomorphic to $\bigsqcup_{\vv} \M^{Q_{mn}}(\vv,\ww)$.
Therefore, we have an  inclusion
  $\bigsqcup_{\vv} \M^{Q_{mn}}(\vv,\ww)\hookrightarrow \bigsqcup_{\vv'}\M^{Q_m}(\vv',\ww')$.
Let $\mathcal{I}_{m,n}\colon K(\ww')_{\mathrm{loc}} \to H(\ww)_{\mathrm{loc}}$
be a composition of an equivariant Chern character map and a pull-back in
localized equivariant cohomology. This map is diagonal in the fixed point bases,
hence, it is an isomorphism.

Our main result of this subsection reveals a geometric realization
of $I^{r;\pp,\vv}_{m,n;\omega_{mn}}$.

\begin{thm}\label{geometric intertwiner}
The following diagram is commutative:

\setlength{\unitlength}{1cm}
\begin{picture}(4,3.2)
  \put(4.0,2.5){$F^{\pp'}(\uu')$}
  \put(8.4,2.5){${^a}F^{\pp}(\vv)$}
  \put(4.0,0.5){$K(\ww')_{\mathrm{loc}}$}
  \put(8.4,0.5){$H(\ww)_{\mathrm{loc}}$}

  \put(5.4,2.6){\vector (1,0){2.8}}
  \put(4.6,2.3){\vector (0,-1){1.4}}
  \put(5.8,0.6){\vector (1,0){2.3}}
  \put(9.0,2.3){\vector (0,-1){1.4}}

  \put(6.4,2.75){$I^{r;\pp,\vv}_{m,n;\omega_{mn}}$}
  \put(4.3,1.5){$\alpha$}
  \put(8.5,1.5){${^a}\alpha$}
  \put(6.5,0.75){$\mathcal{I}_{m,n}$}
\end{picture}
\end{thm}

\begin{proof}[Proof of Theorem~\ref{geometric intertwiner}]
$\ $

This tedious verification is straightforward and is left to the interested reader.
\end{proof}


\section{Shuffle interpretation}\label{section proof via shuffle}

In this section, following~\cite{Neg0}--\cite{Neg2} we recall the shuffle realizations
of \emph{positive halves} $\U^{(n),>}_{q_1,q_2,q_3}$ and $\Y^{(n),>}_{h_1,h_2,h_3}$ and
provide a shuffle interpretation of $\Phi^{\omega_{mn}}_{m,n}$.
This implies the compatibility of $\Phi^{\omega_{mn}}_{m,n}$ with (T6), completing our
straightforward proof of Theorem~\ref{main 1} from Section~\ref{section partial proof}.


\subsection{Multiplicative shuffle algebras $S^{(n)}$}
\

Consider an $\NN^{[n]}$-graded $\CC$-vector space
  $\sS^{(n)}=\underset{\ol{k}\in \NN^{[n]}}\oplus\sS^{(n)}_{\overline{k}},$
where $\sS^{(n)}_{(k_0,\ldots,k_{n-1})}$ consists of $\prod S_{k_i}$-symmetric
rational functions in the variables $\{x_{i,r}\}_{i\in [n]}^{1\leq r\leq k_i}$.
Following~\cite{FT2}, we also fix an $n\times n$ matrix of rational functions
$(\omega_{i,j}(z,w))_{i,j\in [n]} \in \mathrm{Mat}_{n\times n}(\CC(z,w))$
by setting
\begin{equation}\label{omega matrix}
  \omega_{i,j}(z,w)=d^{-\delta_{j,i+1}\delta_{n>2}}
  \left(\frac{z-q_3^{-1}w}{z-w}\right)^{\delta_{j,i+1}}
  \left(\frac{z-q_2^{-1}w}{z-w}\right)^{\delta_{j,i}}
  \left(\frac{z-q_1^{-1}w}{z-w}\right)^{\delta_{j,i-1}}.
\end{equation}
Let us introduce the bilinear $\star$ product on $\sS^{(n)}$:
for $F\in \sS^{(n)}_{\overline{k}}, G\in \sS^{(n)}_{\overline{l}}$,
define $F\star G\in \sS^{(n)}_{\overline{k}+\overline{l}}$ by
\begin{equation}\label{star product on mult-shuffle}
\begin{split}
  & (F\star G)(x_{0,1},\ldots,x_{0,k_0+l_0};\ldots;x_{n-1,1},\ldots, x_{n-1,k_{n-1}+l_{n-1}}):=\\
  & \Sym \left(F\left(\{x_{i,r}\}_{i\in [n]}^{1\leq r\leq k_i}\right)G\left(\{x_{j,s}\}_{j\in [n]}^{k_{j}<s\leq k_{j}+l_{j}}\right)\cdot
  \prod_{i\in [n]}^{j\in [n]}\prod_{r\leq k_i}^{s>k_{j}}\omega_{i,j}(x_{i,r},x_{j,s})\right).
\end{split}
\end{equation}
Here and afterwards, given a function
$f\in \CC(\{x_{i,1},\ldots,x_{i,m_i}\}_{i\in [n]})$,
we define its symmetrization as follows:
 $\Sym(f):=\prod_{i\in [n]}\frac{1}{m_i!}\cdot
  \sum_{(\sigma_0,\ldots,\sigma_{n-1})\in S_{m_0}\times \ldots\times S_{m_{n-1}}}
  f(\{x_{i,\sigma_i(1)},\ldots,x_{i,\sigma_i(m_i)}\}_{i\in [n]}).$

\medskip
This endows $\sS^{(n)}$ with a structure of an associative unital algebra with
the unit $\textbf{1}\in \sS^{(n)}_{(0,\ldots,0)}$. We will be interested only in
a certain subspace of $\sS^{(n)}$, defined by the \emph{pole} and \emph{wheel conditions}:

\noindent
$\bullet$
We say that $F\in \sS^{(n)}_{\overline{k}}$ \emph{satisfies the pole conditions} if and only if
  $$F=\frac{f(x_{0,1},\ldots,x_{n-1,k_{n-1}})}
  {\prod_{i\in [n]}\prod_{r\leq k_i,s\leq k_{i+1}}^{(i,r)\ne (i+1,s)}(x_{i,r}-x_{i+1,s})},\
    \mathrm{where}\ f\in (\CC[x_{i,r}^{\pm 1}]_{i\in [n]}^{1\leq r\leq k_i})^{\prod S_{k_i}}.$$

\noindent
$\bullet$
We say that $F\in \sS^{(n)}_{\overline{k}}$ \emph{satisfies the wheel conditions} if and only if
  $$F(\{x_{i,r}\})=0\ \mathrm{once}\ x_{i,r_1}/x_{i+\epsilon,l}=qd^{\epsilon}\
    \mathrm{and}\ x_{i+\epsilon,l}/x_{i,r_2}=qd^{-\epsilon}\
    \mathrm{for\ some}\ \epsilon, i, r_1, r_2, l,$$
where $\epsilon\in \{\pm 1\}, i\in [n], 1\leq r_1,r_2\leq k_i, 1\leq l\leq k_{i+\epsilon}$
and we use the cyclic notation as before.

Let $S^{(n),>}_{\overline{k}}\subset \sS^{(n)}_{\overline{k}}$ be the subspace of
all elements $F$ satisfying the above two conditions. Set
  $S^{(n),>}:=\underset{\overline{k}\in \NN^{[n]}}\oplus S^{(n),>}_{\overline{k}}.$
Further
  $S^{(n),>}_{\overline{k}}=\oplus_{r\in \ZZ}S^{(n),>}_{\overline{k},r},\
   S^{(n),>}_{\overline{k},r}:=\{F\in S^{(n),>}_{\overline{k}}|\mathrm{tot.deg}(F)=r\}$.
It is straightforward to see that the subspace $S^{(n),>}\subset\sS^{(n)}$ is $\star$-closed.

\begin{defn}
The algebra $(S^{(n),>},\star)$ is called the multiplicative shuffle algebra
(of $\widehat{\ssl}_n$-type).
\end{defn}

Let $\U^{(n),>}$ be the subalgebra of $\U^{(n)}_{q_1,q_2,q_3}$ generated by
$\{e_{i,k}\}_{i\in [n]}^{k\in \ZZ}$. The former is known to be generated by
$\{e_{i,k}\}_{i\in [n]}^{k\in \ZZ}$ with the defining relations (T2,T6).
We equip $\U^{(n),>}$ with the $\NN^{[n]}\times \ZZ$--grading by assigning
$\deg(e_{i,k})=(1_i;k)$ for all $i\in [n],k\in \ZZ$, where $1_i\in \NN^{[n]}$
is the vector with the $i$-th coordinate $1$ and all other coordinates being zero.

The following beautiful result is due to A.~Negut:

\begin{thm}\cite{Neg0, Neg1}\label{Negut theorem 1}
The assignment $e_{i,k}\mapsto x_{i,1}^k$ for $i\in [n], k\in \ZZ$,
gives rise to an $\NN^{[n]}\times \ZZ$--graded $\CC$-algebra
isomorphism $\Theta_n\colon \U^{(n),>}\iso S^{(n),>}$.
\end{thm}


\subsection{Additive shuffle algebras $W^{(n)}$}
\

Consider an $\NN^{[n]}$-graded $\CC$-vector space
  $\sW^{(n)}=\underset{\ol{k}\in \NN^{[n]}}\oplus\sW^{(n)}_{\overline{k}},$
where $\sW^{(n)}_{(k_0,\ldots,k_{n-1})}$ consists of $\prod S_{k_i}$-symmetric
rational functions in the variables $\{x_{i,r}\}_{i\in [n]}^{1\leq r\leq k_i}$.
We also fix an $n\times n$ matrix of rational functions
  $(\varpi_{i,j}(z,w))_{i,j\in [n]} \in \mathrm{Mat}_{n\times n}(\CC(z,w))$
by setting
\begin{equation}\label{varpi matrix}
  \varpi_{i,j}(z,w)=
  \left(\frac{z-w+h_3}{z-w}\right)^{\delta_{j,i+1}}
  \left(\frac{z-w+h_2}{z-w}\right)^{\delta_{j,i}}
  \left(\frac{z-w+h_1}{z-w}\right)^{\delta_{j,i-1}}.
\end{equation}
We endow $\sW^{(n)}$ with a structure of an associative unital algebra via
the bilinear $\star$ product defined by the formula~(\ref{star product on mult-shuffle})
with $\omega_{i,j}(z,w)\rightsquigarrow \varpi_{i,j}(z,w)$ and with the unit
$\textbf{1}\in \sW^{(n)}_{(0,\ldots,0)}$.  We will be interested only in a certain subspace
of $\sW^{(n)}$, defined by the \emph{pole} and \emph{wheel conditions}:

\noindent
$\bullet$
We say that $F\in \sW^{(n)}_{\overline{k}}$ \emph{satisfies the pole conditions} if and only if
  $$F=\frac{f(x_{0,1},\ldots,x_{n-1,k_{n-1}})}{\prod_{i\in [n]}\prod_{r\leq k_i,s\leq k_{i+1}}^{(i,r)\ne (i+1,s)}(x_{i,r}-x_{i+1,s})},\
    \mathrm{where}\ f\in (\CC[x_{i,r}]_{i\in [n]}^{1\leq r\leq k_i})^{\prod S_{k_i}}.$$

\noindent
$\bullet$
We say that $F\in \sW^{(n)}_{\overline{k}}$ \emph{satisfies the wheel conditions} if and only if
  $$F(\{x_{i,r}\})=0\ \mathrm{once}\ x_{i,r_1}-x_{i+\epsilon,l}=h+\epsilon\beta\
    \mathrm{and}\ x_{i+\epsilon,l}-x_{i,r_2}=h-\epsilon\beta\
    \mathrm{for\ some}\ \epsilon, i, r_1, r_2, l,$$
where $\epsilon\in \{\pm 1\}, i\in [n], 1\leq r_1,r_2\leq k_i, 1\leq l\leq k_{i+\epsilon}$,
and $h=h_2/2, \beta=(h_1-h_3)/2$ as before.

Let $W^{(n),>}_{\overline{k}}\subset \sW^{(n)}_{\overline{k}}$ be the subspace of all elements
$F$ satisfying the above two conditions. Set
  $W^{(n),>}:=\underset{\overline{k}\in \NN^{[n]}}\oplus W^{(n),>}_{\overline{k}}.$
It is easy to see that the subspace $W^{(n),>}\subset\sW^{(n)}$ is $\star$-closed.

\begin{defn}
The algebra $(W^{(n),>},\star)$ is called the additive shuffle algebra
(of $\widehat{\ssl}_n$-type).
\end{defn}

Recall the subalgebra $\Y^{(n),>}$ of $\Y^{(n)}_{h_1,h_2,h_3}$ generated by
$\{x^+_{i,r}\}_{i\in [n]}^{r\in \NN}$. We equip $\Y^{(n),>}$ with the
$\NN^{[n]}$--grading by assigning $\deg(x^+_{i,r})=1_i$.
The following beautiful result is due to A.~Negut:

\begin{thm}\cite{Neg2}\label{Negut theorem 2}
The assignment $x^+_{i,r}\mapsto x_{i,1}^r$ for $i\in [n], r\in \NN$,
gives rise to an $\NN^{[n]}$--graded $\CC$-algebra isomorphism
$\Xi_n\colon \Y^{(n),>}\iso W^{(n),>}$.
\end{thm}

We extend $W^{(n),>}$ to a larger algebra $W^{(n),\geq}$ by adjoining commuting
elements $\{\xi_{i,r}\}_{i\in [n]}^{r\in \NN}$ so that $\Xi_n$ extends to the homonymous
isomorphism $\Xi_n\colon \Y^{(n),\geq}\iso W^{(n),\geq}$ with $\Xi_n(\xi_{i,r})=\xi_{i,r}$.


\subsection{Shuffle realization of $\Phi^{\omega_{mn}}_{m,n}$}
\

Fix $m,n\geq 1$ and an $mn$-th root of unity $\omega_{mn}=\exp(2\pi \kk\mathbf{i}/mn)$
with $\kk\in \ZZ$ and $\gcd(\kk,n)=1$. First, let us introduce the corresponding
formal versions of the above shuffle algebras:

\medskip
\noindent
$\bullet$
The algebra $S^{(m),\omega_{mn},>}_{\hbar_1,\hbar_2}$ is a
$\CC[[\hbar_1,\hbar_2]]$-counterpart of $S^{(m),>}$ with the following modifications
  $q_1\rightsquigarrow \q_1=\omega_{mn}\exp(\hbar_1/m),
   q_2\rightsquigarrow \q_2=\exp(\hbar_2/m),
   q_3\rightsquigarrow \q_3=\omega^{-1}_{mn}\exp(-(\hbar_1+\hbar_2)/m)$.

\medskip
\noindent
$\bullet$
The algebra $\wt{W}^{(mn),>}_{\hbar_1,\hbar_2}$ is a $\CC[[\hbar_1,\hbar_2]]$-counterpart
of $W^{(mn),>}$ with $h_s\rightsquigarrow \frac{\hbar_s}{mn}, s\in \{1,2,3\}$.
Unlike $W^{(mn),>}$, the algebra $\wt{W}^{(mn),>}_{\hbar_1,\hbar_2}$ is $\ZZ$-graded by
the total degree with $\deg(x_{i,r})=\deg(\hbar_s)=1$. Let
$W^{(mn),>}_{\hbar_1,\hbar_2}\subset \wt{W}^{(mn),>}_{\hbar_1,\hbar_2}$
be the subspace of all elements  of non-negative degree. Adjoining commuting elements
$\{\xi_{i,r}\}_{i\in [mn]}^{r\in \NN}$ with $\deg(\xi_{i,r})=r$, we obtain
an extended version $W^{(mn),\geq}_{\hbar_1,\hbar_2}$.

Due to Theorems~\ref{Negut theorem 1} and~\ref{Negut theorem 2},
we have $\CC[[\hbar_1,\hbar_2]]$-algebra isomorphisms
 $$\Theta_m\colon \U^{(m),\omega_{mn},>}_{\hbar_1,\hbar_2}\iso S^{(m),\omega_{mn},>}_{\hbar_1,\hbar_2},\
   \Xi_{mn}\colon \Y^{(mn),>}_{\hbar_1,\hbar_2}\iso W^{(mn),>}_{\hbar_1,\hbar_2},\
   \Xi_{mn}\colon \Y^{(mn),\geq}_{\hbar_1,\hbar_2}\iso W^{(mn),\geq}_{\hbar_1,\hbar_2}.$$
We note that the former isomorphism is $\NN^{[n]}\times \ZZ$-graded,
while the latter two are $\NN^{[n]}\times \NN$-graded.

Let $\widehat{W}^{(mn),\geq}_{\hbar_1,\hbar_2}$ be the completion of
$W^{(mn),\geq}_{\hbar_1,\hbar_2}$ with respect to the above $\NN$-grading.
In what follows, $\left(\omega^{(m)}_{i,j}(z,w)\right)_{i,j\in [m]}$ and
$\left(\varpi^{(mn)}_{i',j'}(z,w)\right)_{i',j'\in [mn]}$ denote the
matrices~(\ref{omega matrix}) and ~(\ref{varpi matrix}) corresponding to
the algebras $S^{(m),\omega_{mn},>}_{\hbar_1,\hbar_2}$ and $W^{(mn),>}_{\hbar_1,\hbar_2}$,
respectively. Finally, we will use shorthand notations $F(x_{i_1},\ldots,x_{i_k})$ and
$F(x_{i'_1},\ldots,x_{i'_k})$ for shuffle elements (skipping double indices).

The following is the key result of this section:

\begin{thm}\label{shuffle homom}
(a) The assignment
\begin{equation}\label{shuffle map}
\begin{split}
  & F\left(\{x_{i_a}\}_{a=1}^k\right)\mapsto
  \sum_{i'_1\equiv i_1,\ldots, i'_k\equiv i_k}^{i'_1,\ldots,i'_k\in [mn]}
  g_{i'_1}(x_{i'_1})\cdots g_{i'_k}(x_{i'_k})\cdot B\left(\{x_{i'_a}\}_{a=1}^k\right)\cdot
  F\left(\{\omega_{mn}^{-i'_a}e^{nx_{i'_a}}\}_{a=1}^k\right)\\
  & \mathrm{with}\ B(x_{i'_1},\ldots,x_{i'_k}):=
  \left[\prod_{1\leq a\ne b\leq k}\frac{\varpi^{(mn)}_{i'_a,i'_b}(x_{i'_a},x_{i'_b})}
  {\omega^{(m)}_{i_a,i_b}(\omega_{mn}^{-i'_a}e^{nx_{i'_a}},\omega_{mn}^{-i'_b}e^{nx_{i'_b}})}\right]^{1/2}
\end{split}
\end{equation}
gives rise to a $\CC[[\hbar_1,\hbar_2]]$-algebra homomorphism
  $\Gamma^{\omega_{mn}}_{m,n}\colon S^{(m),\omega_{mn},>}_{\hbar_1,\hbar_2}
   \to \widehat{W}^{(mn),\geq}_{\hbar_1,\hbar_2}$.

\noindent
(b) The following diagram is commutative:
  $$\begin{CD}
    \U^{(m),\omega_{mn},>}_{\hbar_1,\hbar_2} @>{\Phi^{\omega_{mn}}_{m,n}}>> \widehat{\Y}^{(mn),\geq}_{\hbar_1,\hbar_2}\\
    @V{\Theta_m}V{\wr}V   @V{\wr}V{\Xi_{mn}}V\\
    S^{(m),\omega_{mn},>}_{\hbar_1,\hbar_2} @>>{\Gamma^{\omega_{mn}}_{m,n}}> \widehat{W}^{(mn),\geq}_{\hbar_1,\hbar_2}
    \end{CD}$$
\end{thm}

\begin{proof}[Proof of Theorem~\ref{shuffle homom}]
\

(a) First, we note that in the above product
$g_{i'_1}(x_{i'_1})\cdots g_{i'_k}(x_{i'_k})$ our convention is to take
all the variables $\{x_{i'_a}\}$ to the right of all the Cartan terms.
For $F\in S^{(m),\omega_{mn},>}_{\hbar_1,\hbar_2}$, we denote its image
under the assignment~(\ref{shuffle map}) by $\Gamma^{\omega_{mn}}_{m,n}(F)$.
The verification of the wheel conditions for $\Gamma^{\omega_{mn}}_{m,n}(F)$
is straightforward (they follow from the wheel conditions for $F$). Likewise,
the verification of the pole conditions for $\Gamma^{\omega_{mn}}_{m,n}(F)$ is straightforward
and follows from the explicit formulas~(\ref{omega matrix},\ref{varpi matrix}).

Our key computation is based on the following result:

\begin{lem}\label{last lemma}
For any $i,j\in [m]$ and $i',j'\in [mn]$ such that $i'\equiv i, j'\equiv j$, we have
\begin{equation*}
  \lambda^+_{i'}(u)(g_{j'}(v))=g_{j'}(v)\cdot
  \left(\frac{\varpi^{(mn)}_{j',i'}(v,u)}{\varpi^{(mn)}_{i',j'}(u,v)}\cdot
  \frac{\omega^{(m)}_{i,j}(\omega_{mn}^{-i'}e^{nu},\omega_{mn}^{-j'}e^{nv})}
       {\omega^{(m)}_{j,i}(\omega_{mn}^{-j'}e^{nv},\omega_{mn}^{-i'}e^{nu})}\right)^{1/2}.
\end{equation*}
\end{lem}

\begin{proof}[Proof of Lemma~\ref{last lemma}]
\

According to our proof of Proposition~\ref{compatibility with T2} and
the formulas~(\ref{lambda action on gg},\ref{ratio of p},\ref{ratio of g}), we have
\begin{equation*}
\begin{split}
  & \lambda^+_{i'}(u)(g_{j'}(v))=g_{j'}(v)\cdot \q^{a^{(m)}_{i,j}/2}\times\\
  & \left(\frac{u-v-\frac{\hbar_1}{mn}}{u-v+\frac{\hbar_3}{mn}}\right)^{\frac{\delta_{j',i'+1}}{2}}
  \left(\frac{u-v-\frac{\hbar_2}{mn}}{u-v+\frac{\hbar_2}{mn}}\right)^{\frac{\delta_{j',i'}}{2}}
  \left(\frac{u-v-\frac{\hbar_3}{mn}}{u-v+\frac{\hbar_1}{mn}}\right)^{\frac{\delta_{j',i'-1}}{2}}\times\\
  & \left(\frac{\omega_{mn}^{-i'}e^{nu}-\q_3^{-1}\omega_{mn}^{-j'}e^{nv}}{\omega_{mn}^{-i'}e^{nu}-\q_1\omega_{mn}^{-j'}e^{nv}}\right)^{\frac{\delta_{j,i+1}}{2}}
  \left(\frac{\omega_{mn}^{-i'}e^{nu}-\q_2^{-1}\omega_{mn}^{-j'}e^{nv}}{\omega_{mn}^{-i'}e^{nu}-\q_2\omega_{mn}^{-j'}e^{nv}}\right)^{\frac{\delta_{j,i}}{2}}
  \left(\frac{\omega_{mn}^{-i'}e^{nu}-\q_1^{-1}\omega_{mn}^{-j'}e^{nv}}{\omega_{mn}^{-i'}e^{nu}-\q_3\omega_{mn}^{-j'}e^{nv}}\right)^{\frac{\delta_{j,i-1}}{2}}.
\end{split}
\end{equation*}
The result now follows by recalling the explicit formulas~(\ref{omega matrix})
and~(\ref{varpi matrix}).
\end{proof}

Using Lemma~\ref{last lemma}, we can finally verify that the assignment
$F\mapsto \Gamma^{\omega_{mn}}_{m,n}(F)$ is an algebra homomorphism. For
  $F(\{x_{i_a}\}_{a=1}^k), G(\{x_{i_b}\}_{b=k+1}^{k+l})\in S^{(m),\omega_{mn},>}_{\hbar_1,\hbar_2}$,
the following holds:
\begin{equation*}
\begin{split}
  & \Gamma^{\omega_{mn}}_{m,n}(F)\star \Gamma^{\omega_{mn}}_{m,n}(G)=\\
  & \left(\sum_{i'_1\equiv i_1,\ldots, i'_k\equiv i_k}^{i'_1,\ldots,i'_k\in [mn]}
        g_{i'_1}(x_{i'_1})\cdots g_{i'_k}(x_{i'_k})\cdot B(\{x_{i'_a}\}_{a=1}^k)\cdot F\left(\{\omega_{mn}^{-i'_a}e^{nx_{i'_a}}\}_{a=1}^k\right)\right)\star\\
  & \left(\sum_{i'_{k+1}\equiv i_{k+1},\ldots, i'_{k+l}\equiv i_{k+l}}^{i'_{k+1},\ldots,i'_{k+l}\in [mn]}
        g_{i'_{k+1}}(x_{i'_{k+1}})\cdots g_{i'_{k+l}}(x_{i'_{k+l}})\cdot B(\{x_{i'_b}\}_{b=k+1}^{k+l})\cdot G\left(\{\omega_{mn}^{-i'_b}e^{nx_{i'_b}}\}_{b=k+1}^{k+l}\right)\right)=\\
  & \sum_{i'_1\equiv i_1,\ldots,i'_{k+l}\equiv i_{k+l}}^{i'_1,\ldots,i'_{k+l}\in [mn]}
  \left(g_{i'_1}(x_{i'_1})\cdots g_{i'_k}(x_{i'_k})g_{i'_{k+1}}(x_{i'_{k+1}})\cdots g_{i'_{k+l}}(x_{i'_{k+l}})\times\right.\\
  & \left.\Sym\left(F\left(\{\omega_{mn}^{-i'_a}e^{nx_{i'_a}}\}_{a=1}^k\right) G\left(\{\omega_{mn}^{-i'_b}e^{nx_{i'_b}}\}_{b=k+1}^{k+l}\right)
  B(\{x_{i'_a}\}_{a=1}^k) B(\{x_{i'_b}\}_{b=k+1}^{k+l})\times\right.\right.\\
  & \left.\left. \prod_{1\leq a\leq k}^{k<b\leq k+l} \varpi^{(mn)}_{i'_a,i'_b}(x_{i'_a},x_{i'_b})\cdot
  \prod_{1\leq a\leq k}^{k<b\leq k+l}
  \left(\frac{\varpi^{(mn)}_{i'_b,i'_a}(x_{i'_b},x_{i'_a})}{\varpi^{(mn)}_{i'_a,i'_b}(x_{i'_a},x_{i'_b})}\cdot
  \frac{\omega^{(m)}_{i_a,i_b}(\omega_{mn}^{-i'_a}e^{nx_{i'_a}},\omega_{mn}^{-i'_b}e^{nx_{i'_b}})}
  {\omega^{(m)}_{i_b,i_a}(\omega_{mn}^{-i'_b}e^{nx_{i'_b}},\omega_{mn}^{-i'_a}e^{nx_{i'_a}})}\right)^{1/2}\right)\right)=\\
  & \sum_{i'_1\equiv i_1,\ldots,i'_{k+l}\equiv i_{k+l}}^{i'_1,\ldots,i'_{k+l}\in [mn]}
  \left(g_{i'_1}(x_{i'_1})\cdots g_{i'_{k+l}}(x_{i'_{k+l}})
  \Sym\left(F\left(\{\omega_{mn}^{-i'_a}e^{nx_{i'_a}}\}_{a=1}^k\right) G\left(\{\omega_{mn}^{-i'_b}e^{nx_{i'_b}}\}_{b=k+1}^{k+l}\right)\times\right.\right.\\
  & \left.\left.\prod_{1\leq a,b\leq k+l}^{a\ne b} \left(\frac{\varpi^{(mn)}_{i'_a,i'_b}(x_{i'_a},x_{i'_b})}
  {\omega^{(m)}_{i_a,i_b}(\omega_{mn}^{-i'_a}e^{nx_{i'_a}},\omega_{mn}^{-i'_b}e^{nx_{i'_b}})}\right)^{1/2}\cdot
  \prod_{1\leq a\leq k}^{k<b\leq k+l} \omega^{(m)}_{i_a,i_b}(\omega_{mn}^{-i'_a}e^{nx_{i'_a}},\omega_{mn}^{-i'_b}e^{nx_{i'_b}})\right)\right)=\\
  & \Gamma^{\omega_{mn}}_{m,n}(F\star G).
\end{split}
\end{equation*}
This completes our proof of Theorem~\ref{shuffle homom}(a).

(b) For any $i\in [m]$ and $k\in \ZZ$, we have
\begin{equation*}
\begin{split}
  & \Gamma^{\omega_{mn}}_{m,n}(\Theta_m(e_{i,k}))=\Gamma^{\omega_{mn}}_{m,n}(x_{i,1}^k)=
  \sum_{i'\equiv i}^{i'\in [mn]} g_{i'}(x_{i'})\omega_{mn}^{-ki'}e^{knx_{i'}}=\\
  & \Xi_{mn}\left(\sum_{i'\equiv i}^{i'\in [mn]} \omega_{mn}^{-ki'}e^{kn\sigma^+_{i'}}g_{i'}(\sigma^+_{i'})x^+_{i',0}\right)=
  \Xi_{mn}(\Phi^{\omega_{mn}}_{m,n}(e_{i,k})).
\end{split}
\end{equation*}
This implies part (b), since $\U^{(m),\omega_{mn},>}_{\hbar_1,\hbar_2}$
is generated by $\{e_{i,k}\}_{i\in [m]}^{k\in \ZZ}$.
\end{proof}


\subsection{Second proof of Theorem~\ref{main 1}}
\

As an immediate corollary of Theorem~\ref{shuffle homom}, we see that
the assignment $e_{i,k}\mapsto \Phi^{\omega_{mn}}_{m,n}(e_{i,k})$ defined
by~(\ref{Phi2}) is compatible with the relations~(T2,T6). Similar arguments
also prove the compatibility of the assignment
$f_{i,k}\mapsto \Phi^{\omega_{mn}}_{m,n}(f_{i,k})$ defined by~(\ref{Phi3})
with the relations~(T3,T6). Combining this with the verifications of
Section~\ref{section partial proof} completes our direct proof of Theorem~\ref{main 1}.

\begin{rem}
Informally speaking, the most complicated (Serre) defining relations~(T6,Y5)
are getting replaced by rather simple wheel conditions in the corresponding shuffle algebras.
\end{rem}



\begin{thebibliography} {XXX}

\bibitem[BBT]{BBT}
A.~Belavin, M.~Bershtein, G.~Tarnopolsky,
  {\em Bases in coset conformal field theory from AGT correspondence and Macdonald polynomials at the roots of unity},
J. High Energy Phys. {\bf 1303} (2013), 019.

\bibitem[BKLY]{BKLY}
C.~Boyallian, V.~Kac, J.~Liberati, C.~Yan,
  {\em Quasifinite highest weight modules over the Lie algebra of matrix differential operators on the circle},
J. Math. Phys. {\bf 39} (1998), no. 5, 2910--2928.

\bibitem[D]{D}
V.~Drinfeld,
  {\em A new realization of Yangians and quantized affine algebras},
preprint FTINT 30--86 (1986).

\bibitem[DI]{DI}
J.~Ding, K.~Iohara,
  {\em Generalization of Drinfeld quantum affine algebras},
Lett. Math. Phys. {\bf 41} (1997), no.~2, 181--193.

\bibitem[FJMM1]{FJMM1}
B.~Feigin, M.~Jimbo, T.~Miwa, E.~Mukhin,
  {\em Representations of quantum toroidal $\gl_n$},
J. Algebra {\bf 380} (2013), 78--108.

\bibitem[FJMM2]{FJMM2}
B.~Feigin, M.~Jimbo, T.~Miwa, E.~Mukhin,
  {\em Branching rules for quantum toroidal $\gl_n$},
Adv. Math. {\bf 300} (2016), 229--274.

\bibitem[FT1]{FT}
B.~Feigin, A.~Tsymbaliuk,
  {\em Equivariant $K$-theory of Hilbert schemes via shuffle algebra},
Kyoto J. Math. {\bf 51} (2011), no.~4, 831--854.

\bibitem[FT2]{FT2}
B.~Feigin, A.~Tsymbaliuk,
  {\em Bethe subalgebras of $U_q(\widehat{\gl}_n)$ via shuffle algebras},
Sel. Math. New Ser. {\bf 22} (2016), no. 2, 979--1011.

\bibitem[G]{G}
N.~Guay,
  {\em Affine Yangians and deformed double current algebras in type A},
Adv. Math. {\bf 211} (2007), no.~2, 436--484.

\bibitem[GKV]{GKV}
V.~Ginzburg, M.~Kapranov, E.~Vasserot,
  {\em Langlands reciprocity for algebraic surfaces},
Math. Res. Lett. {\bf 2} (1995), no.~2, 147--160.

\bibitem[GTL]{GTL}
S.~Gautam, V.~Toledano Laredo,
  {\em Yangians and quantum loop algebras},
Sel. Math. New Ser. {\bf 19} (2013), no.~2, 271--336.

\bibitem[K]{K}
R.~Kodera,
  {\em Affine Yangian action on the Fock space},
preprint, arXiv:1506.01246.

\bibitem[KT]{KT}
S.~Khoroshkin, V.~Tolstoy,
  {\em Yangian double},
Lett. Math. Phys. {\bf 36} (1996), no. 4, 373--402.

\bibitem[M]{M}
K.~Miki,
  {\em A $(q,\gamma)$ analog of the $W_{1+\infty}$ algebra},
J. Math. Phys. {\bf 48} (2007), no.~12, 123520.

\bibitem[MO]{MO}
D.~Maulik, A.~Okounkov,
  {\em Quantum groups and quantum cohomology},
preprint, arXiv:1211.1287.

\bibitem[Nak]{Na}
H.~Nakajima,
  {\em Quiver varieties and Kac-Moody algebras},
Duke Math. J. {\bf 91} (1998), no. 3, 515--560.

\bibitem[Neg1]{Neg0}
A.~Negut,
  {\em The shuffle algebra revisited},
Int. Math. Res. Not. IMRN (2014), no. 22, 6242--6275.

\bibitem[Neg2]{Neg1}
A.~Negut,
  {\em Quantum toroidal and shuffle algebras, R-matrices and a conjecture of Kuznetsov},
preprint, arXiv:1302.6202.

\bibitem[Neg3]{Neg2}
A.~Negut,
  {\em personal communication},
(2015).

\bibitem[SV1]{SV1}
O.~Schiffmann, E.~Vasserot,
  {\em The elliptic Hall algebra and the $K$-theory of the Hilbert scheme of $\mathbb{A}^2$},
Duke Math. J. {\bf 162} (2013), no.~2, 279--366.

\bibitem[SV2]{SV2}
O.~Schiffmann, E.~Vasserot,
  {\em Cherednik algebras, W-algebras and the equivariant cohomology of the moduli space of instantons on $\mathbb{A}^2$},
Publ. Math. Inst. Hautes \'{E}tudes Sci. {\bf 118} (2013), 213--342.

\bibitem[T1]{T}
A.~Tsymbaliuk,
  {\em The affine Yangian of $\mathfrak{gl}_1$ revisited},
Adv. Math. {\bf 304} (2017), 583--645.

\bibitem[T2]{T2}
A.~Tsymbaliuk,
  {\em Classical limits of quantum toroidal and affine Yangian algebras},
J. Pure Appl. Algebra {\bf 221} (2017), no. 10, 2633--2646.

\bibitem[V]{V}
M.~Varagnolo,
  {\em Quiver varieties and Yangians},
Lett. Math. Phys. {\bf 53} (2000), no. 4, 273--283.

\end{thebibliography}
\end{document}